\documentclass{amsart}

\usepackage[utf8]{inputenc}
\usepackage{amssymb}
\usepackage{amsrefs}
\usepackage{mathpazo}
\usepackage{graphicx}
\usepackage[linktocpage=true,
	colorlinks=true,
	urlcolor=blue!20!black,
	citecolor=green!50!black]{hyperref}
\usepackage{color}
\usepackage{todonotes}
\usepackage{subcaption}
\usepackage{chngcntr}
\usepackage{apptools}
	\AtAppendix{\counterwithin{theorem}{section}}
	\AtAppendix{\counterwithin{remark}{section}}
	\AtAppendix{\counterwithin{proposition}{section}}
\usepackage{thmtools}
\usepackage{thm-restate}
\usepackage{enumerate}
\usepackage{tikz,pgf}
	\pgfdeclarelayer{background}
	\pgfsetlayers{background,main}
\newtheorem{theorem}{Theorem}[part]
\newtheorem{proposition}[theorem]{Proposition}
\newtheorem{lemma}[theorem]{Lemma}
\newtheorem{corollary}[theorem]{Corollary}
\newtheorem{openproblem}[theorem]{Open Problem}
\theoremstyle{remark}
\newtheorem{remark}[theorem]{Remark}
\theoremstyle{definition}
\newtheorem{definition}[theorem]{Definition}
\newtheorem{construction}[theorem]{Construction}
\newtheorem{example}[theorem]{Example}
\newcommand{\defs}{\stackrel{\mathrm{def}}{=}}
\newcommand{\tdef}[1]{{\color{green!50!black}\emph{#1}}}
\newcommand{\ie}{i.e.\;}
\newcommand{\Symmetric}{\mathfrak{S}}
\newcommand{\inv}{\mathrm{inv}}
\newcommand{\Tamari}{\mathcal{T}}
\newcommand{\NC}{N\!C}

\newcommand{\minbracket}{\mathbf{b}^{\mathrm{min}}}
\newcommand{\bracket}{\mathbf{b}}
\newcommand{\bracketred}{\underline{\mathbf{b}}}
\newcommand{\horiz}{\mathrm{horiz}}
\newcommand{\Dyck}{\mathcal{D}}

\newcommand{\Lattice}{\mathcal{L}}
\newcommand{\Poset}{\mathcal{P}}
\newcommand{\least}{\hat{0}}
\newcommand{\grtst}{\hat{1}}
\newcommand{\JI}{\mathcal{J}}
\newcommand{\MI}{\mathcal{M}}
\newcommand{\Galois}{\mathrm{Gal}}
\newcommand{\MO}{\mathrm{MO}}
\newcommand{\LAC}{L\!A\!C}
\newcommand{\LAClabel}{\widetilde{L\!A\!C}}
\newcommand{\Trees}{\mathbb{T}}

\newcommand{\vertical}{\operatorname{Vert}} 
\newcommand{\diagonal}{\operatorname{Diag}} 
\newcommand{\dinv}{\mathrm{dinv}}
\newcommand{\area}{\mathrm{area}}
\newcommand{\bounce}{\mathrm{bounce}}
\newcommand{\steep}{\mathrm{steep}}
\newcommand{\SP}{\mathcal{SP}} 
\newcommand{\BP}{\mathcal{BP}} 
\newcommand{\SPlabel}{\vertical(\widetilde{\mathcal{SP}}_n)} 
\newcommand{\BPlabel}{\diagonal(\widetilde{\mathcal{BP}}_n)} 
\newcommand{\bij}{\Xi}							
\newcommand{\invbij}{\Lambda}					
\newcommand{\lactoperm}{\bij_{\mathrm{perm}}}
\newcommand{\lactodyck}{\bij_{\mathrm{dyck}}}
\newcommand{\lactonc}{\bij_{\mathrm{nc}}}
\newcommand{\lactonn}{\bij_{\mathrm{nn}}}
\newcommand{\permtolac}{\invbij_{\mathrm{perm}}}
\newcommand{\dycktolac}{\invbij_{\mathrm{dyck}}}
\newcommand{\nctolac}{\invbij_{\mathrm{nc}}}
\newcommand{\nntolac}{\invbij_{\mathrm{nn}}}
\newcommand{\lactosteep}{\bij_{\mathrm{steep}}}
\newcommand{\steeptolac}{\invbij_{\mathrm{steep}}}
\newcommand{\lactobounce}{\bij_{\mathrm{bounce}}}
\newcommand{\bouncetolac}{\invbij_{\mathrm{bounce}}}
\newcommand{\steeppairtobouncepair}{\Gamma}
\newcommand{\lactosteeplabel}{\widetilde\bij_{\mathrm{steep}}}
\newcommand{\steeptolaclabel}{\widetilde\invbij_{\mathrm{steep}}}
\newcommand{\lactobouncelabel}{\widetilde\bij_{\mathrm{bounce}}}
\newcommand{\bouncetolaclabel}{\widetilde\invbij_{\mathrm{bounce}}}
\newcommand{\steeppairtobouncepairlabel}{\widetilde\Gamma}
\newcommand{\dycktosteeppair}{\phi}

\makeatletter
\let\orgdescriptionlabel\descriptionlabel
\renewcommand*{\descriptionlabel}[1]{%
  \let\orglabel\label
  \let\label\@gobble
  \phantomsection
  \protected@edef\@currentlabel{#1\unskip}%
  \let\label\orglabel
  \orgdescriptionlabel{(#1)}%
}
\makeatother

\makeatletter
\def\part{\@startsection{part}{1}%
\z@{.7\linespacing\@plus\linespacing}{.8\linespacing}%
{\LARGE\sffamily\centering}}
\@addtoreset{section}{part}
\makeatother

\makeatletter
\def\l@section{\@tocline{1}{5pt}{0pc}{}{}}
\makeatother
\let\oldtocpart=\tocpart
\renewcommand{\tocpart}[2]{\bf\large\oldtocpart{#1}{#2}}
\let\oldtocsection=\tocsection
\renewcommand{\tocsection}[2]{\bf\oldtocsection{#1}{#2}}

\title{The Steep-Bounce Zeta Map in Parabolic Cataland}

\author{Cesar Ceballos}
\address{Faculty of Mathematics, University of Vienna, 1090 Vienna, Austria.} 
\email{cesar.ceballos@univie.ac.at}
\thanks{Cesar Ceballos was supported by the Austrian Science Foundation FWF, grant F 5008-N15, in the framework of the Special Research Program Algorithmic and Enumerative Combinatorics.}

\author{Wenjie Fang}
\address{Universit\'e Paris-Est Marne-la-Vallée, LIGM (UMR 8094), CNRS, ENPC, ESIEE Paris, France.} 
\email{wenjie.fang@u-pem.fr}
\thanks{Wenjie Fang was supported by the Austrian Science Foundation FWF, grant P27290 and I2309.}

\author{Henri M{\"u}hle}
\address{Technische Universit{\"a}t Dresden, Institut f{\"u}r Algebra, Zellescher Weg 12--14, 01069 Dresden, Germany.}
\email{henri.muehle@tu-dresden.de}

\subjclass[2010]{}

\begin{document}

\begin{abstract}
	As a classical object, the Tamari lattice has many generalizations, including $\nu$-Tamari lattices and parabolic Tamari lattices.  In this article, we unify these generalizations in a bijective fashion. We first prove that parabolic Tamari lattices are isomorphic to $\nu$-Tamari lattices for bounce paths $\nu$. We then introduce a new combinatorial object called ``left-aligned colorable tree'', and show that it provides a bijective bridge between various parabolic Catalan objects and certain nested pairs of Dyck paths.  As a consequence, we prove the Steep-Bounce Conjecture using a generalization of the famous zeta map in $q,t$-Catalan combinatorics.
A generalization of the zeta map on parking functions, which arises in the theory of diagonal harmonics, is also obtained as a labeled version of our bijection.
\end{abstract}

\maketitle

\tableofcontents

\definecolor{c1}{HTML}{FF9999}
\definecolor{c2}{HTML}{FFD27F}
\definecolor{c3}{HTML}{99C199}
\definecolor{c4}{HTML}{9999FF}
\definecolor{c5}{HTML}{99D1D1}
\definecolor{c6}{HTML}{D9A6F9}


\newcommand{\ncA}[1]{
	\includegraphics[scale=#1,page=1]{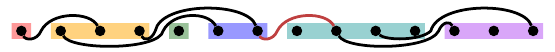}
}

\newcommand{\permA}[1]{
	\includegraphics[scale=#1,page=2]{fig/figures-std.pdf}
}

\newcommand{\pathA}[2]{
	\ifthenelse{\equal{#2}{1}}
	{
		\includegraphics[scale=#1,page=4]{fig/figures-std.pdf}
	}{
		\includegraphics[scale=#1,page=3]{fig/figures-std.pdf}
	}
}

\newcommand{\ncB}[1]{
	\includegraphics[scale=#1,page=5]{fig/figures-std.pdf}
}

\newcommand{\permB}[1]{
	\includegraphics[scale=#1,page=6]{fig/figures-std.pdf}
}

\newcommand{\pathB}[2]{
	\ifthenelse{\equal{#2}{1}}
	{
		\includegraphics[scale=#1,page=8]{fig/figures-std.pdf}
	}{
		\includegraphics[scale=#1,page=7]{fig/figures-std.pdf}
	}
}

\newcommand{\pathBReversed}[1]{
	\includegraphics[scale=#1,page=9]{fig/figures-std.pdf}
}

\newcommand{\paraTam}[1]{
	\includegraphics[width=#1,page=10]{fig/figures-std.pdf}
}

\newcommand{\paraTamDual}[1]{
	\includegraphics[width=#1,page=11]{fig/figures-std.pdf}
}

\newcommand{\nuTam}[1]{
	\includegraphics[width=#1,page=12]{fig/figures-std.pdf}
}

\newcommand{\nuTamvec}[1]{
    \includegraphics[width=#1, page=13]{fig/figures-std.pdf}
}

\newcommand{\galoisoncells}[1]{
	\includegraphics[scale=#1, page=14]{fig/figures-std.pdf}
}

\newcommand{\pathBReversednodot}[1]{
	\includegraphics[scale=#1,page=15]{fig/figures-std.pdf}
}


\newcommand{\lactreeB}[1]{
	\includegraphics[scale=#1, page=4]{fig/ipe-fig.pdf}
}

\newcommand{\lactreeBdot}[1]{
	\includegraphics[scale=#1, page=5]{fig/ipe-fig.pdf}
}

\newcommand{\treepostfix}[1]{
	\includegraphics[scale=#1, page=6]{fig/ipe-fig.pdf}
}

\newcommand{\markeddyck}[1]{
	\includegraphics[scale=#1, page=7]{fig/ipe-fig.pdf}
}

\newcommand{\quadrantwalk}[1]{
	\includegraphics[scale=#1, page=8]{fig/ipe-fig.pdf}
}

\newcommand{\steeppair}[1]{
	\includegraphics[scale=#1, page=9]{fig/ipe-fig.pdf}
}

\newcommand{\lactreeBnc}[1]{
	\includegraphics[scale=#1, page=10]{fig/ipe-fig.pdf}
}

\newcommand{\bouncepair}[1]{
	\includegraphics[scale=#1, page=20]{fig/ipe-fig.pdf}
}

\section*{Introduction}
	\label{sec:introduction}
\renewcommand{\thetheorem}{\Roman{theorem}}

The Tamari lattice can be realized as a partial order on various Catalan objects.  It has been studied widely from various perspectives, leading to numerous generalizations. Two of its recent variants are the parabolic Tamari lattices of Williams and the third author~\cite{muehle18tamari} and the $\nu$-Tamari lattices of Pr{\'e}ville-Ratelle and Viennot~\cite{preville17enumeration}. 
The parabolic Tamari lattices are defined as certain lattice quotients of the weak order in parabolic quotients of the symmetric group, which generalizes a construction by Bj{\"o}rner and Wachs~\cite{bjorner97shellable}.
The $\nu$-Tamari lattices are, in contrast, partial orders defined by manipulating lattice paths. They generalize the $m$-Tamari lattices which were initially motivated by connections 
to trivariate diagonal harmonics~\cite{BergeronPrevilleRatelle}, and now have remarkable applications to the theory of multivariate diagonal harmonics~\cite{ceballos_hopf_2018} and bijective links to other objects in combinatorics \cite{fpr17nonsep}. 

\medskip

In this article, we reunite the objects in these two variants of the Tamari lattice. 
More precisely, we introduce a new family of parabolic Catalan objects which we call \emph{left-aligned colorable trees} (or simply \emph{LAC trees}). These trees connect the parabolic generalizations of 231-avoiding permutations, noncrossing set partitions, and Dyck paths introduced in \cite{muehle18tamari}, all of which can be recovered easily from the tree. We have the following result, with precise definitions delayed to subsequent sections.

\begin{restatable}{theorem}{bijections}\label{thm:main_bijections}
	For every integer composition $\alpha$, there is an explicit bijection from the set~$\Trees_{\alpha}$ of $\alpha$-trees to each of the following: the set $\Symmetric_{\alpha}(231)$ of $(\alpha,231)$-avoiding permutations, the set $\NC_{\alpha}$ of noncrossing $\alpha$-partitions, and the set $\Dyck_{\alpha}$ of $\alpha$-Dyck paths.
\end{restatable}

\begin{figure}[!bhtp]
	\centering
	\resizebox{\textwidth}{!}{
	\begin{tikzpicture}\small
		\def\dx{4.5};
		\def\dy{3};
		\draw(.75*\dx,3*\dy) node{Parabolic Dyck path};
		\draw(.75*\dx,2.35*\dy) node(nn)[outer sep=-1.2cm]{\pathBReversednodot{0.85}};
		\draw(1.85*\dx,2.8*\dy) node(perm){\permB{1}};
		\draw(1.85*\dx,3*\dy) node{Parabolic 231-avoiding permutation};
		\draw(3.25*\dx,2.8*\dy) node(nc){\ncB{1}};
		\draw(3.25*\dx,3*\dy) node{Parabolic noncrossing partition};
		\draw(1.75*\dx,2*\dy) node(tree){\lactreeB{0.7}};
		\draw(2.3*\dx,1.9*\dy) node{LAC tree};
		\draw(.75*\dx,.85*\dy) node(bounce){\bouncepair{0.7}};
		\draw(.9*\dx,.62*\dy) node{Bounce pair};
		\draw(1.85*\dx,.85*\dy) node(walk){\markeddyck{0.7}};
		\draw(2*\dx,.62*\dy) node{Level-marked};
		\draw(1.94*\dx,.47*\dy) node{Dyck path};
		\draw(3*\dx,.85*\dy) node(steep){\steeppair{0.7}};
		\draw(3.15*\dx,.62*\dy) node{Steep pair};
		\draw[-latex](tree) -- (perm) node[right,midway] {\scriptsize $\lactoperm$};
		\draw[-latex](tree) -- (nc) node[above,midway] {\scriptsize $\lactonc$};
		\draw[-latex](tree) -- (nn) node[above,midway] {\scriptsize $\lactonn$};
		\draw[-latex](tree) -- (walk) node[right,midway] {\scriptsize $\lactodyck$};
		\draw[-latex](tree) -- (steep) node[above,midway] {\scriptsize $\lactosteep$};
		\draw[-latex](tree) -- (bounce) node[left,midway] {\scriptsize $\lactobounce$};
		\begin{pgfonlayer}{background}
			\draw[line width=2pt,opacity=.5,gray](.25*\dx,1.65*\dy) -- (1.25*\dx,1.65*\dy);
			\draw[line width=2pt,opacity=.5,gray](1.87*\dx,1.65*\dy) -- (3.95*\dx,1.65*\dy);
			\draw(3.3*\dx,2.2*\dy) node{$\alpha = (1,3,1,2,4,3)$};
			\draw(3.8*\dx,1.75*\dy) node{Part~\ref{part:tamari}};
			\draw(3.8*\dx,1.85*\dy) node{Part~\ref{part:parabolic_cataland}};
			\draw(3.8*\dx,1.55*\dy) node{Part~\ref{part:zeta}};
		\end{pgfonlayer}
	\end{tikzpicture}
	}
	\caption{An overview of the bijections presented in this article.}
	\label{fig:big_picture}
\end{figure}
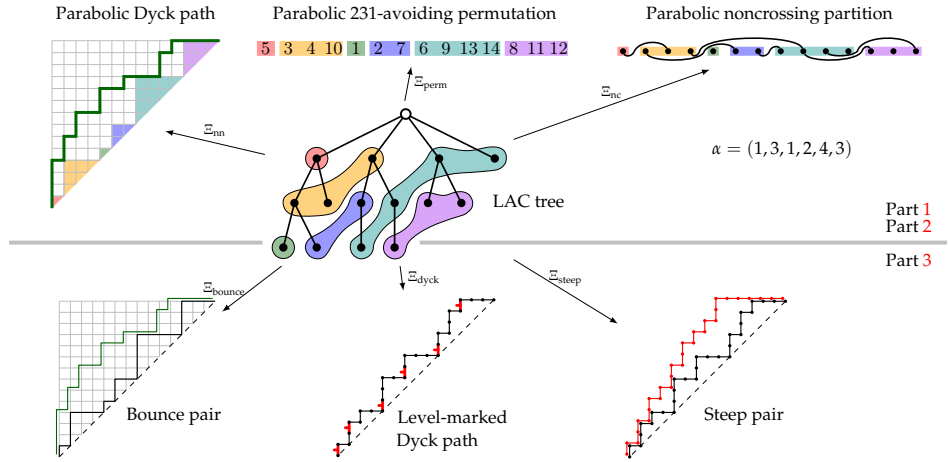

We continue with a structural study of the two variants of the Tamari lattice.  
With each composition $\alpha$ one can associate a bounce path $\nu_\alpha$, 
and we show that the corresponding $\nu_\alpha$-Tamari lattice $\Tamari_{\nu_{\alpha}}$ and the
parabolic Tamari lattice $\Tamari_{\alpha}$ are isomorphic. 
This result is particularly interesting because these two lattices provide completely different perspectives.
By definition of $\Tamari_{\nu_{\alpha}}$, it is simple  to check when two elements form a cover relation, but in general it is not easy to check whether they are comparable.  In $\Tamari_{\alpha}$, the situation is quite the opposite: by definition it is easy to check when two elements are comparable, but to determine whether they form a cover relation is a nontrivial task.
 
We exploit the fact that $\Tamari_{\alpha}$ and $\Tamari_{\nu_{\alpha}}$ are extremal lattices, and can therefore be uniquely represented by certain directed graphs.  We show that the bijection between parabolic $231$-avoiding permutations and parabolic Dyck paths from \cite{muehle18tamari} extends to an isomorphism of these graphs, and we may therefore conclude the following result.

\begin{restatable}{theorem}{tamari}\label{thm:parabolic_nu_tamari_isomorphism}
	For every integer composition $\alpha$, the parabolic Tamari lattice $\Tamari_{\alpha}$ is isomorphic to the $\nu_\alpha$-Tamari lattice $\Tamari_{\nu_{\alpha}}$.
\end{restatable}

In Part~\ref{part:zeta}, we use LAC trees to prove the Steep-Bounce Conjecture of Bergeron, Ceballos, and Pilaud~\cite[Conjecture~2.2.8]{ceballos_hopf_2018}, which connects the graded dimensions of a certain Hopf algebra of pipe dreams to the enumeration of certain family of lattice walks in the positive quarter plane. Our proof is based on a bijection $\steeppairtobouncepair$ between two families of nested Dyck paths via LAC trees. 

\begin{restatable}[The Steep-Bounce Theorem]{theorem}{steepbounce}\label{thm:steep_bounce_theorem}
	For $n>0$ and every $r\in[n]$, the map $\steeppairtobouncepair$ is a bijection from 
	\begin{itemize}
		\item the set of nested pairs $(\mu_{1},\mu_{2})$ of Dyck paths with $2n$ steps, where $\mu_{2}$ is a steep path ending with exactly $r$ east-steps with $y$-coordinate equal to $n$, to
		\item the set of nested pairs $(\mu'_{1},\mu'_{2})$ of Dyck paths with $2n$ steps, where $\mu'_{1}$ is a bounce path that touches the main diagonal $r+1$ times.
	\end{itemize}
\end{restatable}

Interestingly, we show
that the bijection $\steeppairtobouncepair$ generalizes the famous zeta map in $q,t$-Catalan combinatorics. Given a Dyck path $\mu$, we use two different greedy algorithms to define its associated \emph{steep path} $\mu_{\steep}$ and its \emph{bounce path} $\mu_{\bounce}$. Details are delayed to related sections. With these notions, we have the following result.

\begin{restatable}{theorem}{steepbouncezeta}\label{thm:steep_bounce_zeta_map}
	For $n>0$, the map $\steeppairtobouncepair$ restricts to a bijection from
	\begin{itemize}
		\item the set of pairs $(\mu,\mu_{\steep})$, where $\mu$ is a Dyck path with $2n$ steps, to
		\item the set of pairs $(\nu_{\bounce},\nu)$, where $\nu$ is a Dyck path with $2n$ steps.
	\end{itemize}
	Moreover, if $(\nu_{\bounce},\nu)=\steeppairtobouncepair(\mu,\mu_{\steep})$, then $\nu=\zeta(\mu)$, where $\zeta$ is the zeta map from $q,t$-Catalan combinatorics.
\end{restatable}

The study of $q,t$-Catalan combinatorics originated over twenty years ago in connection to Macdonald polynomials and Garsia--Haiman's theory of diagonal harmonics~\cites{GarsiaHaiman-remarkableCatalanSequence,Haiman-vanishingTheorems,haglund_conjectured_2003,garsia_proof_2002,Haglund-qt-catalan}.
The $q,t$-Catalan polynomial can be obtained as the bi-graded Hilbert series of the alternating component of certain module of diagonal harmonics, whose dimension is equal to the Catalan number~\cite{Haiman-vanishingTheorems}.
The zeta map is a bijection on Dyck paths that explains the equivalence of two combinatorial interpretations of these polynomials~\cite{Haglund-qt-catalan}, which were discovered almost simultaneously by Haglund and Haiman in terms of two pairs of statistics on Dyck paths.

The zeta map was further generalized by Haglund and Loehr~\cite{haglund_conjectured_2005}, who introduced an extension from the set of parking functions to  certain diagonally labeled Dyck paths. As in the $q,t$-Catalan case, their extension explains the equivalence of two combinatorial interpretations of the Hilbert series of the full diagonal harmonics module~\cite[Chapter~5]{Haglund-qt-catalan}~\cite{carlsson_proof_2018}, whose dimension is equal to the number of parking functions~\cite{Haiman-vanishingTheorems}. 
In Corollary~\ref{cor_zeta_parkingfunctions}, we show that this generalized zeta map can be obtained as a labeled version of our bijection $\steeppairtobouncepair$.

The organization of this article is as follows. In the first part, we introduce several families of parabolic Catalan objects, including left-aligned colorable trees, and then describe bijections between these families. In the second part, we dive into the order structure of parabolic Tamari lattices to explain how they are isomorphic to certain $\nu$-Tamari lattices. In the third part, we demonstrate how our bijection $\steeppairtobouncepair$ settles the Steep-Bounce Conjecture while generalizing the zeta map in $q,t$-Catalan combinatorics,
and present a labeled version which gives rise to the extension of the zeta map from parking functions to diagonally labeled Dyck paths.

\renewcommand{\thetheorem}{\arabic{part}.\arabic{theorem}}

\part{Parabolic Cataland}\label{part:parabolic_cataland}

In the first part of this article, we present an overview of the currently explored parts of Parabolic Cataland.  The name ``Cataland'' was coined in \cite{williams13cataland}, and describes the interaction of the various tribes of Catalan families associated with finite Coxeter groups\footnote{By a ``Catalan family'' we mean a family of combinatorial objects enumerated by some Coxeter-Catalan numbers.}. 

Recently, the frontiers of Cataland were pushed back in two different directions.  Firstly, in \cite{stump18cataland}, a positive integral parameter $m$ was introduced, which yields Fu{\ss}-Catalan families associated with finite Coxeter groups.  These families live in ``Fu{\ss}-Cataland''.  Secondly, in \cite{muehle18tamari} (building on \cite{williams13cataland}), the various Catalan families were generalized to parabolic quotients of finite Coxeter groups.  These generalizations constitute ``Parabolic Cataland'', of which only the type $A$ case is reasonably well understood.  In this section, we explore these parabolic Catalan families and their web of connections.

\section{Members of Parabolic Cataland}
	\label{sec:members_parabolic_cataland}
We now introduce the main protagonists of this story, but we first fix some notation.

For a natural number $n>0$, let $[n]\defs\{1,2,\ldots,n\}$, and let $\alpha=(\alpha_{1},\alpha_{2},\ldots,\alpha_{r})$ be a composition of $n$.  Moreover, let $\overline{\alpha}\defs(\alpha_{r},\alpha_{r-1},\ldots,\alpha_{1})$, and $\lvert\alpha\rvert=\alpha_{1}+\alpha_{2}+\cdots+\alpha_{r}$.

For $i\in\{0,1,\ldots,r\}$, we define $s_{i}\defs\alpha_{1}+\alpha_{2}+\cdots+\alpha_{i}$ and $t_{i}\defs n-s_{r-i}$, where we consider $s_{0}=0$.  For $i\in[r]$ the set $\{s_{i-1}+1,s_{i-1}+2,\ldots,s_{i}\}$ is the $i$-th \tdef{$\alpha$-region}.  

\subsection{Parabolic $231$-Avoiding Permutations}
	\label{sec:parabolic_231_avoidings}
Let $\Symmetric_{n}$ be the symmetric group of degree $n$.  The \tdef{parabolic quotient} of $\Symmetric_n$ with respect to $\alpha$ is defined by
\begin{displaymath}
	\Symmetric_{\alpha} \defs \bigl\{w\in\Symmetric_{\lvert\alpha\rvert}\mid w(k)<w(k+1)\;\text{for}\;k\notin\{s_{1},s_{2},\ldots,s_{r-1}\}\bigr\}.
\end{displaymath}

\begin{definition}[\cite{muehle18tamari}*{Definition~3.2}]\label{def:parabolic_231_avoiding_permutation}
	In a permutation $w\in\Symmetric_{\alpha}$, an \tdef{$(\alpha,231)$-pattern} is a triple of indices $i<j<k$ each in different $\alpha$-regions such that $w(k)<w(i)<w(j)$ and $w(i)=w(k)+1$.  A permutation in $\Symmetric_{\alpha}$ without $(\alpha,231)$-patterns is \tdef{$(\alpha,231)$-avoiding}.  
\end{definition}

A pair of indices $(i,k)$ such that $w(i)=w(k)+1$ is a \tdef{descent} of $w$.  We denote the set of all $(\alpha,231)$-avoiding permutations of $\Symmetric_{\alpha}$ by $\Symmetric_{\alpha}(231)$.  An example is shown in the top-left corner of Figure~\ref{fig:big_picture}. In this article, a permutation $w$ is always presented in its one-line notation $w(1) w(2) \cdots w(n)$.

If $\alpha=(1,1,\ldots,1)$, then we obtain precisely the stack-sortable permutations introduced in \cite{knuth73art1}*{Exercise~2.2.1.4}.  

\subsection{Parabolic Noncrossing Partitions}
	\label{sec:parabolic_noncrossings}
For $\lvert\alpha\rvert=n$, an \tdef{$\alpha$-partition} is a set partition of $[n]$ where every block intersects an $\alpha$-region in at most one element.  Let $\Pi_{\alpha}$ denote the set of all $\alpha$-partitions.  A \tdef{bump} of $\mathbf{P}\in\Pi_{\alpha}$ is a pair of consecutive elements in some block of $\mathbf{P}$, or more precisely, a pair $(a,b)$ such that $a$ and $b$ belong to the same block of $\mathbf{P}$, and there exists no $c\in[n]$ with $a<c<b$ such that $a$ and $c$ belong to the same block of $\mathbf{P}$.

\begin{definition}[\cite{muehle18tamari}*{Definition~4.1}]\label{def:parabolic_noncrossing_partition}
	An $\alpha$-partition $\mathbf{P}\in\Pi_{\alpha}$ is \tdef{noncrossing} if it satisfies the following two conditions.
	\begin{description}
		\item[NC1\label{it:nc1}] If two distinct bumps $(a_{1},b_{1})$ and $(a_{2},b_{2})$ of $\mathbf{P}$ satisfy $a_{1}<a_{2}<b_{1}<b_{2}$, then either $a_{1}$ and $a_{2}$ belong to the same $\alpha$-region, or $b_{1}$ and $a_{2}$ belong to the same $\alpha$-region.
		\item[NC2\label{it:nc2}] If two distinct bumps $(a_{1},b_{1})$ and $(a_{2},b_{2})$ of $\mathbf{P}$ satisfy $a_{1}<a_{2}<b_{2}<b_{1}$, then $a_{1}$ and $a_{2}$ belong to different $\alpha$-regions.
	\end{description}
\end{definition}

Let us denote the set of all noncrossing $\alpha$-partitions by $\NC_{\alpha}$.  We represent $\alpha$-partitions graphically by its \tdef{diagram} as follows.  First of all, we draw $n$ vertices on a horizontal line, label them from $1$ through $n$, and group them by $\alpha$-regions.  For $\mathbf{P}\in\Pi_{\alpha}$ we represent a bump $(a,b)$ of $\mathbf{P}$ by a curve that leaves the $a$-th vertex to the bottom, stays below all vertices in the same $\alpha$-region as $a$, then moves up and continues above every vertex in the subsequent $\alpha$-regions, until it reaches the $b$-th vertex, which it enters from the top.  Throughout this article we say ``elements of $\mathbf{P}$'' instead of ``vertices of the diagram of $\mathbf{P}$''.

We can verify that an $\alpha$-partition is noncrossing if and only if it admits a diagram in which no two curves cross.

An example for $\alpha=(1,3,1,2,4,3)$ is the noncrossing $\alpha$-partition 
\begin{displaymath}
	\bigl\{\{1,3\},\{2,6\},\{4,9,12\},\{5\},\{7,8\},\{10,14\},\{11\},\{13\}\bigr\},
\end{displaymath}
whose diagram is shown in the top-central part of Figure~\ref{fig:big_picture}.

If $\alpha=(1,1,\ldots,1)$, then we obtain precisely the noncrossing set partitions studied in \cite{kreweras72sur}.

Let $\mathbf{P}\in\NC_{\alpha}$.  The unique block of $\mathbf{P}$ containing $1$ is the \tdef{first block} of $\mathbf{P}$, and we usually denote it by $P_{\circ}$.  

If $i$ is an element of $\mathbf{P}$, and $(a,b)$ a bump of $\mathbf{P}$ with $a<i\leq b$ such that $a$ and $i$ belong to different $\alpha$-regions, then we say that $i$ \tdef{lies below} $(a,b)$. Moreover, another bump $(a',b')$ of $\mathbf{P}$ \tdef{separates} $i$ from $(a,b)$ if $i$ lies below $(a',b')$ and either $a'<a$ with $a, a'$ in the same $\alpha$-region, or $a<a'$ for $a, a'$ in different $\alpha$-regions. In the diagram, it means that the bump $(a',b')$ is between $i$ and $(a,b)$. 
Finally, $i$ \tdef{lies directly} below $(a,b)$ if either $b=i$ or $i$ is not separated from $(a,b)$ by any other bump of $\mathbf{P}$.

For $P,P'\in\mathbf{P}$ we say that $P'$ \tdef{starts below} $P$ if $\min P'$ lies below some bump of $P$.

\subsection{Parabolic Dyck Paths}
	\label{sec:parabolic_dycks}
A \tdef{Dyck path} is a lattice path in $\mathbb{N}^{2}$ starting from the origin, composed of steps $E\defs(1,0)$ (so-called \tdef{east steps}) and $N\defs(0,1)$ (so-called \tdef{north steps}), ending on the main diagonal while staying always weakly above it.  A Dyck path can be regarded as a word in the alphabet $\{N, E\}$.  For a Dyck path $\nu$, we derive the path $\overline{\nu}$ by reversing $\nu$ and exchanging east and north steps. 

A \tdef{valley} on a Dyck path is a lattice point $\vec{p}$ that is preceded by an east-step and followed by a north-step.  Similarly, a \tdef{peak} is a lattice point $\vec{p}$ that is preceded by a north-step and followed by an east-step.

A Dyck path is \tdef{steep} if it does not have consecutive east-steps before the last north step.  
A Dyck path is \tdef{bounce} if it is of the form $N^{i_{1}}E^{i_{1}}N^{i_{2}}E^{i_{2}}\ldots N^{i_{s}}E^{i_{s}}$, for some strictly positive integers $i_{1},i_{2},\ldots,i_{s}$ (because it ``bounces off'' the main diagonal $s-1$ times).  

The \tdef{$\alpha$-bounce path} is $\nu_{\alpha}\defs N^{\alpha_{1}}E^{\alpha_{1}}N^{\alpha_{2}}E^{\alpha_{2}}\ldots N^{\alpha_{r}}E^{\alpha_{r}}$.  In particular, we have $\overline{\nu}_{\alpha}=\nu_{\overline{\alpha}}$.

\begin{definition}\label{def:parabolic_dyck_path}
	Any Dyck path with $2\lvert\alpha\rvert$ steps that stays weakly above $\nu_{\alpha}$ is an \tdef{$\alpha$-Dyck path}.
\end{definition}

We denote the set of all $\alpha$-Dyck paths by $\Dyck_{\alpha}$.  An example is shown on the left of Figure~\ref{fig:big_picture}.

If $\alpha=(1,1,\ldots,1)$, then we obtain the well-studied Dyck paths.  See for instance \cite{deutsch99dyck} for an enumerative study of these paths.

\subsection{Left-Aligned Colorable Trees}
	\label{sec:lac_trees}
In this section we introduce a new family of combinatorial objects that are naturally parametrized by an integer $n$ and some composition $\alpha$ of $n$.  In Section~\ref{sec:parabolic_cataland_bijections} we show that these objects are in bijection with each the previously mentioned parabolic Catalan objects.  

Let $T$ be a plane rooted tree with vertex set $V$, whose root is denoted by $v_{\circ}$.  Throughout this article we refer to the vertices of a tree as \tdef{nodes}.

A \tdef{partial coloring} of $T$ is a map $C\colon V'\to\mathbb{N}$ for some $V'\subseteq V\setminus\{v_{\circ}\}$.  If $V'=V\setminus\{v_{\circ}\}$, then $C$ is a \tdef{full coloring}.  A \tdef{colored tree} is a pair $(T,C)$, where $T$ is a plane rooted tree, and $C$ is a full coloring of $T$.  We refer to the elements in the image of $C$ as $C$-colors.  

A node $v\in V\setminus\{v_{\circ}\}$ is \tdef{active} with respect to a partial coloring $C$ on $V'\subseteq V$ if $v\notin V'$ and its parent belongs to $V'\cup\{v_{\circ}\}$.

Recall that the \tdef{left-to-right traversal} of a plane rooted tree $T$ is a depth-first search in $T$ starting from the root, where children of the same node are visited \emph{from left to right}. 
The \tdef{LR-prefix order} (resp. \tdef{LR-postfix order}) of $T$ is the linearization of the nodes of $T$, where we record the nodes in order of first (resp. last) visits in left-to-right traversal.

\begin{construction}\label{constr:lac_tree}
	Let $n>0$ and let $\alpha=(\alpha_{1},\alpha_{2},\ldots,\alpha_{r})$ be a composition of $n$.  Let $T$ be a plane rooted tree with $n$ non-root nodes.  
	
	We initialize our algorithm by $\mathcal{T}_{0}=(T,C_{0})$, with $C_0$ the empty coloring.  In the $k$-th step, we denote by $A_{k}$ the set of active nodes of $T$ with respect to the partial coloring $C_{k-1}$.  If $\lvert A_{k}\rvert<\alpha_{k}$, then the algorithm fails.  Otherwise, we let $F_{k}$ be the set of the first $\alpha_{k}$ elements of $A_{k}$ in the LR-prefix order of $T$, and we set $V_{k}=V_{k-1}\uplus F_{k}$.  (Here $\uplus$ denotes disjoint set union.)  We define a partial coloring $C_{k}\colon V_{k}\to\mathbb{N}$ by
	\begin{displaymath}
		C_{k}(v) = \begin{cases}
			C_{k-1}(v), & \text{if}\;v\in V_{k-1},\\
			k, & \text{if}\;v\in F_{k},
		\end{cases}
	\end{displaymath}
	and we set $\mathcal{T}_{k}=(T,C_{k})$.
	
	If the algorithm has not failed after $r$ steps, then we return the colored tree $\mathcal{T}_{r}=(T,C_{r})$.
\end{construction}

Every plane rooted tree for which Construction~\ref{constr:lac_tree} does not fail is \tdef{compatible} with $\alpha$; the corresponding coloring is a \tdef{left-aligned coloring}.  It is clear that every plane rooted tree has at most one left-aligned coloring with respect to $\alpha$.

\begin{definition}\label{def:lac_tree}
	Any plane rooted tree with $\lvert\alpha\rvert$ non-root nodes that is compatible with $\alpha$ is an \tdef{$\alpha$-tree}.
\end{definition}

We denote the set of all $\alpha$-trees by $\Trees_{\alpha}$.  An example is shown in the middle of Figure~\ref{fig:big_picture}.  We remark that every plane rooted tree is compatible with $\alpha=(1,1,\ldots,1)$.  Moreover, if we fix a plane rooted tree $T$, and consider $\alpha_{T}=(\alpha_{1},\alpha_{2},\ldots)$, where $\alpha_{i}$ equals the number of nodes of $T$ whose distance to the root is exactly $i$, then $T$ is compatible with $\alpha_{T}$.  (In every step of Construction~\ref{constr:lac_tree} we simply color all active nodes.)  This composition will play an important role in the last part of this article.

We say that two nodes in a rooted tree are \tdef{comparable} if one is an ancestor of the other. Otherwise, the two nodes are \tdef{incomparable}.

\begin{lemma}\label{lem:lac_property}
	Every $\alpha$-tree has the following properties.
	\begin{enumerate}[(i)]
		\item The color of every node is smaller than the color of its children, and nodes of the same color are incomparable.
		\item The active nodes of each step of the construction will eventually receive colors that are weakly increasing in LR-prefix order.
	\end{enumerate}
\end{lemma}
\begin{proof}
	The first point follows from the fact that a node becomes active before its children do.  The second point follows from the construction of coloring where active nodes are selected for coloring in LR-prefix order, starting always from the first active node.
\end{proof}

\section{Bijections in Parabolic Cataland}
	\label{sec:parabolic_cataland_bijections}
In this section we describe explicit bijections between the parabolic Catalan objects introduced in Section~\ref{sec:members_parabolic_cataland}.  The bijections in the two following subsections~\ref{sec:noncrossings_to_paths}~and~\ref{sec:permutations_to_noncrossings} are reproduced from \cite{muehle18tamari}, the bijections of the remaining subsections~\ref{sec:trees_to_permutations},~\ref{sec:trees_to_noncrossings}~and~\ref{sec:trees_to_paths} are new.  Figure~\ref{fig:example_parabolic} shows the various parabolic Catalan families for the compositions of $3$.

\begin{figure}
	\includegraphics[width=0.29\textwidth,page=12]{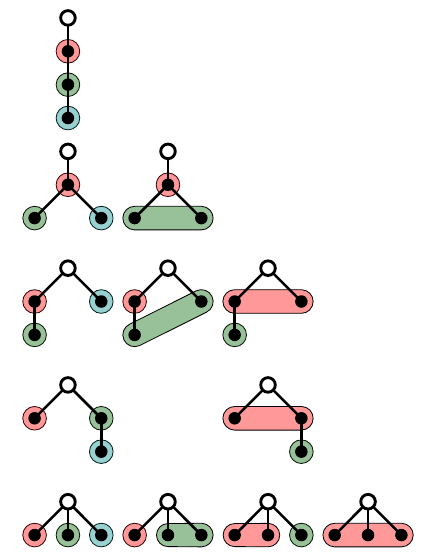}\includegraphics[width=0.31\textwidth,page=14]{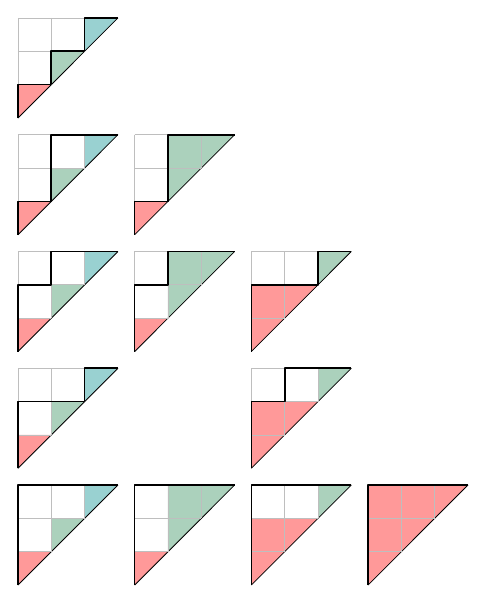}\includegraphics[width=0.4\textwidth,page=15]{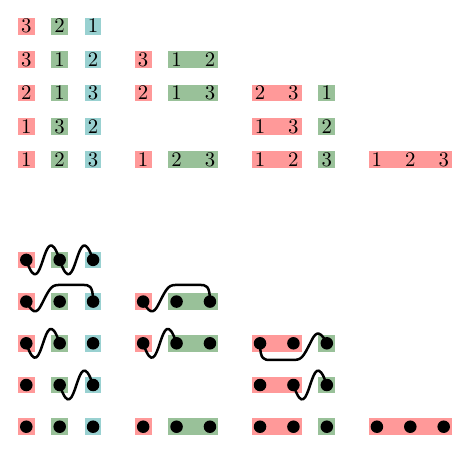}	
	\caption{Examples of $\alpha$-trees, $\alpha$-Dyck paths, $(\alpha,231)$-avoiding permutations and noncrossing $\alpha$-partitions of size $n=3$. Different columns correspond to different $\alpha$'s.}
	\label{fig:example_parabolic}
\end{figure}

\subsection{Noncrossing Partitions and Dyck Paths}
	\label{sec:noncrossings_to_paths}
The following construction associates a noncrossing $\alpha$-partition with each $\overline{\alpha}$-Dyck path.

\begin{construction}\label{constr:path_to_nc}
	Let $\alpha=(\alpha_{1},\alpha_{2},\ldots,\alpha_{r})$ be an integer composition.  Let $\mu\in\Dyck_{\overline{\alpha}}$, whose valleys are $(p_{1},q_{1}),(p_{2},q_{2}),\ldots,(p_{m},q_{m})$ with $p_{1}<p_{2}<\cdots<p_{m}$.  We set $p_{0}=0$.
	
	We initialize our algorithm with the $\alpha$-partition $\mathbf{P}_{0}$ without any bumps.  In the $i$-th step, we construct the $\alpha$-partition $\mathbf{P}_{i}$ from $\mathbf{P}_{i-1}$ by adding a bump corresponding to the valley $(p_{i},q_{i})$.  To do so, we pick the unique index $0 \leq k < r$ such that $t_{k}\leq q_{i}\leq t_{k+1}$, so that we can write $q_{i}=t_{k}+\ell$ for some $\ell\in\{0,1,\ldots,t_{k+1}-t_{k}\}$.  We then obtain $\mathbf{P}_{i}$ by adding the bump $(a,b)$ to $\mathbf{P}_{i-1}$, where $a$ is the $(\ell+1)$-st element in the $(r-k)$-th $\alpha$-region and $b$ is the $(p_{i}-p_{i-1})$-th element after the $(r-k)$-th $\alpha$-region that is not already below some bump. 
	
	After $m$ steps we return the $\alpha$-partition $\mathbf{P}_{\mu}=\mathbf{P}_{m}$.
\end{construction}

Two examples of this construction are illustrated in the top part of Figure~\ref{fig:cover_map}.  

\begin{theorem}[\cite{muehle18tamari}*{Theorem~5.2}]\label{thm:bijection_paths_nc}
	For $\mu\in\Dyck_{\overline{\alpha}}$ the $\alpha$-partition $\mathbf{P}_{\mu}$ is noncrossing.  Moreover, the map
	\begin{equation}\label{eq:dyck_to_nc}
		\Theta_{1}\colon\Dyck_{\overline{\alpha}}\to\NC_{\alpha},\quad \mu\mapsto\mathbf{P}_{\mu}.
	\end{equation}
	is a bijection that sends valleys to bumps.
\end{theorem}

\subsection{$231$-Avoiding Permutations and Noncrossing Partitions}
	\label{sec:permutations_to_noncrossings}
Recall that a binary relation $R\subseteq M\times M$, for $M$ a finite set, is \tdef{acyclic} if there does not exist a sequence $(a_{1},a_{2}),(a_{2},a_{3}),\ldots,(a_{k},a_{1})\in R$ for any $k>0$.  It is well known that an acyclic relation can be extended to a partial order by taking the reflexive and transitive closure.

We now define a binary relation $\vec{R}_{\mathbf{P}}$ on $\mathbf{P}$ by setting $(P,P')\in\vec{R}_{\mathbf{P}}$ if and only if $P'$ starts below $P$.

\begin{lemma}\label{lem:block_relation_acyclic}
	For every $\mathbf{P}\in\NC_{\alpha}$, the relation $\vec{R}_{\mathbf{P}}$ is acyclic.
\end{lemma}
\begin{proof}
	This follows from the fact that $(P,P')\in\vec{R}_{\mathbf{P}}$ implies $\min P<\min P'$.
\end{proof}

As a consequence, we may define a partial order $\vec{O}_{\mathbf{P}}$ on the blocks of $\mathbf{P}$ by taking the reflexive and transitive closure of $\vec{R}_{\mathbf{P}}$.  By construction the first block of $\mathbf{P}$ is a minimal element of the poset $\bigl(\mathbf{P},\vec{O}_{\mathbf{P}}\bigr)$.

\begin{example}\label{ex:nc_to_perm_1}
	Let $\alpha=(1,3,1,2,4,3,1)$, and consider the noncrossing $\alpha$-partition
	\begin{displaymath}
		\mathbf{P} = \bigl\{\underset{P_{1}}{\underset{\rotatebox{90}{=}}{\{1,3,10\}}},\underset{P_{2}}{\underset{\rotatebox{90}{=}}{\{2,6\}}},\underset{P_{3}}{\underset{\rotatebox{90}{=}}{\{4\}}},\underset{P_{4}}{\underset{\rotatebox{90}{=}}{\{5\}}},\underset{P_{5}}{\underset{\rotatebox{90}{=}}{\{7,8\}}},\underset{P_{6}}{\underset{\rotatebox{90}{=}}{\{9,13\}}},\underset{P_{7}}{\underset{\rotatebox{90}{=}}{\{11,14\}}},\underset{P_{8}}{\underset{\rotatebox{90}{=}}{\{12,15\}}}\bigr\}.
	\end{displaymath}
	The diagram of $\mathbf{P}$ is shown below, with its first block highlighted.
	\begin{center}\begin{tikzpicture}\small
		\def\x{.75};
		\def\s{.75*\x};
		\draw(1*\x,1*\x) node[fill,circle,scale=\s](n1){};
		\draw(2*\x,1*\x) node[fill,circle,scale=\s](n2){};
		\draw(3*\x,1*\x) node[fill,circle,scale=\s](n3){};
		\draw(4*\x,1*\x) node[fill,circle,scale=\s](n4){};
		\draw(5*\x,1*\x) node[fill,circle,scale=\s](n5){};
		\draw(6*\x,1*\x) node[fill,circle,scale=\s](n6){};
		\draw(7*\x,1*\x) node[fill,circle,scale=\s](n7){};
		\draw(8*\x,1*\x) node[fill,circle,scale=\s](n8){};
		\draw(9*\x,1*\x) node[fill,circle,scale=\s](n9){};
		\draw(10*\x,1*\x) node[fill,circle,scale=\s](n10){};
		\draw(11*\x,1*\x) node[fill,circle,scale=\s](n11){};
		\draw(12*\x,1*\x) node[fill,circle,scale=\s](n12){};
		\draw(13*\x,1*\x) node[fill,circle,scale=\s](n13){};
		\draw(14*\x,1*\x) node[fill,circle,scale=\s](n14){};
		\draw(15*\x,1*\x) node[fill,circle,scale=\s](n15){};
		\begin{pgfonlayer}{background}
			\fill[c1](.75*\x,1*\x) -- (.75*\x,.8*\x) -- (1.25*\x,.8*\x) -- (1.25*\x,1.2*\x) -- (.75*\x,1.2*\x) -- cycle;
			\fill[c2](1.75*\x,1*\x) -- (1.75*\x,.8*\x) -- (4.25*\x,.8*\x) -- (4.25*\x,1.2*\x) -- (1.75*\x,1.2*\x) -- cycle;
			\fill[c3](4.75*\x,1*\x) -- (4.75*\x,.8*\x) -- (5.25*\x,.8*\x) -- (5.25*\x,1.2*\x) -- (4.75*\x,1.2*\x) -- cycle;
			\fill[c4](5.75*\x,1*\x) -- (5.75*\x,.8*\x) -- (7.25*\x,.8*\x) -- (7.25*\x,1.2*\x) -- (5.75*\x,1.2*\x) -- cycle;
			\fill[c5](7.75*\x,1*\x) -- (7.75*\x,.8*\x) -- (11.25*\x,.8*\x) -- (11.25*\x,1.2*\x) -- (7.75*\x,1.2*\x) -- cycle;
			\fill[c6](11.75*\x,1*\x) -- (11.75*\x,.8*\x) -- (14.25*\x,.8*\x) -- (14.25*\x,1.2*\x) -- (11.75*\x,1.2*\x) -- cycle;
			\fill[orange!80!gray](14.75*\x,1*\x) -- (14.75*\x,.8*\x) -- (15.25*\x,.8*\x) -- (15.25*\x,1.2*\x) -- (14.75*\x,1.2*\x) -- cycle;
		\end{pgfonlayer}
		\draw[very thick,red!50!gray](n1) .. controls (1.1*\x,.75*\x) and (1.4*\x,.75*\x) .. (1.5*\x,1*\x) .. controls (1.75*\x,1.5*\x) and (2.75*\x,1.5*\x) .. (n3);
		\draw[thick](n2) .. controls (2.25*\x,.25*\x) and (4.4*\x,.25*\x) .. (4.6*\x,1*\x) .. controls (4.8*\x,1.5*\x) and (5.75*\x,1.5*\x) .. (n6);
		\draw[very thick,red!50!gray](n3) .. controls (3.25*\x,.5*\x) and (4.2*\x,.5*\x) .. (4.4*\x,1*\x) .. controls (4.5*\x,2*\x) and (9.75*\x,2*\x) .. (n10);
		\draw[thick](n7) .. controls (7.1*\x,.75*\x) and (7.4*\x,.75*\x) .. (7.5*\x,1*\x) .. controls (7.6*\x,1.25*\x) and (7.9*\x,1.25*\x) .. (n8);
		\draw[thick](n9) .. controls (9.25*\x,.25*\x) and (11.4*\x,.25*\x) .. (11.6*\x,1*\x) .. controls (11.8*\x,1.5*\x) and (12.75*\x,1.5*\x) .. (n13);
		\draw[thick](n11) .. controls (11.1*\x,.75*\x) and (11.3*\x,.75*\x) .. (11.4*\x,1*\x) .. controls (11.6*\x,1.75*\x) and (13.75*\x,1.75*\x) .. (n14);
		\draw[thick](n12) .. controls (12.25*\x,.5*\x) and (14.25*\x,.5*\x) .. (14.5*\x,1*\x) .. controls (14.6*\x,1.25*\x) and (14.9*\x,1.25*\x) .. (n15);
	\end{tikzpicture}\end{center}
		
	The block $P_{2}$, for instance, starts below $P_{1}$ because $(1,3)$ is a bump and $\min P_{2}=2$.  The relation $\vec{R}_{\mathbf{P}}$ is the following:
	\begin{displaymath}
		\vec{R}_{\mathbf{P}} = \bigl\{(P_{1},P_{2}),(P_{1},P_{4}),(P_{1},P_{5}),(P_{1},P_{6}),(P_{2},P_{4}),(P_{6},P_{8}),(P_{7},P_{8})\bigr\}.
	\end{displaymath}
	The poset diagram of $(\mathbf{P},\vec{O}_{\mathbf{P}}\bigr)$ is shown below.
	\begin{center}\begin{tikzpicture}
		\def\x{1.5};
		\draw(1*\x,1) node(n1){$\{4\}$};
		\draw(2*\x,1) node(n2){$\{1,3,10\}$};
		\draw(1*\x,2) node(n3){$\{2,6\}$};
		\draw(2*\x,2) node(n4){$\{7,8\}$};
		\draw(3*\x,2) node(n5){$\{9,13\}$};
		\draw(4*\x,2) node(n6){$\{11,14\}$};
		\draw(1*\x,3) node(n7){$\{5\}$};
		\draw(3*\x,3) node(n8){$\{12,15\}$};
		\draw(n2) -- (n3);
		\draw(n2) -- (n4);
		\draw(n2) -- (n5);
		\draw(n3) -- (n7);
		\draw(n5) -- (n8);
		\draw(n6) -- (n8);
	\end{tikzpicture}\end{center}
\end{example}

Now let $\lvert\alpha\rvert=n$, and $\mathbf{P}\in\NC_{\alpha}$ whose first block is $P_{\circ}$.  Further, let $\mathbf{P}'=\mathbf{P}\setminus P_{\circ}$.  We may view $\mathbf{P}'$ as a noncrossing $\alpha'$-partition of $[n]\setminus P_{\circ}$ for some appropriate composition $\alpha'$.  By construction, we obtain $\vec{O}_{\mathbf{P}'}$ by removing every pair of the form $(P_{\circ},\cdot)$ from $\vec{O}_{\mathbf{P}}$.

\begin{construction}\label{constr:nc_to_perm}
	Let $\alpha$ be an integer composition, and let $\lvert\alpha\rvert=n$.  Let $\mathbf{P}\in\NC_{\alpha}$ whose first block is $P_{\circ}=\{i_{1},i_{2},\ldots,i_{s}\}$ with $i_{1}<i_{2}<\cdots<i_{s}$.  (Thus, $i_{1}=1$.)
	
	We construct a permutation $w_{\mathbf{P}}$ recursively such that 
	\begin{displaymath}
		w_{\mathbf{P}}(i_{1}) = w_{\mathbf{P}}(i_{2})+1 = \cdots = w_{\mathbf{P}}(i_{s})+s-1,
	\end{displaymath}
	and $w_{\mathbf{P}}(i_{1})$ is as small as possible.  To achieve that, we determine
	\begin{displaymath}
		D = \bigcup\{P\in\mathbf{P}\mid (P_{\circ},P)\in\vec{O}_{\mathbf{P}}\},
	\end{displaymath}
	and we set $w_{\mathbf{P}}(i_{1})=\lvert D\rvert$.  (In other words, $D$ consists of all elements of $[n]$ that belong to a block of $\mathbf{P}$ that lies weakly above $P_{\circ}$ in $\bigl(\mathbf{P},\vec{O}_{\mathbf{P}}\bigr)$.)
	
	Now we break $\mathbf{P}$ into two pieces, namely 
	\begin{align*}
		\mathbf{P}_{1} & = \bigl\{P\cap D\mid P\in\mathbf{P}\setminus P_{\circ}\bigr\}\setminus\{\emptyset\},\\
		\mathbf{P}_{2} & = \bigl\{P\cap \bigl([n]\setminus D\bigr)\mid P\in\mathbf{P}\setminus P_{\circ}\bigr\}\setminus\{\emptyset\}.
	\end{align*}
	We may view $\mathbf{P}_{1}$ and $\mathbf{P}_{2}$ as noncrossing partitions with respect to some integers $n_{1}$ and $n_{2}$ and compositions $\alpha_{1}$ of $n_{1}$ and $\alpha_{2}$ of $n_{2}$, respectively.  In particular, $n_{1},n_{2}<n$, so that we may determine the missing values of $w_{\mathbf{P}}$ via this recursion.  We will have $w_{\mathbf{P}}(i)>\lvert D\rvert$ if $i\in[n]\setminus D$, and $w_{\mathbf{P}}(i)<\lvert D\rvert+s-1$ if $i\in D\setminus P_{\circ}$.

	The initial condition for this recursion assigns the identity permutation to the $\alpha$-partition without bumps.  
\end{construction}

\begin{example}\label{ex:nc_to_perm_2}
	Let us continue Example~\ref{ex:nc_to_perm_1}.  The first block of $\mathbf{P}$ is $P_{\circ}=\{1,3,10\}$, and we see from the poset diagram above that $D=\{1,2,3,5,6,7,8,9,10,12,13,15\}$.  Therefore we set $w_{\mathbf{P}}(1)=12$, $w_{\mathbf{P}}(3)=11$, and $w_{\mathbf{P}}(10)=10$.  The two smaller partitions are $\mathbf{P}_{1}=\bigl\{\{2,6\},\{5\},\{7,8\},\{9,13\},\{12,15\}\bigr\}$ and $\mathbf{P}_{2}=\bigl\{\{4\},\{11,14\}\bigr\}$ whose diagrams are shown below.  
	\begin{center}\begin{tikzpicture}\small
		\def\x{.75};
		\def\s{.75*\x};
		\draw(1*\x,2.5*\x) node[fill,circle,scale=\s](n1){};
		\draw(2*\x,2.5*\x) node[fill,circle,scale=\s](n2){};
		\draw(3*\x,2.5*\x) node[fill,circle,scale=\s](n3){};
		\draw(4*\x,2.5*\x) node[fill,circle,scale=\s](n4){};
		\draw(5*\x,2.5*\x) node[fill,circle,scale=\s](n5){};
		\draw(6*\x,2.5*\x) node[fill,circle,scale=\s](n6){};
		\draw(7*\x,2.5*\x) node[fill,circle,scale=\s](n7){};
		\draw(8*\x,2.5*\x) node[fill,circle,scale=\s](n8){};
		\draw(9*\x,2.5*\x) node[fill,circle,scale=\s](n9){};
		\draw(10*\x,2.5*\x) node[fill,circle,scale=\s](n10){};
		\draw(11*\x,2.5*\x) node[fill,circle,scale=\s](n11){};
		\draw(12*\x,2.5*\x) node[fill,circle,scale=\s](n12){};
		\draw(13*\x,2.5*\x) node[fill,circle,scale=\s](n13){};
		\draw(14*\x,2.5*\x) node[fill,circle,scale=\s](n14){};
		\draw(15*\x,2.5*\x) node[fill,circle,scale=\s](n15){};
		\begin{pgfonlayer}{background}
			\fill[c1](.75*\x,2.5*\x) -- (.75*\x,2.3*\x) -- (1.25*\x,2.3*\x) -- (1.25*\x,2.7*\x) -- (.75*\x,2.7*\x) -- cycle;
			\fill[c2](1.75*\x,2.5*\x) -- (1.75*\x,2.3*\x) -- (4.25*\x,2.3*\x) -- (4.25*\x,2.7*\x) -- (1.75*\x,2.7*\x) -- cycle;
			\fill[c3](4.75*\x,2.5*\x) -- (4.75*\x,2.3*\x) -- (5.25*\x,2.3*\x) -- (5.25*\x,2.7*\x) -- (4.75*\x,2.7*\x) -- cycle;
			\fill[c4](5.75*\x,2.5*\x) -- (5.75*\x,2.3*\x) -- (7.25*\x,2.3*\x) -- (7.25*\x,2.7*\x) -- (5.75*\x,2.7*\x) -- cycle;
			\fill[c5](7.75*\x,2.5*\x) -- (7.75*\x,2.3*\x) -- (11.25*\x,2.3*\x) -- (11.25*\x,2.7*\x) -- (7.75*\x,2.7*\x) -- cycle;
			\fill[c6](11.75*\x,2.5*\x) -- (11.75*\x,2.3*\x) -- (14.25*\x,2.3*\x) -- (14.25*\x,2.7*\x) -- (11.75*\x,2.7*\x) -- cycle;
			\fill[orange!80!gray](14.75*\x,2.5*\x) -- (14.75*\x,2.3*\x) -- (15.25*\x,2.3*\x) -- (15.25*\x,2.7*\x) -- (14.75*\x,2.7*\x) -- cycle;
		\end{pgfonlayer}
		\fill[white,opacity=.7](.75*\x,2.5*\x) -- (.75*\x,2.3*\x) -- (1.25*\x,2.3*\x) -- (1.25*\x,2.7*\x) -- (.75*\x,2.7*\x) -- cycle;
		\fill[white,opacity=.7](2.5*\x,2.5*\x) -- (2.5*\x,2.3*\x) -- (4.25*\x,2.3*\x) -- (4.25*\x,2.7*\x) -- (2.5*\x,2.7*\x) -- cycle;
		\fill[white,opacity=.7](9.5*\x,2.5*\x) -- (9.5*\x,2.3*\x) -- (11.25*\x,2.3*\x) -- (11.25*\x,2.7*\x) -- (9.5*\x,2.7*\x) -- cycle;
		\fill[white,opacity=.7](13.5*\x,2.5*\x) -- (13.5*\x,2.3*\x) -- (14.25*\x,2.3*\x) -- (14.25*\x,2.7*\x) -- (13.5*\x,2.7*\x) -- cycle;
		\draw[thick](n2) .. controls (2.25*\x,1.75*\x) and (4.4*\x,1.75*\x) .. (4.6*\x,2.5*\x) .. controls (4.8*\x,3*\x) and (5.75*\x,3*\x) .. (n6);
		\draw[thick](n7) .. controls (7.1*\x,2.25*\x) and (7.4*\x,2.25*\x) .. (7.5*\x,2.5*\x) .. controls (7.6*\x,2.75*\x) and (7.9*\x,2.75*\x) .. (n8);
		\draw[thick](n9) .. controls (9.25*\x,1.75*\x) and (11.4*\x,1.75*\x) .. (11.6*\x,2.5*\x) .. controls (11.8*\x,3*\x) and (12.75*\x,3*\x) .. (n13);
		\draw[thick](n12) .. controls (12.25*\x,2*\x) and (14.25*\x,2*\x) .. (14.5*\x,2.5*\x) .. controls (14.6*\x,2.75*\x) and (14.9*\x,2.75*\x) .. (n15);
		\draw(1*\x,1*\x) node[fill,circle,scale=\s](m1){};
		\draw(2*\x,1*\x) node[fill,circle,scale=\s](m2){};
		\draw(3*\x,1*\x) node[fill,circle,scale=\s](m3){};
		\draw(4*\x,1*\x) node[fill,circle,scale=\s](m4){};
		\draw(5*\x,1*\x) node[fill,circle,scale=\s](m5){};
		\draw(6*\x,1*\x) node[fill,circle,scale=\s](m6){};
		\draw(7*\x,1*\x) node[fill,circle,scale=\s](m7){};
		\draw(8*\x,1*\x) node[fill,circle,scale=\s](m8){};
		\draw(9*\x,1*\x) node[fill,circle,scale=\s](m9){};
		\draw(10*\x,1*\x) node[fill,circle,scale=\s](m10){};
		\draw(11*\x,1*\x) node[fill,circle,scale=\s](m11){};
		\draw(12*\x,1*\x) node[fill,circle,scale=\s](m12){};
		\draw(13*\x,1*\x) node[fill,circle,scale=\s](m13){};
		\draw(14*\x,1*\x) node[fill,circle,scale=\s](m14){};
		\draw(15*\x,1*\x) node[fill,circle,scale=\s](m15){};
		\begin{pgfonlayer}{background}
			\fill[c1](.75*\x,1*\x) -- (.75*\x,.8*\x) -- (1.25*\x,.8*\x) -- (1.25*\x,1.2*\x) -- (.75*\x,1.2*\x) -- cycle;
			\fill[c2](1.75*\x,1*\x) -- (1.75*\x,.8*\x) -- (4.25*\x,.8*\x) -- (4.25*\x,1.2*\x) -- (1.75*\x,1.2*\x) -- cycle;
			\fill[c3](4.75*\x,1*\x) -- (4.75*\x,.8*\x) -- (5.25*\x,.8*\x) -- (5.25*\x,1.2*\x) -- (4.75*\x,1.2*\x) -- cycle;
			\fill[c4](5.75*\x,1*\x) -- (5.75*\x,.8*\x) -- (7.25*\x,.8*\x) -- (7.25*\x,1.2*\x) -- (5.75*\x,1.2*\x) -- cycle;
			\fill[c5](7.75*\x,1*\x) -- (7.75*\x,.8*\x) -- (11.25*\x,.8*\x) -- (11.25*\x,1.2*\x) -- (7.75*\x,1.2*\x) -- cycle;
			\fill[c6](11.75*\x,1*\x) -- (11.75*\x,.8*\x) -- (14.25*\x,.8*\x) -- (14.25*\x,1.2*\x) -- (11.75*\x,1.2*\x) -- cycle;
			\fill[orange!80!gray](14.75*\x,1*\x) -- (14.75*\x,.8*\x) -- (15.25*\x,.8*\x) -- (15.25*\x,1.2*\x) -- (14.75*\x,1.2*\x) -- cycle;
		\end{pgfonlayer}
		\fill[white,opacity=.7](.75*\x,1*\x) -- (.75*\x,.8*\x) -- (1.25*\x,.8*\x) -- (1.25*\x,1.2*\x) -- (.75*\x,1.2*\x) -- cycle;
		\fill[white,opacity=.7](1.75*\x,1*\x) -- (1.75*\x,.8*\x) -- (3.5*\x,.8*\x) -- (3.5*\x,1.2*\x) -- (1.75*\x,1.2*\x) -- cycle;
		\fill[white,opacity=.7](4.75*\x,1*\x) -- (4.75*\x,.8*\x) -- (5.25*\x,.8*\x) -- (5.25*\x,1.2*\x) -- (4.75*\x,1.2*\x) -- cycle;
		\fill[white,opacity=.7](5.75*\x,1*\x) -- (5.75*\x,.8*\x) -- (7.25*\x,.8*\x) -- (7.25*\x,1.2*\x) -- (5.75*\x,1.2*\x) -- cycle;
		\fill[white,opacity=.7](7.75*\x,1*\x) -- (7.75*\x,.8*\x) -- (10.5*\x,.8*\x) -- (10.5*\x,1.2*\x) -- (7.75*\x,1.2*\x) -- cycle;
		\fill[white,opacity=.7](11.75*\x,1*\x) -- (11.75*\x,.8*\x) -- (13.5*\x,.8*\x) -- (13.5*\x,1.2*\x) -- (11.75*\x,1.2*\x) -- cycle;
		\fill[white,opacity=.7](14.75*\x,1*\x) -- (14.75*\x,.8*\x) -- (15.25*\x,.8*\x) -- (15.25*\x,1.2*\x) -- (14.75*\x,1.2*\x) -- cycle;
		\draw[thick](m11) .. controls (11.1*\x,.75*\x) and (11.3*\x,.75*\x) .. (11.4*\x,1*\x) .. controls (11.6*\x,1.75*\x) and (13.75*\x,1.75*\x) .. (m14);
	\end{tikzpicture}\end{center}
	
	Recursively, we obtain two permutations $w_{1}=w_{\mathbf{P}_{1}}=3\mid 1\mid 2\;5\mid 4\;9\mid 7\;8\mid 6$ and $w_{2}=w_{\mathbf{P}_{2}}=1\mid 3\mid 2$.  We now embed these permutations into $w_{\mathbf{P}}$, where we have to increase the values of $w_{2}$ by $\lvert D\rvert=12$.  The resulting permutation is
	\begin{displaymath}
		w_{\mathbf{P}} = 12\mid 3\;11\;13\mid 1\mid 2\;5\mid 4\;9\;10\;15\mid 7\;8\;14\mid 6.
	\end{displaymath}
\end{example}

Two more examples of this construction are illustrated in the bottom part of Figure~\ref{fig:cover_map}.

\begin{theorem}[{\cite[Theorem~4.2]{muehle18tamari}}]\label{thm:bijection_nc_perms}
	For $\mathbf{P}\in\NC_{\alpha}$ the permutation $w_{\mathbf{P}}$ is $(\alpha,231)$-avoiding.  Moreover, the map
	\begin{equation}\label{eq:nc_to_perm}
		\Theta_{2}\colon\NC_{\alpha}\to\Symmetric_{\alpha}(231),\quad\mathbf{P}\mapsto w_{\mathbf{P}}
	\end{equation}
	is a bijection that sends bumps to descents.
\end{theorem}

\subsection{$231$-Avoiding Permutations and Left-Aligned Colorable Trees}
	\label{sec:trees_to_permutations}
The following construction associates a permutation in $\Symmetric_{\alpha}$ with an $\alpha$-tree.

\begin{construction}\label{constr:lac_to_perm}
	Let $\alpha$ be an integer composition.  Let $T\in\Trees_{\alpha}$ whose left-aligned coloring produced by Construction~\ref{constr:lac_tree} is $C$.  
	
	We label the non-root nodes of $T$ in LR-postfix order.  We then group together the labels that correspond to nodes with the same color under $C$, and order these labels increasingly. Finally, we order the blocks according to their color, and obtain a permutation $w_{T}\in\Symmetric_{\alpha}$.
\end{construction}

An example of Construction~\ref{constr:lac_to_perm} is illustrated in Figure~\ref{fig:tree_perm_bij}.  

\begin{figure}
	\centering
	\includegraphics[page=1,width=\textwidth]{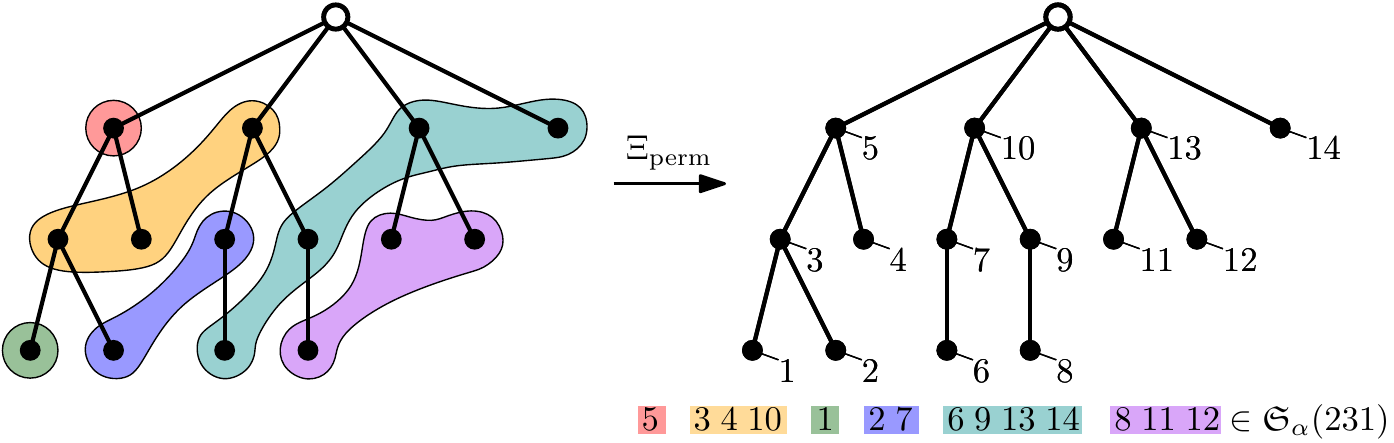}
	\caption{An illustration of Construction~\ref{constr:lac_to_perm} for $\alpha=(1,3,1,2,4,3)$.}
	\label{fig:tree_perm_bij}
\end{figure}

\begin{lemma}\label{lem:postfix_next}
	For a non-root node $u$ in a plane rooted tree $T$, the next node $v$ in the LR-postfix order is either its parent, of which $u$ is the last child, or the left-most leaf of the subtree induced by its sibling immediately to the right.
\end{lemma}
\begin{proof}
	This follows from the definition of the LR-postfix order.
\end{proof}

Construction~\ref{constr:lac_to_perm} suggests the definition of the following map, which we next show is well-defined:
\begin{equation}\label{eq:lac_to_perm}
	\lactoperm\colon\Trees_{\alpha}\to\Symmetric_{\alpha}(231),\quad T\mapsto w_{T}.
\end{equation}

\begin{proposition} \label{prop:lac_to_perm_valid}
	For an $\alpha$-tree $T$, the permutation $w_{T}$ obtained by Construction~\ref{constr:lac_to_perm} is $(\alpha,231)$-avoiding.
\end{proposition}
\begin{proof}
	Let $T$ be an $\alpha$-tree with corresponding left-aligned coloring $C$.  Suppose that $(i,j,k)$ is an $(\alpha,231)$-pattern in $w_{T}$.  Let $u_{i},u_{j},u_{k}$ denote the nodes of $T$ labeled by $i,j,k$, respectively, in Construction~\ref{constr:lac_to_perm}.  It follows that $C(u_{i})<C(u_{j})<C(u_{k})$.  
	
	By Lemma~\ref{lem:postfix_next}, since $w_{T}(k)+1=w_{T}(i)$, the node $u_{i}$ is either the parent of $u_{k}$, or the left-most leaf of the subtree induced by the sibling of $u_{k}$ on the right.  However, the second case cannot happen, since the node $u_{k}$ would then become active before $u_{i}$ (because $u_{k}$ precedes $u_{i}$ in LR-prefix order).  By Lemma~\ref{lem:lac_property}(ii), we would have $C(u_{i})\geq C(u_{k})$, which is a contradiction.
	
	Therefore, $u_{i}$ is the parent of $u_{k}$, which means that $u_{k}$ is active during steps $i+1,i+2,\ldots,k$ in Construction~\ref{constr:lac_tree}, in particular in step $j$.  Since $C(u_{j})=j<C(u_{k})$, the node $u_{j}$ must come before $u_{k}$ in LR-prefix order.  However, since $w(i)<w(j)$, we know that $u_{j}$ comes after $u_{i}$ in LR-postfix order, and $C(u_{i})<C(u_{j})$ means that $u_{j}$ is not an ancestor of $u_{i}$.  Therefore, $u_{j}$ must come after $u_{i}$ in LR-prefix order, and thus also after $u_{k}$.  This is a contradiction, and we conclude that $w_{T}$ is $(\alpha,231)$-avoiding.
\end{proof}

Let us now describe how to obtain a plane rooted tree from an $(\alpha,231)$-avoiding permutation.  

\begin{construction}\label{constr:perm_to_lac}
	Let $\alpha$ be an integer composition, and let $w\in\Symmetric_{\alpha}(231)$.  
	
	We initialize our algorithm with the tree $T_{0}$ that has a single node labeled by $\lvert\alpha\rvert+1$.  In the $i$-th step we create the tree $T_{i}$ by inserting a node labeled by $w(i)$ into $T_{i-1}$.  This insertion proceeds as follows.  We start at the root and walk around $T_{i-1}$.  Suppose that we have reached a node $v$ with label $a$.  If $w(i)<a$, then we move to the left-most child of $v$, otherwise we move to the first sibling of $v$ on the right.  If the destination does not exist, then we create it and label it by $w(i)$.
	
	After $\lvert\alpha\rvert$ steps we return the plane rooted tree $T_{w}=T_{\lvert\alpha\rvert}$.
\end{construction}

Readers familiar with binary search trees may notice that this construction is essentially the insertion algorithm for binary search trees composed with the classical bijection between binary and plane trees.

It turns out that the output of Construction~\ref{constr:perm_to_lac} is a plane rooted tree compatible with $\alpha$, which enables us to define the following map:
\begin{equation}\label{eq:perm_to_lac}
	\permtolac\colon\Symmetric_{\alpha}(231)\to\Trees_{\alpha},\quad w\mapsto T_{w}.
\end{equation}

\begin{proposition} \label{prop:perm_to_lac_valid}
	For an $(\alpha,231)$-avoiding permutation $w$, the plane rooted tree $T_{w}$ obtained by Construction~\ref{constr:perm_to_lac} is compatible with $\alpha$.	
\end{proposition}
\begin{proof}
	Let us define a full coloring $f$ of $T_{w}$ as follows.  For every node $v$ of $T_{w}$ with label $a$ we set $f(v)=i$ if and only if $w^{-1}(a)$ belongs to the $i$-th $\alpha$-region.  It remains to prove that $f$ is the left-aligned coloring of $T_{w}$ produced in Construction~\ref{constr:lac_tree}.  (Because in that case, Construction~\ref{constr:lac_tree} does not fail, and $T_{w}$ is indeed an $\alpha$-tree.)

	Construction~\ref{constr:perm_to_lac} implies that the nodes of $T_{w}$ are labeled in LR-postfix order.  Now we show that for every non-root node $u$ of $T_{w}$ whose parent $v$ is not the root, we have $f(u)>f(v)$.  Since $v$ is the parent of $u$, it is inserted before $u$, which yields $f(u)\geq f(v)$.  If $f(u)=f(v)$, then the labels of $u$ and $v$ must come from indices in the same $\alpha$-region.  However, by construction, the label of $u$ is smaller than that of $v$, which would produce an inversion in their common $\alpha$-region, contradicting the assumption that $w\in\Symmetric_{\alpha}$.  We conclude that nodes of $T_{w}$ labeled by elements coming from the same $\alpha$-region are incomparable.
	
	Now let $\alpha=(\alpha_{1},\alpha_{2},\ldots,\alpha_{r})$. 
We now prove that $T_w$ is an $\alpha$-tree, and $f$ is its left-aligned coloring. We consider the partial colorings $C_{s}$ obtained after step $s$ of Construction~\ref{constr:lac_tree}, and proceed by induction to prove that, for all $0\leq s\leq r$, the algorithm does not fail in step $s$, and $C_s$ is the same as the restriction $f_{s}$ of $f$ to values from $1$ to $s$. The base case $s=0$ is clearly correct. Now suppose that the induction hypothesis holds for $s-1$.

	We first show that there are always enough active nodes for step $s$. Since $w \in \Symmetric_\alpha(231)$, labels $w(i)$ for $i$ in the $s$-th $\alpha$-region are inserted in increasing order, and by Construction~\ref{constr:perm_to_lac}, these newly inserted nodes cannot be comparable. Therefore, they are all children of nodes in the domain of $C_{s-1}$, thus active at step $s$.  We therefore have at least $\alpha_s$ active nodes, and $C_s$ is well defined.
	
	\medskip	
	
	We now show that $C_s$ coincides with $f_s$. Suppose otherwise, and let $u_{*}$ be the first node of $T_{w}$ in LR-prefix order that is in the domain of $C_s$ but not in that of $f_s$.  We have $f(u_{*}) > s$.  Let $u_{i}$ be the parent of $u_{*}$, which is in the domain of $C_{s-1}$, and therefore in that of $f_{s-1}$ by induction hypothesis. We thus have $f(u_{i})\leq s-1$.  
	
	By minimality of $u_{*}$, the node $u_{i}$ is in the domain of $C_s$.  Let $u_{k}$ be the rightmost child of $u_{i}$ (which exists as $u_{i}$ has a least one child $u_{*}$).  Since $u_{k}$ is inserted after $u_{*}$, we conclude that $f(u_{k})\geq f(u_{*})$. It is clear that the domains of $C_s$ and of $f_s$ are of the same size. Hence, there is a node $u_j$ in the domain of $f_s$ but not in that of $C_s$, and we take it the minimal in LR-prefix order.
	
	By induction hypothesis, we have $f(u_j) \geq s$, thus $f(u_j)=s$. Since nodes are colored in LR-prefix order in Construction~\ref{constr:perm_to_lac}, $u_*$ is also the first node that differs in the domains of $C_s$ and $f_s$, meaning that $u_*$ precedes $u_j$ in LR-prefix order. If $u_j$ is a child of $u_i$, since $f(u_*)>s=f(u_j)$, it would lead to a contradiction to Construction~\ref{constr:perm_to_lac}. Therefore, $u_j$ comes after $u_i$ in LR-postfix order. 
	
	Let $i,j,k$ be the indices such that $u_{i},u_{j},u_{k}$ are labeled by $w(i),w(j),w(k)$ respectively.  Since $f(u_{i})<f(u_{j})=s < f(u_{k})$, we have $i<j<k$.  Moreover, by construction, since $u_{k}$ is the rightmost child of $u_{i}$, we have $w(k)+1=w(i)$, and since $u_{j}$ comes after $u_{i}$ in LR-postfix order, we have $w(j)>w(i)$.  Therefore, $(i,j,k)$ is an $(\alpha,231)$-pattern in $w$, which is a contradiction.
	  
	We thus conclude the induction, meaning that $f$ is the left-aligned coloring of $T_{w}$, which is therefore an $\alpha$-tree.
\end{proof}

We conclude this section with the proof that $\lactoperm$ is a bijection with inverse $\permtolac$.  In a plane rooted tree, an \tdef{internal} node is a node that is neither the root nor a leaf.

\begin{theorem}\label{thm:lac_to_perm}
	For every integer composition $\alpha$, the map $\lactoperm$ is a bijection whose inverse is $\permtolac$.  Moreover, $\lactoperm$ sends internal nodes to descents.
\end{theorem}
\begin{proof}
	Given Propositions~\ref{prop:lac_to_perm_valid} and \ref{prop:perm_to_lac_valid}, it remains to prove that $\lactoperm \circ \permtolac = \operatorname{id}$ and $\permtolac \circ \lactoperm = \operatorname{id}$.  Let $\alpha=(\alpha_{1},\alpha_{2},\ldots,\alpha_{r})$.
	
	To show that $\lactoperm \circ \permtolac = \operatorname{id}$, for $w\in\Symmetric_{\alpha}(231)$, we consider $T=\permtolac(w)$ and $w' = \lactoperm(T)$. Let $C$ be the full coloring of $T$ constructed in the proof of Proposition~\ref{prop:lac_to_perm_valid}.  We have seen in that proof that $C$ is precisely the left-aligned coloring of $T$ from Construction~\ref{constr:lac_tree}.  Therefore, entries in each $\alpha$-region of $w$ and $w'$ agree, and since there is a unique way to order a set of distinct integers increasingly, we conclude that $w=w'$.
	
	Now for $\permtolac \circ \lactoperm = \operatorname{id}$, given $T\in\Trees_{\alpha}$, we consider $w=\lactoperm(T)$ and $T'=\permtolac(w)$.  Let $C,C'$ be the left-aligned colorings of $T$ and $T'$, respectively, and for $s\in[r]$ let $T_{s}$ and $T'_{s}$ denote the induced subtrees consisting of the nodes of color at most $s$ in $T$ and $T'$ respectively.  We prove by induction on $s$ that $T_{s}=T'_{s}$ holds for all $0\leq s\leq r$.  
	
	The base case $s=0$ is clear, as $T_{0}$ and $T'_{0}$ consist only of the root.  Now suppose that $T_{s-1}=T'_{s-1}$.  Let $u$ be a node of $T$ with $C(u)=s$ and label $w(i)$, and let $u'$ be the node of $T'$ with the same label.  Since elements in any $\alpha$-region are ordered increasingly, we know that every node of color $s$ inserted in $T'$ before $u'$ has a smaller label than $u'$, and by construction no such node can be the parent of $u'$.  It follows that the parent $v'$ of $u'$ belongs to $T'_{s-1}$.  By our induction hypothesis, it follows that the parent $v$ of $u$ in $T$ has the same label as $v'$.  Now, for the order of newly inserted children of a node, by Construction~\ref{constr:perm_to_lac}, in $T'$ the children of every node at each step are ordered by increasing labels, which is the same as in $T$ due to the LR-postfix order. We thus have $T_{s}=T'_{s}$, which completes the induction step. We conclude that $T=T_{r}=T'_{r}=T'$.
\end{proof}

\subsection{Noncrossing Partitions and Left-Aligned Colorable Trees}
	\label{sec:trees_to_noncrossings}
The next construction associates an $\alpha$-partition with every $\alpha$-tree.

\begin{construction}\label{constr:lac_to_nc}
	Let $\alpha$ be an integer composition.  Let $T\in\Trees_{\alpha}$ whose left-aligned coloring produced by Construction~\ref{constr:lac_tree} is $C$. We write the $\lvert\alpha\rvert$ non-root nodes of $T$ on a horizontal line, where we group them by their color under $C$.  Each group is then sorted by LR-prefix order, while groups are sorted by increasing color.  We then connect two nodes $u$ and $v$ by a bump if and only if $v$ is the rightmost child of $u$.  The diagram thus obtained is the \tdef{flattening} of $T$.  
	
	By Lemma~\ref{lem:lac_property}(i) there are no bumps between elements in the same color group, which implies that the flattening of $T$ is the diagram of some $\alpha$-partition $\mathbf{P}_{T}$.
\end{construction}

See Figure~\ref{fig:lac_to_nc} for an illustration of Construction~\ref{constr:lac_to_nc}.  We now prove that every $\alpha$-partition obtained by Construction~\ref{constr:lac_to_nc} is in fact noncrossing, which allows us to define the following map:
\begin{equation}\label{eq:lac_to_nc}
	\lactonc\colon\Trees_{\alpha}\to\NC_{\alpha},\quad T\mapsto\mathbf{P}_{T}.
\end{equation}

\begin{figure}
	\centering
	\begin{tikzpicture}
		\draw(0,0) node(lac){\lactreeBnc{0.7}};
		\draw(7,0) node(nc){\ncB{1.2}};
		\draw[->] (lac) -- (nc) node[above,midway]{$\lactonc$};
	\end{tikzpicture}
	\caption{An illustration of Construction~\ref{constr:lac_to_nc} for $\alpha=(1,3,1,2,4,3)$.}
	\label{fig:lac_to_nc}
\end{figure}
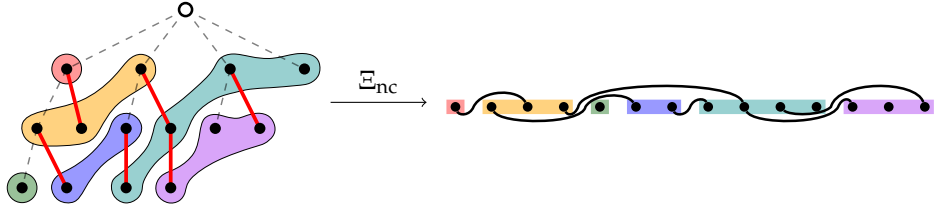

\begin{proposition}\label{prop:lac_to_nc_valid}
	For an $\alpha$-tree $T$, the $\alpha$-partition $\mathbf{P}_{T}$ obtained by Construction~\ref{constr:lac_to_nc} is noncrossing.
\end{proposition}
\begin{proof}
	We recall that by Lemma~\ref{lem:lac_property}(i), nodes with the same color in $T$ are incomparable.  Thus, $\mathbf{P}_{T}$ is indeed an $\alpha$-partition. Now suppose that $\mathbf{P}_{T}$ is not noncrossing.  There are thus two bumps $(a_{1},b_{1}),(a_{2},b_{2})$ that violate either \eqref{it:nc1} or \eqref{it:nc2}.  Let $u_{1},v_{1},u_{2},v_{2}$ be the nodes of $T$ corresponding to the elements $a_{1},b_{1},a_{2},b_{2}$, whose colors in the left-aligned coloring of $T$ are $c_{1},d_{1},c_{2},d_{2}$.
	
	If the two bumps violate \eqref{it:nc1}, then $a_{1}<a_{2}<b_{1}<b_{2}$, but $a_{1},a_{2},b_{1}$ belong to three different $\alpha$-regions.  It follows that $c_{1}<c_{2}<d_{1}$.  We conclude that $u_{2}$ precedes $v_{1}$ in LR-prefix order, since $v_{1}$ becomes active immediately after step $c_{1}$ of Construction~\ref{constr:lac_tree} in which its parent $u_{1}$ receives its color.  Moreover, since $c_{1}<c_{2}$, the node $u_{2}$ is not an ancestor of $u_{1}$.  Consequently, the right-most child of $u_{2}$, which is $v_{2}$, precedes $v_{1}$ in LR-prefix order, too.   However, since $b_{1}<b_{2}$,  we have $d_{1}\leq d_{2}$, which implies that $v_{1}$ precedes $v_{2}$ in LR-prefix order.  This is a contradiction, and we conclude that \eqref{it:nc1} is not violated.
	
	If the two bumps violate \eqref{it:nc2}, then $a_{1}<a_{2}<b_{2}<b_{1}$, and $a_{1}, a_{2}$ are in the same $\alpha$-region.  It follows that $c_{1}=c_{2}$, and from Lemma~\ref{lem:lac_property}(i) we know that $u_{1}$ and $u_{2}$ are incomparable in $T$.  By Construction~\ref{constr:lac_to_nc}, the node $u_{1}$ thus precedes $u_{2}$ in LR-prefix order.  Since $v_{1}$ is the rightmost child of $u_{1}$, it also precedes $u_{2}$ in LR-prefix order.  Since $v_{2}$ is the rightmost child of $u_{2}$, it is also preceded by $v_{1}$ in the LR-prefix order.  However, since $b_{2}<b_{1}$, we have $d_{2}\leq d_{1}$, which implies that $v_{2}$ precedes $v_{1}$ in LR-prefix order.  This is a contradiction, and we conclude that \eqref{it:nc2} is not violated, either. 
	
	Hence, $\mathbf{P}_{T}$ must be noncrossing.
\end{proof}

Let us now describe how to obtain a plane rooted tree from a noncrossing $\alpha$-partition.  

\begin{construction}\label{constr:nc_to_lac}
	For $\alpha$ a composition of $n$,  let $\mathbf{P}$ be a noncrossing $\alpha$-partition, and we label the elements of $\mathbf{P}$ by $1,2,\ldots,n$ from left to right.
	
	We start with a collection of $n+1$ nodes $v_{0},v_{1},\ldots,v_{n}$.  If the element $i$ does not lie below any bump of $\mathbf{P}$, then we add an edge from $v_{0}$ to $v_{i}$, making $v_{i}$ a child of $v_{0}$.  If $i$ lies directly below a bump $(a,b)$, then we add an edge from $v_{a}$ to $v_{i}$ in the same way.
	
	Since we associate with each element of $\mathbf{P}$ a parent with strictly smaller index, the resulting graph does not have cycles.  Moreover, it has $n+1$ nodes and $n$ edges, and therefore is a tree.  Finally, we order the children of each node by their label in increasing order, and root the tree at $v_{0}$.  We thus obtain a plane rooted tree $T_{\mathbf{P}}$.
\end{construction}

We now prove that the plane rooted tree $T_\mathbf{P}$ obtained from Construction~\ref{constr:nc_to_lac} is compatible with $\alpha$ so that we obtain the following map:
\begin{equation}\label{eq:nc_to_lac}
	\nctolac\colon\NC_{\alpha}\to\Trees_{\alpha},\quad\mathbf{P}\mapsto T_{\mathbf{P}}.
\end{equation}

\begin{proposition} \label{prop:nc_to_lac_valid}
	For a noncrossing $\alpha$-partition $\mathbf{P}$, the plane rooted tree $T_{\mathbf{P}}$ is compatible with $\alpha$.  
\end{proposition}
\begin{proof}
	Let $\alpha=(\alpha_{1},\alpha_{2},\ldots,\alpha_{r})$ with $\lvert\alpha\rvert=n$, and suppose that the nodes of $T_{\mathbf{P}}$ are $v_{0},v_{1},\ldots,v_{n}$ as in Construction~\ref{constr:nc_to_lac}.  We define a full coloring $f$ of $T_{\mathbf{P}}$, where $f(v_{a})=i$ if and only if $a$ belongs to the $i$-th $\alpha$-region in $\mathbf{P}$.  We now prove that Construction~\ref{constr:lac_tree} does not fail on $T_{\mathbf{P}}$ for $\alpha$, and $f$ agrees with the coloring of $T_{\mathbf{P}}$ thus obtained. We proceed by induction on the number $s$ of steps.  
	
	The elements in the first $\alpha$-region do not lie below any bump of $\mathbf{P}$, which means that Construction~\ref{constr:lac_tree} does not fail in the first step.  Moreover, by Construction~\ref{constr:nc_to_lac} the nodes $v_{1},v_{2},\ldots,v_{\alpha_{1}}$ are the leftmost children of the root of $T_{\mathbf{P}}$, and the partial labeling of Construction~\ref{constr:lac_tree} therefore assigns the color $1$ to all these nodes.  Hence, this partial labeling agrees with $f$ on the nodes colored by $1$.  
	This establishes the base case $s=1$ of our induction.

	Now suppose that the induction hypothesis holds of $s-1<r$, and we now consider step $s$. For $a\in\{\alpha_{s-1}+1,\alpha_{s-1}+2,\ldots,\alpha_{s}\}$, let $w$ be the parent of $v_{a}$.  By Construction~\ref{constr:nc_to_lac}, $w$ is either the root, or $w=v_{a'}$ for some $a'\leq\alpha_{s-1}$.  This implies $f(w)\leq s-1$.  Thus, $v_{a}$ is active in step $s$ of Construction~\ref{constr:lac_tree}.  Since $a$ was chosen arbitrarily, we conclude that all vertices of $T_{\mathbf{P}}$ that come from the $s$-th $\alpha$-region are active in step $s$ of Construction~\ref{constr:lac_tree}.  Consequently, Construction~\ref{constr:lac_tree} does not fail in step $s$.
	
	We prove now that the nodes corresponding to elements in the $s$-th $\alpha$-region are the first active nodes in LR-prefix order at step $s$.  Suppose that at step $s$ there are two active nodes $v_{b_{1}},v_{b_{2}}$ such that $v_{b_{1}}$ precedes $v_{b_{2}}$ in LR-prefix order, and $f(v_{b_{1}})>s=f(v_{b_{2}})$.  Consequently, $b_{2}<b_{1}$, and both nodes belong to different $\alpha$-regions.  Let $v_{a_{1}},v_{a_{2}}$ denote the parents of $v_{b_{1}},v_{b_{2}}$, respectively.  If $v_{a_{1}}=v_{a_{2}}$, then by construction we must have $b_{1}<b_{2}$, which is a contradiction.  If $v_{a_{1}}$ is the root, then $v_{b_{1}}$ also precedes $v_{a_{2}}$ in LR-prefix order, and it is either colored or active in the step in which $v_{a_{2}}$ receives its color.  By our induction hypothesis we have $s<f(v_{b_{1}})\leq f(v_{a_{2}}) < s$, which is a contradiction.  If $v_{a_{2}}$ is the root, then since $b_{2}$ is not below any bump, the region of $b_{2}$ must not come later than that of $a_{1}$ (they can be in the same region). In combination with the induction hypothesis, we have $s=f(v_{b_{2}})\leq f(v_{a_{1}})<s$, which is a contradiction. We conclude that $b_{1}$ and $b_{2}$ are directly below bumps $(a_{1},b'_{1})$ and $(a_{2},b'_{2})$ respectively.  
	
	Since $v_{b_{1}},v_{b_{2}}$ are active in the $s$-th step, it follows that $a_{1},a_{2}$ are in $\alpha$-regions before the $s$-th $\alpha$-region, with $b_{1}<b'_{1}$ and $b_{2}<b'_{2}$.  There are four possibilities.
	
	(i) If $a_{1}<a_{2}<b'_{1}<b'_{2}$, then either $a_{1}$ and $a_{2}$ are in the same $\alpha$-region, or $a_{2}$ and $b'_{1}$ are in the same $\alpha$-region by \eqref{it:nc1}.  In both cases, we would have $b_{1}<b_{2}$, which is a contradiction.
	
	(ii) If $a_{1}<a_{2}<b'_{2}<b'_{1}$, then  $a_{1}$ and $a_{2}$ belong to different $\alpha$-regions by \eqref{it:nc2}. Thus, $v_{a_{1}}$ is colored before $v_{a_{2}}$ is, which implies that $v_{b_{1}}$ is active before $v_{a_{2}}$ receives its color.  Since $v_{b_{1}}$ precedes $v_{b_{2}}$ in LR-prefix order and $v_{b_{1}}$ is not a descendant of $v_{a_{2}}$, we conclude that $v_{b_{1}}$ precedes $v_{a_{2}}$ in LR-prefix order, too.  It follows that $v_{b_{1}}$ receives its color not later than $v_{a_{2}}$, and by our induction hypothesis, we have $f(v_{b_{1}})\leq f(v_{a_{2}})<s$.  This contradicts the assumption $f(v_{b_{1}}) > s$.	
	
	(iii) If $a_{2}<a_{1}<b'_{1}<b'_{2}$, then $a_{1}$ and $a_{2}$ belong to different $\alpha$-regions by \eqref{it:nc2}. This means that $(a_{1},b'_{1})$ separates $b_{1}$ from $(a_{2},b'_{2})$. Since $f(v_{a_1}) < s = f(v_{b_2})$, we have $a_1 < b_2$, thus $b_1 < b_1' < b_2$ by the separation, which is a contradiction.

	(iv) If $a_{2}<a_{1}<b'_{2}<b'_{1}$, then either $a_{1}$ and $a_{2}$ belong to the same $\alpha$-region, or $a_{1}$ and $b'_{2}$ belong to the same $\alpha$-region by \eqref{it:nc1}.  If $a_{1}$ and $a_{2}$ belong to the same $\alpha$-region, then $v_{a_{2}}$ precedes $v_{a_{1}}$ in LR-prefix order (since $a_{2}<a_{1}$).  Moreover, since $f(v_{a_{1}})=f(v_{a_{2}})\leq s-1$, by our induction hypothesis these two nodes receive the same color in Construction~\ref{constr:lac_tree}, and by Lemma~\ref{lem:lac_property}(i) they are thus incomparable in $T_{\mathbf{P}}$. 
	Since $v_{b_{2}}$ is a child of $v_{a_{2}}$, it follows that $v_{b_{2}}$ precedes $v_{b_{1}}$ in LR-prefix order, which is a contradiction.  If $a_{1}$ and $b'_{2}$ belong to the same $\alpha$-region, then necessarily $b_{2}$ belongs to this $\alpha$-region, too.  By induction hypothesis we have $s>f(v_{a_{1}})=f(v_{b_{2}})=s$, which is a contradiction.
	
	\medskip
	
	We therefore conclude that all the nodes of $T_{\mathbf{P}}$ from the $s$-th $\alpha$-region receive color $s$ in step $s$ of Construction~\ref{constr:lac_tree}, and we finish our induction step.  We conclude that $T_{\mathbf{P}}$ is an $\alpha$-tree, and $f$ is its left-aligned coloring.   
\end{proof}

We conclude this section with the proof that $\lactonc$ is a bijection with inverse $\nctolac$.

\begin{theorem}\label{thm:lac_to_nc}
	For every integer composition $\alpha$, the map $\lactonc$ is a bijection whose inverse is $\nctolac$.  Moreover, $\lactonc$ sends internal nodes to bumps.
\end{theorem}
\begin{proof}
	In view of Propositions~\ref{prop:lac_to_nc_valid}~and~\ref{prop:nc_to_lac_valid}, it remains to prove that $\lactonc \circ \nctolac = \operatorname{id}$ and $\nctolac \circ \lactonc = \operatorname{id}$.

	\medskip

	To show that $\lactonc \circ \nctolac = \operatorname{id}$, for $\mathbf{P} \in \NC_\alpha$, let $T = \nctolac(\mathbf{P})$ be the $\alpha$-tree obtained from $\mathbf{P}$ by Construction~\ref{constr:nc_to_lac}, and $\mathbf{P}'=\lactonc(T)$ the noncrossing $\alpha$-partition obtained from $T$ by Construction~\ref{constr:lac_to_nc}.  
	
	Since every bump of $\mathbf{P}$ corresponds to an internal node of $T$, and every internal node of $T$ corresponds to a bump of $\mathbf{P}'$, we conclude that $\mathbf{P}$ and $\mathbf{P}'$ have the same number of bumps.  Let $(a,b)$ be a bump of $\mathbf{P}$. The node $v_a$ is thus the parent of $v_b$. Let $i$ be an element of $\mathbf{P}$ such that $v_i$ is a child of $v_a$ in $T$. By Construction~\ref{constr:nc_to_lac}, we have $a < i \leq b$. Hence, $v_{i}$ weakly precedes $v_{b}$ in LR-prefix order. This is valid for any child of $v_a$ in $T$, implying that $v_b$ is the rightmost child of $v_a$ in $T$, leading to a bump $(a,b)$ in $\mathbf{P}'$ by Construction~\ref{constr:lac_to_nc}. We thus conclude that $\mathbf{P}=\mathbf{P}'$.

	\medskip

	To show that $\nctolac \circ \lactonc = \operatorname{id}$, for $T\in\Trees_{\alpha}$, let $C$ be the left-aligned coloring of $T$, $\mathbf{P}=\lactonc(T)$ and $T'=\nctolac(\mathbf{P})$.  We now prove $T=T'$. For $u$ a non-leaf node of $T$ and $v$ a child of $u$, 
	let $a,b$ be the elements of $\mathbf{P}$ corresponding to $u,v$, and let $u',v'$ be the nodes of $T'$ corresponding to $a,b$.  We first show that $v'$ is a child of $u'$ in $T'$.  
	
	If $u$ is the root, then $b$ is not below any bump in $\mathbf{P}$, meaning that $u'$ is the root of $T'$ and $v'$ one of its children by Construction~\ref{constr:nc_to_lac}. We now suppose that $u$ is not the root.  If $v$ is the rightmost child of $u$, then $(a,b)$ is a bump of $\mathbf{P}$, and it follows from Construction~\ref{constr:nc_to_lac} that $v'$ is the rightmost child of $u'$, too.  
	
	It remains to consider the case that $v$ is not the rightmost child of $u$.  Let $v_{r}$ be the rightmost child of $u$, and let $b_{r}$ be the element of $\mathbf{P}$ corresponding to $v_{r}$.  It follows that $(a,b_{r})$ is a bump of $\mathbf{P}$.  Moreover, by Construction~\ref{constr:lac_tree} we have $C(u)<C(v)\leq C(v_{r})$, and Construction~\ref{constr:lac_to_nc} then implies that $a<b<b_{r}$.  Hence, $b$ lies below some bump, which implies that $u'$ is not the root of $T'$.  Moreover, $b$ lies directly below a certain bump $(a_{*},b_{*})$.  
	
	Suppose that $a_{*}\neq a$.  Let $u_{*}$ and $v_{*}$ denote the nodes of $T$ that correspond to $a_{*}$ and $b_{*}$.  In particular, $a_{*}$ and $b$ belong to different $\alpha$-regions, which implies $C(u_{*})<C(v)$. Since $b<b_{*}$, we have $C(v)\leq C(v_{*})$.  By Construction~\ref{constr:lac_to_nc}, the node $v_{*}$ is the rightmost child of $u_{*}$, and we conclude that $b\leq b_{*}<b_{r}$.  There are two cases.
	
	(i) If $a<a_{*}<b_{*}<b_{r}$, then $a$ and $a_{*}$ are in different $\alpha$-regions by \eqref{it:nc2}.  Thus, $C(u)<C(u_{*})$, which means that $v$ is active in the step in which $u_{*}$ receives its color in Construction~\ref{constr:lac_tree}.  Since $C(u_{*})<C(v)$, it follows that $u_{*}$ precedes $v$ in LR-prefix order.  Consequently, $v_{*}$ precedes $v$ in LR-prefix order, too, as $u_*$ and $v_*$ are consecutive in LR-prefix order.  But this implies $b_{*}<b$, which is a contradiction.  
	
	(ii) If $a_{*}<a<b_{*}<b_{r}$, then either $a$ and $a_{*}$ belong to the same $\alpha$-region, or $a$ and $b_{*}$ belong to the same $\alpha$-region by \eqref{it:nc1}.  If $a$ and $a_{*}$ belong to the same $\alpha$-region, then it follows that $C(u)=C(u_{*})$ and that $u_{*}$ precedes $u$ in LR-prefix order.  Hence, $v$ and $v_{*}$ become active in the same iteration of Construction~\ref{constr:lac_tree} and $v_{*}$ precedes $v$ in LR-prefix order, too.  Thus, we have $C(v_{*})\leq C(v)$, which by assumption yields $C(v_{*})=C(v)$.  It follows that $b$ and $b_{*}$ belong to the same $\alpha$-region, and by Construction~\ref{constr:lac_to_nc} we obtain that $v$ precedes $v_{*}$ in LR-prefix order, which is a contradiction.  If $a$ and $b_{*}$ belong to the same $\alpha$-region, 
	then combining $b<b_*$ with the case condition, we have $a<b<b_{*}$, and it follows that $a$ and $b$ also belong to the same $\alpha$-region.  This implies $C(u)=C(v)$, which contradicts Lemma~\ref{lem:lac_property}(i), since $v$ is a child of $u$. 
	 
	 We conclude that $a_{*}=a$, which implies that $b$ lies directly below the bump $(a,b_{r})$.  We observe that the order of the children a node in $T$ and $T'$ depends only on the order of the elements of $\mathbf{P}$ by Constructions~\ref{constr:lac_to_nc} and \ref{constr:nc_to_lac}.  
	 We thus conclude that $T=T'$.
\end{proof}

\subsection{Dyck Paths and Left-Aligned Colorable Trees}
	\label{sec:trees_to_paths}
The following construction associates an $\alpha$-Dyck path with every $\alpha$-tree.
	
\begin{construction}\label{constr:lac_to_path}
	Given $\alpha=(\alpha_{1},\alpha_{2},\ldots,\alpha_{r})$ an integer composition with $\lvert\alpha\rvert=n$, let $T\in\Trees_{\alpha}$.  For $k\in[r]$, let $s_{k}=\alpha_{1}+\alpha_{2}+\cdots+\alpha_{k}$ as defined at the beginning of Section~\ref{sec:members_parabolic_cataland}, which is precisely the number of nodes of color $\leq k$.  Moreover, let $a_{k}$ denote the number of active nodes in step $(k+1)$ of Construction~\ref{constr:lac_tree} before coloring, which is precisely the number of nodes of color $>k$ whose parents have color $\leq k$.  It is immediate that $s_{k}+a_{k}\leq n$ for all possible $k$.  
	
	We now construct a Dyck path which will have exactly one valley for each internal node of $T$.  Suppose that the rightmost child of the $i$-th internal node of color $k$ is the $j$-th active node in the $(k+1)$-st step of Construction~\ref{constr:lac_tree} before coloring.  We mark the coordinate $(p,q)$ in the plane, where $p=s_{k}-i+1$ and $q=s_{k}+a_{k}-j$. We then return the unique Dyck path $\mu_{T}$ from $(0,0)$ to $(n,n)$ that has valleys precisely at the marked coordinates. 
\end{construction}

\begin{remark}
There is a simpler way to informally describe Construction~\ref{constr:lac_to_path}. We first draw the bounce path $\nu_\alpha$, which encloses triangular regions with the main diagonal. Each such region corresponds to one color. For each internal node $u$ of color $k$, with $v$ its rightmost child, suppose that $u$ is the $i$-th node \emph{from right to left} of color $k$, and $v$ the $j$-th active node \emph{from right to left} at step $(k+1)$ of Construction~\ref{constr:lac_tree}. We then add a dot on the $i$-th column from left to right intersecting the triangular region of color $k$, in the $j$-th cell above the region. These dots indicate the valleys of $\mu_T$. The description here is equivalent to Construction~\ref{constr:lac_to_path}.
\end{remark}

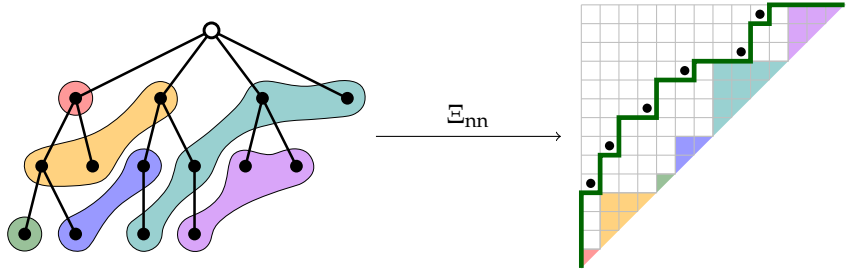
\begin{figure}
	\centering
	\begin{tikzpicture}
	\draw(0,0) node(lac){\lactreeB{0.8}};
	\draw(7,0) node(nn){\pathBReversed{1}};
	\draw[->] (lac) -- (nn) node[above,midway]{$\lactonn$};
	\end{tikzpicture}
	\caption{An illustration of Construction~\ref{constr:lac_to_path} for $\alpha=(1,3,1,2,4,3)$.}
	\label{fig:lac_to_nn}
\end{figure}

See Figure~\ref{fig:lac_to_nn} for an illustration of Construction~\ref{constr:lac_to_path}.  We now prove that every Dyck path obtained by Construction~\ref{constr:lac_to_path} is in fact an $\alpha$-Dyck path, so that we obtain the following map:
\begin{equation}\label{eq:lac_to_nn}
	\lactonn\colon\Trees_{\alpha}\to\Dyck_{\alpha},\quad T\mapsto\mu_{T}.
\end{equation}

\begin{proposition}\label{prop:lac_to_nn_valid}
	For an $\alpha$-tree $T$, the Dyck path $\mu_{T}$ obtained by Construction~\ref{constr:lac_to_path} is an $\alpha$-Dyck path.
\end{proposition}
\begin{proof}
	Let $\alpha=(\alpha_{1},\alpha_{2},\ldots,\alpha_{r})$ be a composition of $n$.  Recall from Section~\ref{sec:parabolic_dycks} that the $\alpha$-bounce path is $\nu_{\alpha}=N^{\alpha_{1}}E^{\alpha_{1}}N^{\alpha_{2}}E^{\alpha_{2}}\ldots N^{\alpha_{r}}E^{\alpha_{r}}$.  We need to show that $\mu_{T}$ is weakly above $\nu_{\alpha}$, or equivalently, that all valleys of $\mu_{T}$ lie weakly above $\nu_{\alpha}$.
	
	Let $u$ be the $i$-th node of color $k$ in $T$ with respect to LR-prefix order.  If $u$ is a leaf, then it does not contribute a valley of $\mu_{T}$, and thus need not be considered.   Let $v$ be the rightmost child of $u$. Suppose that $v$ is the $j$-th active node in the $k$-th step of Construction~\ref{constr:lac_tree}.  
	
	By construction, the corresponding valley of $\mu_{T}$ has coordinate $(p,q)$ with $p=s_{k}-i+1$ and $q=s_{k}+a_{k}-j$, where $a_{k}$ is the number of active nodes in the ${(k+1)}$-st step of Construction~\ref{constr:lac_tree} before coloring.  By construction we have $i\leq\alpha_{k}$, therefore $s_{k-1}<p\leq s_{k}$.  Moreover, since $j\leq a_{k}$, we have $s_{k}\leq q<s_{k}+a_{k}$. 
	Therefore, $(p,q)$ is weakly above $\nu_\alpha$, with which we conclude the proof.  
	
\end{proof}

Let us now describe how to obtain a colored tree from an $\alpha$-Dyck path.

\begin{construction}\label{constr:path_to_lac}
	Given $\alpha=(\alpha_{1},\alpha_{2},\ldots,\alpha_{r})$ an integer composition, let $\mu\in\Dyck_{\alpha}$. We initialize our algorithm with the tree $T_{0}$ which consists of a single root node $v_{0}$ (implicitly colored by color $0$), and the empty coloring $C_{0}$.  
	
	At the $k$-th step, we construct a (partially) colored tree $(T_{k},C_{k})$ from $(T_{k-1},C_{k-1})$ by adding as many children to the $i$-th node of color $k-1$ as there are north-steps in $\mu$ with $x$-coordinate equal to $s_{k-1}-i+1$.  After all childen have been added, we extend $C_{k-1}$ to $C_{k}$ by assigning the value $k$ to the first $\alpha_{k}$ uncolored nodes in LR-prefix order. The algorithm fails if there are not enough uncolored nodes. 
	After $r$ steps, we obtain a colored tree $(T_{r},C_{r})$, and we return the plane rooted tree $T_{\mu}=T_{r}$.
\end{construction}

We now prove that the plane rooted tree obtained in Construction~\ref{constr:path_to_lac} is indeed compatible with $\alpha$, which allows for the definition of a map
\begin{equation}\label{eq:nn_to_lac}
	\nntolac\colon\Dyck_{\alpha}\to\Trees_{\alpha},\quad\mu\mapsto T_{\mu}.
\end{equation}

\begin{proposition}\label{prop:nn_to_lac_valid}
	For an $\alpha$-Dyck path $\mu$, the plane rooted tree $T_{\mu}$ obtained from Construction~\ref{constr:path_to_lac} is compatible with $\alpha$.
\end{proposition}
\begin{proof}
	Let $\alpha=(\alpha_{1},\alpha_{2},\ldots,\alpha_{r})$.  We first prove that we never run out of uncolored nodes in Construction~\ref{constr:path_to_lac}.  To that end, let $s_{0}=0$.  For $i\in\{0,1,\ldots,r-1\}$ we denote by $m_i$ the maximal $y$-coordinate of a lattice point on $\mu$ that has $x$-coordinate $s_{i}$.  Since $\mu$ is weakly above $\nu_{\alpha}$, we have $m_{i}\geq s_{i+1}$.  
	
	We now prove by induction on $k\in[r]$ that the number of uncolored non-root nodes at the end of step $k$ in Construction~\ref{constr:path_to_lac} is precisely $m_{k-1}-s_{k}$.  The base case $k=1$ holds, because in the first step, we add exactly $m_{0}$ children to the root, and we color $\alpha_{1}=s_{1}$ of them. Now suppose that the claim holds up to the $k$-th step of Construction~\ref{constr:path_to_lac}.  In particular, at the beginning of the $(k+1)$-st step we already have $m_{k-1}-s_{k}$ uncolored nodes.  In the $(k+1)$-st step, we add $m_{k}-m_{k-1}$ new nodes, and color $\alpha_{k+1}$ of the uncolored nodes.  This leaves us with $m_{k-1}-s_{k}+m_{k}-m_{k-1}-\alpha_{k+1}=m_{k}-s_{k+1}$ uncolored nodes, and we conclude the induction.
	
	\medskip
	
	Now we prove that $T_\mu$ is an $\alpha$-tree. Let $f$ be the full coloring of $T_{\mu}$ obtained in Construction~\ref{constr:path_to_lac}. We show that $f$ is the left-aligned coloring of $T_{\mu}$. Let $f_{k}$ denote the restriction of $f$ to the nodes of $T_{\mu}$ that have $f$-color at most $k$.  We now prove that Construction~\ref{constr:lac_tree} does not failed up to step $k$ and that $f_{k}$ agrees with $C_{k}$ defined in Construction~\ref{constr:lac_tree}, once again by induction on $k\in[r]$.
	
	For the base case $k=1$, by construction, the root of $T_{\mu}$ has $m_{0}\geq\alpha_{1}$ children, and we have $f(v)=1$ if and only if $v$ belongs to the $\alpha_{1}$ leftmost children of the root.  We see that Construction~\ref{constr:lac_tree} does not fail in the first step, and the same nodes receive color $1$ in both constructions. 
	
	Suppose that the induction hypothesis holds for $k$. The active nodes with respect to $C_{k}$ are thus precisely those whose parents have $C_{k}$-color at most $k$.  In step $k+1$ of Construction~\ref{constr:path_to_lac}, we add precisely $m_{k}-m_{k-1}$ new nodes, and after this addition, the parents of the uncolored nodes have $f_{k}$-color at most $k$.  By induction hypothesis, the uncolored nodes with respect to $f_{k}$ are precisely the active nodes with respect to $C_{k}$.  In view of the first part of this proof, Construction~\ref{constr:lac_tree} does not fail in step $k+1$, and since we color the $\alpha_{k+1}$ leftmost nodes in LR-prefix order by $k+1$ in both constructions, the induction step is established.  
	
	We conclude that $T_{\mu}$ is compatible with $\alpha$ and the coloring from Construction~\ref{constr:path_to_lac} is precisely the left-aligned coloring of $T_{\mu}$.
\end{proof}

\begin{remark}\label{rem:lac_tree_nb_actives}
	It follows from the proof of Proposition~\ref{prop:nn_to_lac_valid} that after we have created all the nodes in the $(k+1)$-st step of Construction~\ref{constr:lac_tree}, there are $m_{k}-s_{k}$ 
	active nodes before coloring, and $m_{k}-s_{k+1}$ active nodes after coloring.
\end{remark}

\begin{lemma}\label{lem:nb_children}
	Given $T\in\Trees_{\alpha}$, let $u$ be an internal node in $T$ and $\vec{p}=(p,q)$ the valley of $\mu = \lactonn(T)$ corresponding to $u$.  The number of children of $u$ is equal to that of consecutive north-steps in $\lactonn(T)$ immediately after $\vec{p}$.
\end{lemma}
\begin{proof}
	Let $c_1$ be the number of children of $u$ in $T$, and $c_2$ the number of consecutive north-steps in $\lactonn(T)$ immediately after $\vec{p}$.  Suppose that $u$ has color $k$, and that the rightmost child of $u$ is the $j$-th active node in LR-prefix order in the $(k+1)$-st step of Construction~\ref{constr:lac_tree} before coloring among a total of $a_k$ active nodes.  Moreover, let $m_{k}$ be the maximal $y$-coordinate of a lattice point on $\mu$ with $x$-coordinate $s_{k}$.  It follows from Remark~\ref{rem:lac_tree_nb_actives} that $a_{k}=m_{k}-s_{k}$.  By Construction~\ref{constr:lac_to_path} we have $q=s_{k}+a_{k}-j=m_{k}-j$.
	
	If $u$ is the first internal node in its color group, then $c_{1}=j$.  By construction, $\vec{p}$ is the last valley whose $x$-coordinate is at most $s_{k}$, which implies that $c_{2}=m_{k}-q=j$.
	
	Otherwise let $u_{*}$ be the internal node of color $k$ that immediately precedes $u$ in LR-prefix order. 
	Suppose that the rightmost child of $u_{*}$ is the $j_{*}$-th active node in LR-prefix order at step $k+1$ of Construction~\ref{constr:lac_tree} before coloring.  It follows that $c_{1}=j-j_{*}$.  Moreover, let $\vec{p}_{*}=(p_{*},q_{*})$ be the valley of $\mu$ corresponding to $u_{*}$.  We know that $\vec{p}_{*}$ is the valley immediately after $\vec{p}$, and we have $c_{2}=q_{*}-q$.  Since $q_{*}=s_{k}+a_{k}-j_{*}=m_{k}-j_{*}$ by Construction~\ref{constr:lac_to_path}, we conclude that $c_{2}=j-j_{*}$.  We thus have $c_1=c_2$ in all cases.
\end{proof}

We now prove that $\lactonn$ is a bijection with inverse $\nntolac$.

\begin{theorem}\label{thm:lac_to_nn}
	For every integer composition $\alpha=(\alpha_{1},\alpha_{2},\ldots,\alpha_{r})$, the map $\lactonn$ is a bijection with inverse $\nntolac$.  Moreover, $\lactonn$ sends internal nodes to valleys.
\end{theorem}
\begin{proof}
	In view of Propositions~\ref{prop:lac_to_nn_valid}~and~\ref{prop:nn_to_lac_valid}, it remains to prove that $\lactonn \circ \nntolac = \operatorname{id}$ and $\nntolac \circ \lactonn = \operatorname{id}$.

	To show that $\lactonn \circ \nntolac = \operatorname{id}$, for $\mu\in\Dyck_{\alpha}$, let $T=\nntolac(\mu)$ be the $\alpha$-tree obtained from $\mu$ by Construction~\ref{constr:path_to_lac}, and $\mu'=\lactonn(T)$ the $\alpha$-Dyck path obtained from $T$ by Construction~\ref{constr:lac_to_path}.  Since every valley of $\mu$ corresponds to an internal node of $T$, which in turn corresponds to a valley of $\mu'$, we conclude that $\mu$ and $\mu'$ have the same number of valleys.  
	
	Let $(p,q)$ be a valley of $\mu$, and $u$ the corresponding internal node of $T$, which is the $(s_{k}-p+1)$-st node of color $k$ in the LR-prefix order of $T$. 
	Let $m_k$ be the maximal $y$-coordinate of a lattice point on $\mu$ with $x$-coordinate $s_{k}$.  By Remark~\ref{rem:lac_tree_nb_actives}, the number of active nodes in the $(k+1)$-st step of Construction~\ref{constr:lac_tree} (before coloring) is precisely $a_{k}=m_{k}-s_{k}$.  The rightmost child of $u$ is thus the $(m_{k}-q)$-th active node at step $k$ in the LR-prefix order of $T$.  By Construction~\ref{constr:path_to_lac}, the internal node $u$ and its rightmost child contribute a valley $(p',q')$ to $\mu'$, where $p'=s_{k}-(s_{k}-p+1)+1=p$ and $q'=s_{k}+a_{k}-(m_{k}-q)=q$.  Thus, the valleys of $\mu$ and $\mu'$ agree, and it follows that $\mu=\mu'$.
	
	\medskip
	
	To show that $\nntolac \circ \lactonn = \operatorname{id}$, for $T\in\Trees_{\alpha}$, let $\mu=\lactonn(T)$ be the $\alpha$-Dyck path obtained from $T$ by Construction~\ref{constr:lac_to_path}, and $T'=\nntolac(\mu)$ the $\alpha$-tree obtained from $\mu$ by Construction~\ref{constr:path_to_lac}.  For $s\in[r]$, let $T_{s},T'_{s}$ denote the induced subtrees of $T,T'$, respectively, that consist of the nodes of color at most $s$.  We prove by induction on $s$ that $T_{s}=T'_{s}$ holds for all $0\leq s\leq r$.
	
	The base case $s=0$ is trivially true, since both trees $T_{0}$ and $T'_{0}$ consist only of a root node.  Suppose that $T_{s-1}=T'_{s-1}$. 
	Suppose that $u$ is a node of $T_{s-1}$ that has $k>0$ children in $T$. 
	By Lemma~\ref{lem:nb_children}, either $\mu$ starts with $k$ north-steps (if $u$ is the root), or the valley corresponding to $u$ is followed by $k$ north-steps.  By Construction~\ref{constr:path_to_lac}, it follows that $u'$ also has $k$ children. 
	We also notice that, by construction, the positions of $u$ and $u'$ among all active nodes in $T_{s-1}$ and $T'_{s-1}$ are the same. Therefore, we have $T_s=T'_s$.
\end{proof}

We may now prove our first main result.
\bijections*
\begin{proof}
	The bijections are established in Theorems~\ref{thm:lac_to_perm}, \ref{thm:lac_to_nc}, and \ref{thm:lac_to_nn}.
\end{proof}

We conclude this section by showing that the bijections $\Theta_{1}$ and $\Theta_{2}$ from Sections~\ref{sec:noncrossings_to_paths} and \ref{sec:permutations_to_noncrossings}, respectively, can be recovered using the new bijections from Sections~\ref{sec:trees_to_permutations}--\ref{sec:trees_to_paths}.

\begin{proposition}\label{prop:theta_1_equiv}
	For every $\mu\in\Dyck_{\alpha}$, we have $\Theta_{1}(\overline{\mu}) = \lactonc\circ\nntolac(\mu)$.
\end{proposition}
\begin{proof}
	Let $T=\nntolac(\mu)$, $\mathbf{P}=\lactonc(T)$, and $\mathbf{P}'=\Theta_{1}(\overline{\mu})$.  By Theorem~\ref{thm:lac_to_nn}, we know that every valley of $\mu$ corresponds to an internal node $u$ of $T$, with its coordinates dictating the exact position of $u$ and its rightmost child $v$ in the $\alpha$-regions.  By Theorem~\ref{thm:lac_to_nc}, each such pair $(u,v)$ corresponds to a bump $(a,b)$ of $\mathbf{P}$, where $a$ and $b$ are determined by the positions of $u$ and $v$.  If we now take the corresponding valley of $\overline{\mu}$, and apply Construction~\ref{constr:path_to_nc}, then we see that $(a,b)$ is also a bump of $\mathbf{P}'$.  We conclude that $\mathbf{P}=\mathbf{P}'$.
\end{proof}

\begin{proposition}\label{prop:theta_2_equiv}
	For every $\mathbf{P}\in\NC_{\alpha}$, we have $\Theta_{2}(\mathbf{P}) = \lactoperm\circ\nctolac(\mathbf{P})$.
\end{proposition}
\begin{proof}
	Let $T=\nctolac(\mathbf{P})$, $w=\lactoperm(T)$, and $w'=\Theta_{2}(\mathbf{P})$.  We prove by induction on $n=\lvert\alpha\rvert$ that $w=w'$.  The base case $n=0$ clearly holds. 
	Now suppose that our claim holds for all $\alpha'$ with $\lvert\alpha'\rvert<n$.  

	We take the block of $\mathbf{P}$ containing $1$, and denote it by $P_{\circ}=\{i_{1},i_{2},\ldots,i_{s}\}$ with $i_{1}<i_{2}<\cdots<i_{s}$ and $i_{1}=1$.  By Construction~\ref{constr:nc_to_perm}, we have $w'(i_{1})=w'_{i_{2}}+1=\cdots=w'_{i_{s}}+s-1$.  The remaining values of $w'$ are determined inductively from two parabolic noncrossing partitions $\mathbf{P}_{1}$ and $\mathbf{P}_{2}$ with strictly less than $n$ elements each.  By induction hypothesis we have $\Theta_2(\mathbf{P}_{i}) = \lactoperm\circ\nctolac(\mathbf{P}_{i})$ for $i\in\{1,2\}$.

	On the other hand, by Construction~\ref{constr:nc_to_lac}, the elements in $P_{\circ}$ correspond to nodes in $T$ on the rightmost branch of the first child $u$ of the root, that is, the set of nodes visited from $u$ by always moving to the rightmost child when possible. 
	The label of $u$ in Construction~\ref{constr:lac_to_perm} is always $w(1)$, and the label of the rightmost child of any node $v$ is always one less than the label of $v$ itself.  Finally, by comparing Constructions~\ref{constr:nc_to_perm} and \ref{constr:nc_to_lac} we see that the elements contributing to $w'(1)$ are precisely the nodes in the subtree of $u$, and we conclude that $w$ and $w'$ agree on $P_{\circ}$.  

	Now let $T_{2}$ be the induced subtree of $T$ containing all the elements that are descendants of all but the first root children, and let $w_{2}=\lactoperm(T_{2})$.  
	Moreover, we construct a new tree $T_1$ from $T$ by first deleting the root and all nodes of $T_{2}$, and then contracting the rightmost branch of $u$ into a new root node while keeping the LR-prefix order of the remaining nodes.
	Let $w_{1}=\lactoperm(T_{1})$.  
	The elements of $T_{1}$ and $T_{2}$, respectively, are in the same relative order as in $T$ with respect to LR-postfix order. 
	Hence, $w$ can be reconstructed from the values $w(i_{1}),w(i_{2}),\ldots,w(i_{s})$, the values in $w_{1}$ and the values in $w_{2}$ increased by $w(1)$.  This is precisely the way that $w'$ was created, so we conclude that $w=w'$.
\end{proof}

\part{Tamari Lattices in Parabolic Cataland}\label{part:tamari}

In the middle part of this article, we equip each of the sets $\Symmetric_{\alpha}(231)$ and $\Dyck_{\alpha}$ with a partial order, and show that the resulting posets are isomorphic (Theorem~\ref{thm:parabolic_nu_tamari_isomorphism}). Both posets are in fact lattices and generalize the well-known \tdef{Tamari lattice} $\Tamari_{n}$ introduced in \cites{tamari51monoides,tamari62algebra}.  See \cite{hoissen12associahedra} for a recent survey on topics related to the Tamari lattices.  

For the proof of this isomorphism, 
we need some lattice-theoretic background, which is only relevant for the current section.  We have thus decided to put it in Appendix~\ref{sec:lattice_theory}, which can be consulted whenever necessary.

\section{Parabolic Tamari Lattices}
	\label{sec:parabolic_tamari_lattices}
For every permutation $w\in\Symmetric_{n}$ we define its \tdef{inversion set} by
\begin{displaymath}
	\inv(w) \defs \bigl\{(i,j)\mid i<j\;\text{and}\;w(i)>w(j)\bigr\}.
\end{displaymath}
The elements of $\inv(w)$ are the \tdef{inversions} of $w$.  In particular, every descent of $w$ is also an inversion.  The \tdef{(left) weak order} on $\Symmetric_{n}$ is defined by setting $w_{1}\leq_{L}w_{2}$ if and only if $\inv(w_{1})\subseteq\inv(w_{2})$ for all $w_{1},w_{2}\in\Symmetric_{n}$.

\begin{definition}
	The \tdef{parabolic Tamari lattice} $\Tamari_{\alpha}\defs\bigl(\Symmetric_{\alpha}(231),\leq_{L}\bigr)$ is defined to be the restriction of the weak order to $(\alpha,231)$-avoiding permutations.
\end{definition}

\begin{figure}
	\centering
	\begin{subfigure}[b]{.45\textwidth}
		\centering
		\paraTam{1.4\textwidth}
		\caption{The parabolic Tamari lattice $\Tamari_{(2,2,1)}$.}
		\label{fig:parabolic_tamari_5_221}
	\end{subfigure}
	\hspace*{1cm}
	\begin{subfigure}[b]{.45\textwidth}
		\centering
		\begin{tikzpicture}\small
			\def\x{2};
			\def\y{2};
			\def\s{.8};
			\draw(1*\x,1*\y) node(n24){$(2,4)$};
			\draw(2*\x,1*\y) node(n23){$(2,3)$};
			\draw(3*\x,1*\y) node(n14){$(1,4)$};
			\draw(1*\x,1.5*\y) node(n25){$(2,5)$};
			\draw(2*\x,1.5*\y) node(n15){$(1,5)$};
			\draw(3*\x,1.5*\y) node(n13){$(1,3)$};
			\draw(1*\x,2*\y) node(n35){$(3,5)$};
			\draw(2*\x,2*\y) node(n45){$(4,5)$};
			\draw[->](n45) -- (n15);
			\draw[->](n45) -- (n25);
			\draw[->](n35) -- (n45);
			\draw[->](n35) -- (n25);
			\draw[->](n35) -- (n15);
			\draw[->](n25) -- (n23);
			\draw[->](n25) -- (n24);
			\draw[->](n15) -- (n23);
			\draw[->](n15) -- (n24);
			\draw[->](n15) -- (n25);
			\draw[->](n15) -- (n13);
			\draw[->](n15) -- (n14);
			\draw[->](n24) -- (n23);
			\draw[->](n14) -- (n13);
			\draw[->](n14) -- (n23);
			\draw[->](n14) .. controls (2.75*\x,.75*\y) and (1.25*\x,.75*\y) .. (n24);
			\draw[->](n13) -- (n23);
		\end{tikzpicture}
		\caption{The Galois graph of $\Tamari_{(2,2,1)}$.}
		\label{fig:parabolic_galois_5_221}
	\end{subfigure}
	\caption{The lattice $\Tamari_{(2,2,1)}$ and its Galois graph.}
	\label{fig:parabolic_5_221}
\end{figure}
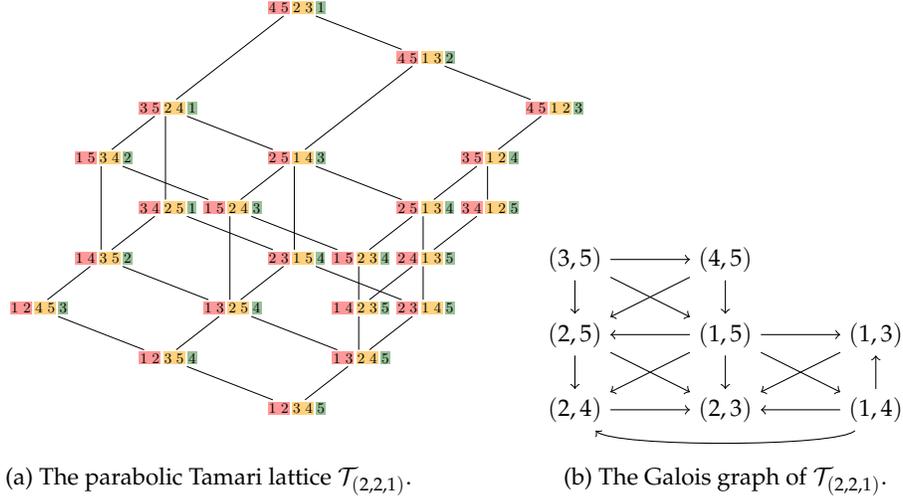

Figure~\ref{fig:parabolic_tamari_5_221} shows the parabolic Tamari lattice $\Tamari_{(2,2,1)}$.  The elements of $\Symmetric_{(2,2,1)}(231)$ are shown in one-line notation, where the $(2,2,1)$-regions are separated by vertical bars.  Theorem~1.1 in \cite{muehle18tamari} states that $\Tamari_{\alpha}$ is indeed a lattice. Moreover, the following structural result holds.

\begin{theorem}[{\cite[Theorem~1.3]{muehle18noncrossing}}]\label{thm:parabolic_tamari_cul_extremal}
	For every integer composition $\alpha$, the parabolic Tamari lattice $\Tamari_{\alpha}$ is extremal and congruence uniform.
\end{theorem}

It follows from \cite[Corollary~4.5]{muehle18noncrossing} that the join-irreducible elements of $\Tamari_{\alpha}$ are precisely the $(\alpha,231)$-avoiding permutations with a unique descent.  As a consequence of Proposition~\ref{prop:extremal_cu_galois}, the Galois graph of $\Tamari_{\alpha}$ (defined in Section~\ref{sec:extremal_lattices}) can be realized as a directed graph whose vertices are pairs of integers.  Let us make this more precise.

\begin{theorem}[{\cite[Theorem~1.8]{muehle18noncrossing}}]\label{thm:parabolic_tamari_galois}
	Let $\alpha$ be an integer composition.  The Galois graph of $\Tamari_{\alpha}$ is isomorphic to the directed graph with vertex set 
	\begin{displaymath}
		\bigl\{(a,b)\mid 1\leq a<b\leq\lvert\alpha\rvert\;\text{and}\;a,b\;\text{belong to different}\;\alpha\text{-regions}\bigr\},
	\end{displaymath}
	where there is a directed edge $(a_{1},b_{1})\to (a_{2},b_{2})$ if and only if $(a_{1},b_{1})\neq(a_{2},b_{2})$ and 
	\begin{itemize}
		\item either $a_{1}$ and $a_{2}$ belong to the same $\alpha$-region and $a_{1}\leq a_{2}<b_{2}\leq b_{1}$,
		\item or $a_{1}$ and $a_{2}$ belong to different $\alpha$-regions and $a_{2}<a_{1}<b_{2}\leq b_{1}$, where $a_{1}$ and $b_{2}$ belong to different $\alpha$-regions, too.
	\end{itemize}
\end{theorem}

Figure~\ref{fig:parabolic_galois_5_221} illustrates this result, and shows the Galois graph of $\Tamari_{(2,2,1)}$.  If $\alpha=(1,1,\ldots,1)$ and $n=\lvert\alpha\rvert$, then the classical Tamari lattice introduced in \cite{tamari51monoides} is $\Tamari_{n}\defs\Tamari_{(1,1,\ldots,1)}$~\cite[Theorem~9.6(ii)]{bjorner97shellable}.

\subsection{On the Duality of the Parabolic Tamari Lattices}
	\label{sec:duality_tamari}
We now show that reversing the composition $\alpha$, produces a parabolic Tamari lattice which is dual to $\Tamari_{\alpha}$.

Recall from Section~\ref{sec:members_parabolic_cataland} that, for an integer composition $\alpha=(\alpha_{1},\alpha_{2},\ldots,\alpha_{r})$, the reverse composition is $\overline{\alpha}=(\alpha_{r},\alpha_{r-1},\ldots,\alpha_{1})$.

Moreover, for every $n>0$, and every permutation $w\in\Symmetric_{n}$ we define the \tdef{reverse permutation} $\overline{w}$ by setting $\overline{w}(i)=w(n+1-i)$ for $i\in[n]$.  If $w\in\Symmetric_{\alpha}$, and we additionally reverse the order of the entries of each $\overline{\alpha}$-region of $\overline{w}$, then we obtain a permutation $\widetilde{w}\in\Symmetric_{\overline{\alpha}}$.

For example, let $w=24|157|3|6\in\Symmetric_{(2,3,1,1)}$.  Then we have $\overline{w}=6375142$, and $\widetilde{w}=6|3|157|24\in\Symmetric_{(1,1,3,2)}$.

\begin{lemma}\label{lem:reversal_weak_order}
	Let $w_{1},w_{2}\in\Symmetric_{\alpha}$.  We have $\inv(w_{1})\subseteq\inv(w_{2})$ if and only if $\inv(\widetilde{w}_{1})\supseteq\inv(\widetilde{w}_{2})$.
\end{lemma}
\begin{proof}
	Let $\alpha=(\alpha_{1},\alpha_{2},\ldots,\alpha_{r})$ be an integer composition.  We only prove the implication ``$\inv(w_{1})\subseteq\inv(w_{2})$ implies $\inv(\widetilde{w}_{1})\supseteq\inv(\widetilde{w}_{2})$'', the other implication follows analogously.
	
	Let $w_{1},w_{2}\in\Symmetric_{\alpha}$ such that $\inv(w_{1})\subseteq\inv(w_{2})$, and choose integers $a,b$ in different $\overline{\alpha}$-regions such that $(a,b)\notin\inv(\widetilde{w}_{1})$.  
	
	Let $a',b'$ be the unique indices with $w_{1}(a')=\widetilde{w}_{1}(a)$ and $w_{1}(b')=\widetilde{w}_{1}(b)$.  Since the entries in each $\overline{\alpha}$-region are ordered linearly, and by going from $\widetilde{w}_{1}$ to $w_{1}$ we revert this order, we conclude that $w_{2}(a')=\widetilde{w}_{2}(a)$ and $w_{2}(b')=\widetilde{w}_{2}(b)$.
	
	Suppose that $a$ belongs to the $i$-th $\overline{\alpha}$-region, and suppose that $b$ belongs to the $j$-th $\overline{\alpha}$-region.  Since $a<b$ we have $i<j$.  (By construction, it follows from $\widetilde{w}_{1}\in\Symmetric_{\overline{\alpha}}$ that there are no inversions within the same $\overline{\alpha}$-region.)  By construction we conclude that $a$ and $a'$ belong to the $(r+1-i)$-th $\alpha$-region and $b$ and $b'$ belong to the $(r+1-j)$-th $\alpha$-region.  It follows that $b'<a'$.

	Since $(a,b)\notin\inv(\widetilde{w}_{1})$ we must have $w_{1}(a') = \widetilde{w}_{1}(a) < \widetilde{w}_{1}(b) = w_{1}(b')$, which implies $(b',a')\in\inv(w_{1})\subseteq\inv(w_{2})$.  It follows that $\widetilde{w}_{2}(b) = w_{2}(b') > w_{2}(a') = \widetilde{w}_{2}(a)$, which yields $(a,b)\notin\inv(\widetilde{w}_{2})$.
	
	By contraposition we conclude that $\inv(\widetilde{w}_{1})\supseteq\inv(\widetilde{w}_{2})$.  
\end{proof}

\begin{definition}\label{def:parabolic_132_avoiding_permutation}
	In a permutation $w\in\Symmetric_{\alpha}$, an \tdef{$(\alpha,132)$-pattern} is a triple of indices $i<j<k$ each in different $\alpha$-regions such that $w(i)<w(k)<w(j)$ and $w(k)=w(i)+1$.  A permutation in $\Symmetric_{\alpha}$ without $(\alpha,132)$-patterns is \tdef{$(\alpha,132)$-avoiding}.  
\end{definition}

We denote the set of all $(\alpha,132)$-avoiding permutations of $\Symmetric_{\alpha}$ by $\Symmetric_{\alpha}(132)$, and consider how reversal acts on $(\alpha,231)$-patterns. 

\begin{lemma}\label{lem:reversal_patterns}
	A permutation $w\in\Symmetric_{\alpha}$ has an $(\alpha,231)$-pattern if and only if $\widetilde{w}\in\Symmetric_{\overline{\alpha}}$ has an $(\overline{\alpha},132)$-pattern.
\end{lemma}
\begin{proof}
	We only prove the implication ``$w$ has an $(\alpha,231)$-pattern implies that $\widetilde{w}$ has an $(\overline{\alpha},132)$-pattern'', the other implication follows analogously.

	Let $w\in\Symmetric_{\alpha}$ have an $(\alpha,231)$-pattern $(i,j,k)$.  Let $a=w(i)$, $b=w(j)$ and $c=w(k)$.  By Definition~\ref{def:parabolic_231_avoiding_permutation} we have $i<j<k$ in different $\alpha$-regions, $a<b$ and $a=c+1$.  Let $i',j',k'$ be such that $\widetilde{w}(i')=a$, $\widetilde{w}(j')=b$ and $\widetilde{w}(k')=c$.  As in the proof of Lemma~\ref{lem:reversal_weak_order} we find that $k'<j'<i'$ are in different $\alpha$-regions.  By Definition~\ref{def:parabolic_132_avoiding_permutation}, $(k',j',i')$ is an $(\alpha,132)$-pattern of $\widetilde{w}$.
\end{proof}

\begin{theorem}\label{thm:parabolic_tamari_duality}
	For every integer composition $\alpha$, the lattice $\Tamari_{\alpha}$ is isomorphic to the dual of $\Tamari_{\overline{\alpha}}$.
\end{theorem}
\begin{proof}
	We use the fact that $\Tamari_{\alpha}$ arises as a quotient lattice of $\bigl(\Symmetric_{\alpha},\leq_{L}\bigr)$~\cite[Proposition~3.18]{muehle18tamari}.  It follows from \cite[Section~3.4]{muehle18tamari} that the set of greatest elements of the corresponding congruence-classes is precisely $\Symmetric_{\alpha}(132)$, and the set of least elements of the corresponding congruence-classes is $\Symmetric_{\alpha}(231)$.  
	
	Let us define a partial order on the symmetric group by setting $w_{1}\leq_{R}w_{2}$ if and only if $\inv(w_{1})\supseteq\inv(w_{2})$.  We conclude that the dual lattice of $\Tamari_{\alpha}=\bigl(\Symmetric_{\alpha}(231),\leq_{L}\bigr)$ is isomorphic to $\bigl(\Symmetric_{\alpha}(132),\leq_{R}\bigr)$.  Lemmas~\ref{lem:reversal_weak_order} and \ref{lem:reversal_patterns} now imply that $\bigl(\Symmetric_{\alpha}(132),\leq_{R}\bigr)\cong\bigl(\Symmetric_{\overline{\alpha}}(231),\leq_{L}\bigr)=\Tamari_{\overline{\alpha}}$.
\end{proof}

Compare Figures~\ref{fig:parabolic_tamari_5_221} and \ref{fig:parabolic_tamari_5_221_dual} for an illustration of Theorem~\ref{thm:parabolic_tamari_duality}.

\begin{figure}
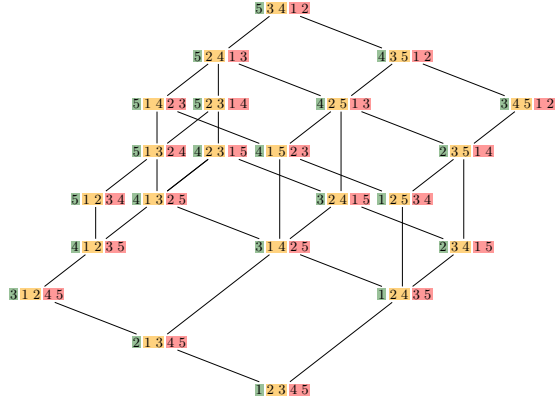

	\centering
	\paraTamDual{0.6\textwidth}
	\caption{The parabolic Tamari lattice $\Tamari_{(1,2,2)}$.}
	\label{fig:parabolic_tamari_5_221_dual}
\end{figure}

\section{$\nu_{\alpha}$-Tamari Lattices}
	\label{sec:nu_tamari_lattices}
In this section we define a partial order on the set $\Dyck_{\alpha}$ of all $\alpha$-Dyck paths.  This construction can in fact be carried out for any set of Dyck paths that stay weakly above a fixed path $\nu$ formed by north and east steps.  See~\cites{ceballos18the,preville17enumeration} for the general setup.  All of the following structural properties of $\nu_{\alpha}$-Tamari lattices hold in this general case, too.

Let $\mu\in\Dyck_{\alpha}$, and consider a lattice point $\vec{p}$ in $\mu$.  We define $\horiz_{\alpha}(\vec{p})$ to be the maximal number of east-steps that can be added after $\vec{p}$ without crossing $\nu_{\alpha}$.  For any valley $\vec{p}$ of $\mu$, let $\vec{q}$ be the first lattice point on $\mu$ such that $\horiz_{\alpha}(\vec{p})=\horiz_{\alpha}(\vec{q})$. We denote by $\mu[\vec{p},\vec{q}]$ the subpath of $\mu$ from $\vec{p}$ to $\vec{q}$.  Let $\mu'$ be the unique $\alpha$-Dyck path which arises from $\mu$ by swapping the subpath $\mu[\vec{p},\vec{q}]$ with the east step preceding $\vec{p}$.  Let us abbreviate this operation by $\mu\lessdot_{\alpha}\mu'$.  

This operation induces an acyclic relation on $\Dyck_{\alpha}$, and we denote by $\leq_{\alpha}$ its reflexive and transitive closure.  

\begin{definition}
	The poset $\Tamari_{\nu_{\alpha}}\defs\bigl(\Dyck_{\alpha},\leq_{\alpha}\bigr)$ is the \tdef{$\nu_{\alpha}$-Tamari lattice}.
\end{definition}

The $\nu_{(1,2,2)}$-Tamari lattice is shown in Figure~\ref{fig:nu_tamari_5_221}.

\begin{theorem}\label{thm:nu_tamari_cul_extremal}
	For every integer composition $\alpha$, the $\nu_{\alpha}$-lattice $\Tamari_{\nu_{\alpha}}$ is extremal and congruence uniform.
\end{theorem}
\begin{proof}
	First of all, \cite[Theorem~1]{preville17enumeration} states that $\Tamari_{\nu_{\alpha}}$ is indeed a lattice.  Moreover, \cite[Theorem~3]{preville17enumeration} states that $\Tamari_{\nu_{\alpha}}$ is an interval in $\Tamari_{2\lvert\alpha\rvert+1}$.  
	
	It was shown in \cite[Theorem~22]{markowsky92primes} that $\Tamari_{N}$ is extremal for every integer $N>0$.  However, it follows from \cite[Theorem~14(ii)]{markowsky92primes} that intervals of extremal lattices need not be extremal.  Fortunately, \cite[Theorem~9]{thomas06analogue} implies that $\Tamari_{N}$ is \tdef{trim}.  (This is a property that is somewhat stronger than extremality.)  Theorem~1 in \cite{thomas06analogue} states that every interval of a trim lattice is trim, too.  We conclude that $\Tamari_{\nu_{\alpha}}$ is trim, and therefore extremal.
	
	Theorem~3.5 in \cite{geyer94on} states that $\Tamari_{N}$ is congruence uniform for every integer $N>0$, and \cite[Theorem~4.3]{day79characterizations} implies that intervals of congruence-uniform lattices are congruence uniform again.  We conclude that $\Tamari_{\nu_{\alpha}}$ is congruence uniform.
\end{proof}

We conclude this section with another useful result, which nicely parallels Theorem~\ref{thm:parabolic_tamari_duality}.

\begin{theorem}[{\cite[Theorem~2]{preville17enumeration}}]\label{thm:nu_tamari_duality}
	For every integer composition $\alpha$, the lattice $\Tamari_{\nu_{\alpha}}$ is isomorphic to the dual of $\Tamari_{\overline{\nu}_{\alpha}}$.
\end{theorem}

\subsection{$\nu_{\alpha}$-Bracket Vectors}
	\label{sec:bracket_vectors}
Following~\cite{ceballos18the}, we may as well represent $\Tamari_{\nu_{\alpha}}$ as a lattice of certain integer vectors under componentwise order.

Let $n=\lvert\alpha\rvert$.  The \tdef{minimal $\nu_{\alpha}$-bracket vector} $\minbracket$ is the vector with $2n+1$ entries that contains the $y$-coordinates of the lattice points of $\nu_{\alpha}$ read in order from $(0,0)$ to $(n,n)$.  The \tdef{fixed positions} are the elements of the set $F=\{f_{0},f_{1},\ldots,f_{n}\}$, where $f_{i}$ denotes the last occurrence of $i$ in $\minbracket$.  

\begin{definition}\label{def:bracket-vector}
	A \tdef{$\nu_{\alpha}$-bracket vector} is a vector $\bracket=(b_{1},b_{2},\ldots,b_{2n+1})$ that satisfies
	\begin{itemize}
		\item $b_{f_{i}} = i$ for $0\leq i\leq n$;
		\item $\minbracket_{i} \leq b_{i} \leq n$ for all $i$;
		\item if $b_{i}=k$, then $b_{j}\leq k$ for all $i\leq j\leq f_{k}$.
	\end{itemize}
\end{definition}

\begin{theorem}[{\cite[Theorem~4.2]{ceballos18the}}]
	For every integer composition $\alpha$, the lattice of $\nu_{\alpha}$-bracket vectors under componentwise order is isomorphic to $\Tamari_{\nu_{\alpha}}$.
\end{theorem}

For $\mu\in\Dyck_{\alpha}$ we may directly construct the corresponding $\nu_{\alpha}$-bracket vector as follows; see \cite[Section~4.3]{ceballos18the}.  

\begin{construction}\label{constr:bracket_vectors}
	Let $\mu\in\Dyck_{\alpha}$ for some integer composition $\alpha$.  Suppose that $\mu$ runs through $k_{i}$ lattice points of $y$-coordinate $i$.  We start with an empty vector $\bracket_{0}$ of length $2n+1$.  In the $i$-th step we construct $\bracket_{i}$ from $\bracket_{i-1}$ by filling the $k_{i}$ rightmost available spaces before and including the fixed position $f_{i}$ with the value $i$.  After $n$ steps we return the bracket vector $\bracket_{\mu}=\bracket_{n}$.
\end{construction}

Given a $\nu_{\alpha}$-bracket vector $\bracket$, we denote by $\bracketred$ the \tdef{reduced $\nu_{\alpha}$-bracket vector}, \ie the vector that contains all entries of $\bracket$ in the same order, except for those at the fixed positions.  It is clear that $\bracket\leq\bracket'$ if and only if $\bracketred\leq\bracketred'$, where $\leq$ is componentwise order.  Figure~\ref{fig:nu_tamari_5_221_bracket} shows the lattice of reduced $\nu_{(1,2,2)}$-bracket vectors.

\begin{figure}
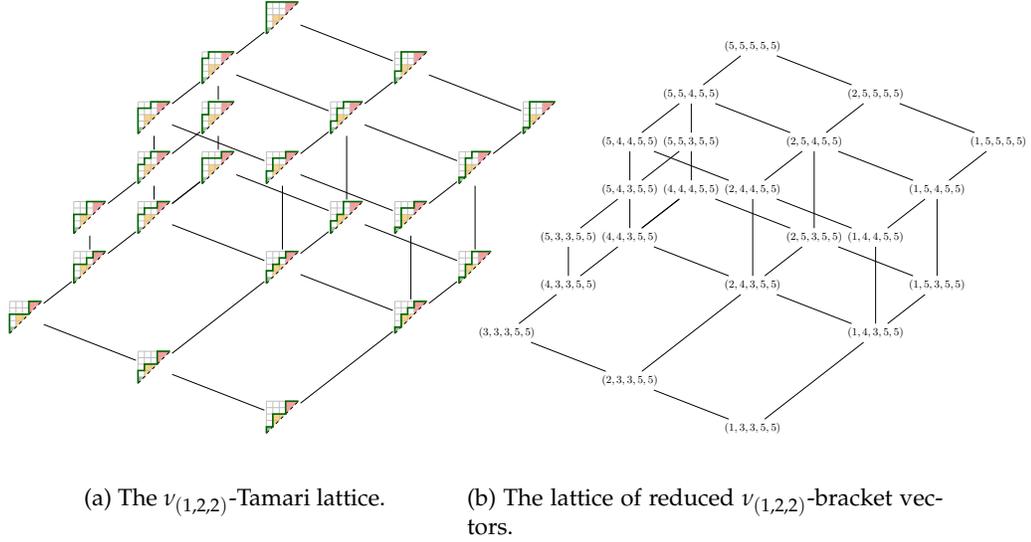

	\centering
	\begin{subfigure}[t]{.5\textwidth}
		\centering
 		\nuTam{1.2\textwidth}
		\caption{The $\nu_{(1,2,2)}$-Tamari lattice.}
		\label{fig:nu_tamari_5_221}
	\end{subfigure}
	\hspace*{-.25cm}
	\begin{subfigure}[t]{.5\textwidth}
		\centering
		\nuTamvec{1.2\textwidth}
		\caption{The lattice of reduced $\nu_{(1,2,2)}$-bracket vectors.}
		\label{fig:nu_tamari_5_221_bracket}
	\end{subfigure}
	\caption{Two representations of a $\nu_{\alpha}$-Tamari lattice.}
	\label{fig:nu_tamari_example}
\end{figure}

\subsection{The Galois Graph of $\Tamari_{\nu_{\alpha}}$}
	\label{sec:galois_graph_nu_tamari}
By Theorem~\ref{thm:nu_tamari_cul_extremal}, the lattice $\Tamari_{\nu_{\alpha}}$ is extremal for every integer composition $\alpha$, and by Theorem~\ref{thm:markowskys_representation} we may represent $\Tamari_{\nu_{\alpha}}$ by its Galois graph.  In this section, we characterize the Galois graph of $\Tamari_{\nu_{\alpha}}$.

Since $\Tamari_{\nu_{\alpha}}$ is also congruence uniform, we may apply Proposition~\ref{prop:extremal_cu_galois} and define $\Galois\bigl(\Tamari_{\nu_{\alpha}}\bigr)$ in terms of the join-irreducible elements of $\Tamari_{\nu_{\alpha}}$.  In general, however, it is much easier to describe meet-irreducible Dyck paths (because they have a unique valley).  In view of Theorem~\ref{thm:nu_tamari_duality}, we may regard $\Galois\bigl(\Tamari_{\nu_{\alpha}}\bigr)$ as a directed graph with vertex set $\MI\bigl(\Tamari_{\overline{\nu}_{\alpha}}\bigr)$, where there is a directed edge $\mu\to\tilde{\mu}$ if and only if $\mu\neq\tilde{\mu}$ and $\tilde{\mu}\geq \tilde{\mu}^{*}\wedge \mu$.

Since meet-irreducible Dyck paths are precisely those with a unique valley, we may write $\mu_{p,q}$ for the unique meet-irreducible element of $\Tamari_{\overline{\nu}_{\alpha}}$ that has its only valley at $(p,q)$.  Moreover, we abbreviate $k^{(m)}\defs\underbrace{k,k,\ldots,k}_{m}$ for any integer $k$, and any nonnegative integer $m$.  

Let $\alpha=(\alpha_{1},\alpha_{2},\ldots,\alpha_{r})$ be a composition of $n$, and recall that we have defined $s_{i}=\alpha_{1}+\alpha_{2}+\cdots+\alpha_{i}$ and $t_{i}=n-s_{r-i}$ for all $i\in\{0,1,\ldots,r\}$ in the beginning of Section~\ref{sec:members_parabolic_cataland}.

\begin{lemma}\label{lem:mi_brackets}
	Let $\mu_{p,q}\in\MI(\Tamari_{\overline{\nu}_{\alpha}})$, where $t_{k}\leq q<t_{k+1}$ for some $k\in[r]$.  The reduced $\overline{\nu}_{\alpha}$-bracket vector of $\mu_{p,q}$ is
	\begin{displaymath}
		\bracketred(\mu_{p,q}) = \Bigl(n^{(t_{k}-p)},q^{(p)},n^{(n-t_{k})}\Bigr).
	\end{displaymath}
	Moreover, the reduced $\overline{\nu}_{\alpha}$-bracket vector of the unique upper cover $\mu_{p,q}^{*}$ of $\mu_{p,q}$ is
	\begin{displaymath}
		\bracketred(\mu_{p,q}^{*}) = 
		\begin{cases}
			\Bigl(n^{(t_{k}-p)},q+1,q^{(p-1)},n^{(n-t_{k})}\Bigr), & \text{if}\;q<t_{k+1}-1,\\
			\Bigl(n^{(t_{k}-p+1)},q^{(p-1)},n^{(n-t_{k})}\Bigr), & \text{if}\;q=t_{k+1}-1.
		\end{cases}
	\end{displaymath}
\end{lemma}
\begin{proof} 
	We will use Construction~\ref{constr:bracket_vectors}. Since $\mu_{p,q}$ has a unique valley, it contains precisely $p$ lattice points of $y$-coordinate $q$, and $n-p$ lattice points of $y$-coordinate $n$.  Every other $y$-coordinate is met exactly once.  Therefore, the $y$-coordinates in $[n]\setminus\{q,n\}$ fill only fixed positions, and all the other positions in the bracket vector of $\mu_{p,q}$ are filled with either $q$ or $n$.  Now, since we reach $y$-coordinate $q$ strictly before we reach $y$-coordinate $n$, we may first insert $p$-times the value $q$, and then fill the remaining available positions with values $n$ (possibly before the first $q$ and certainly after the last).
	
	Let $(p',q')$ be the next lattice point on $\mu_{p,q}$ with $\horiz_{\overline{\nu}_{\alpha}}\bigl((p',q')\bigr)=\horiz_{\overline{\nu}_{\alpha}}\bigl((p,q)\bigr)$.  If $q=t_{k+1}-1$, then $(p',q')$ is $\Bigl(n-\horiz_{\overline{\nu}_{\alpha}}\bigl((p,q)\bigr),n\Bigr)$, and if $q<t_{k+1}-1$, then $(p',q')=(p,q+1)$.  The description of the reduced bracket vector of $\mu_{p,q}^{*}$ follows immediately.
\end{proof}

\begin{lemma}\label{lem:mi_arrows}
	Let $\mu_{p_{1},q_{1}},\mu_{p_{2},q_{2}}\in\MI\bigl(\Tamari_{\overline{\nu}_{\alpha}}\bigr)$, where $t_{k_{1}}\leq q_{1}<t_{k_{1}+1}$ and $t_{k_{2}}\leq q_{2}<t_{k_{2}+1}$ for some $k_{1},k_{2}\in[r]$.   We have $\mu_{p_{2},q_{2}}\geq \mu_{p_{2},q_{2}}^{*}\wedge\mu_{p_{1},q_{1}}$ if and only if $t_{k_{1}}-p_{1}\leq t_{k_{2}}-p_{2}<t_{k_{1}}$ and $q_{1}\leq q_{2}$.
\end{lemma}
\begin{proof}
	From Lemma~\ref{lem:mi_brackets} it follows that
	\begin{align*}
		\bracketred(\mu_{p_{1},q_{1}}) & = \Bigl(n^{(t_{k_{1}}-p_{1})},q_{1}^{(p_{1})},n^{(n-t_{k_{1}})}\Bigr),\\
		\bracketred(\mu_{p_{2},q_{2}}) & = \Bigl(n^{(t_{k_{2}}-p_{2})},q_{2}^{(p_{2})},n^{(n-t_{k_{2}})}\Bigr),\\
		\bracketred(\mu_{p_{2},q_{2}}^{*}) & = \Bigl(n^{(t_{k_{2}}-p_{2})},q'_{2},q_{2}^{(p_{2}-1)},n^{(n-t_{k_{2}})}\Bigr),
	\end{align*}
	where $q'_{2}>q_{2}$.  For $i\in[n]$ let $a_{i}$ denote the $i$-th entry in $\bracketred(\mu_{p_{1},q_{1}})$, let $b_{i}$ denote the $i$-th entry in $\bracketred(\mu_{p_{2},q_{2}})$, let $c_{i}$ denote the $i$-th entry in $\bracketred(\mu_{p_{2},q_{2}}^{*})$.  
	
	In order to determine when $\mu_{p_{2},q_{2}}\geq\mu_{p_{2},q_{2}}^{*}\wedge\mu_{p_{1},q_{1}}$ is satisfied, we have to figure out under which conditions 
	\begin{equation}\label{eq:target}
		b_{i} \geq \min(c_{i},a_{i})
	\end{equation}
	holds for all $i\in[n]$.  By Lemma~\ref{lem:mi_brackets} it follows that $b_{i}=c_{i}$ for $i\neq t_{k_{2}}-p_{2}+1$ and $b_{t_{k_{2}}-p_{2}+1}<c_{t_{k_{2}}-p_{2}+1}$.
	
	Hence, \eqref{eq:target} holds trivially if $i\neq t_{k_{2}}-p_{2}+1$.  It thus remains to consider the case $i=t_{k_{2}}-p_{2}+1$.  We have $c_{i}=q'_{2}>q_{2}$, which implies that $\min(a_{i},c_{i})\leq b_{i}$ holds if and only if $a_{i}\leq b_{i}$.  In particular, we need to have $a_{i}<n$, which is the case precisely when $t_{k_{1}}-p_{1}<i\leq t_{k_{1}}$.
\end{proof}

\begin{figure}
	\centering
	\begin{tikzpicture}\small
		\def\x{2};
		\def\y{2};
		\def\s{.8};
		\draw(1*\x,1*\y) node(n24){$(2,4)$};
		\draw(2*\x,1*\y) node(n23){$(1,4)$};
		\draw(3*\x,1*\y) node(n14){$(2,3)$};
		\draw(1*\x,1.5*\y) node(n25){$(3,4)$};
		\draw(2*\x,1.5*\y) node(n15){$(3,3)$};
		\draw(1*\x,2*\y) node(n35){$(1,1)$};
		\draw(2*\x,2*\y) node(n45){$(1,2)$};
		\draw(3*\x,1.5*\y) node(n13){$(1,3)$};
		\draw[->](n45) -- (n15);
		\draw[->](n45) -- (n25);
		\draw[->](n35) -- (n45);
		\draw[->](n35) -- (n25);
		\draw[->](n35) -- (n15);
		\draw[->](n25) -- (n23);
		\draw[->](n25) -- (n24);
		\draw[->](n15) -- (n23);
		\draw[->](n15) -- (n24);
		\draw[->](n15) -- (n25);
		\draw[->](n15) -- (n13);
		\draw[->](n15) -- (n14);
		\draw[->](n24) -- (n23);
		\draw[->](n14) -- (n13);
		\draw[->](n14) .. controls (2.75*\x,.75*\y) and (1.25*\x,.75*\y) .. (n24);
		\draw[->](n14) -- (n23);
		\draw[->](n13) -- (n23);
	\end{tikzpicture}
	\caption{The Galois graph of $\Tamari_{\nu_{(2,2,1)}}$.}
	\label{fig:nu_galois_5_221}
\end{figure}
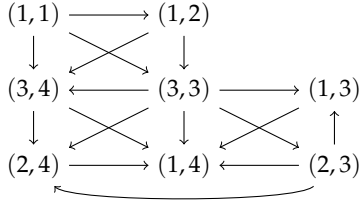

We may thus conclude the following description of the Galois graph of $\Tamari_{\nu_{\alpha}}$, which is illustrated in Figure~\ref{fig:nu_galois_5_221} for $\alpha=(2,2,1)$.

\begin{theorem}\label{thm:nu_tamari_galois}
	Let $\alpha$ be an integer composition into $r$ parts, and let $\nu_{\alpha}$ be the corresponding $\alpha$-bounce path.  The Galois graph of $\Tamari_{\nu_{\alpha}}$ is isomorphic to the directed graph with vertex set
	\begin{displaymath}
		\bigl\{(p,q)\mid 1\leq p\leq t_{k}\leq q<t_{k+1}\;\text{for some}\;k\in[r-1]\bigr\},
	\end{displaymath}
	where there is a directed edge from $(p_{1},q_{1})\to(p_{2},q_{2})$ if and only if $(p_{1},q_{1})\neq (p_{2},q_{2})$ and $t_{k_{1}}-p_{1}\leq t_{k_{2}}-p_{2}<t_{k_{1}}$ and $q_{1}\leq q_{2}$, where $k_{1}$ and $k_{2}$ are the unique indices in $[r]$ with $t_{k_{1}}\leq q_{1}<t_{k_{1}+1}$ and $t_{k_{2}}\leq q_{2}<t_{k_{2}+1}$.
\end{theorem}
\begin{proof}
	We have seen in Theorem~\ref{thm:nu_tamari_cul_extremal} that $\Tamari_{\nu_{\alpha}}$ is extremal and congruence uniform, and therefore, by Proposition~\ref{prop:extremal_cu_galois}, its Galois graph is (isomorphic to) the directed graph on the vertex set $\JI(\Tamari_{\nu_{\alpha}})$, which has a directed edge $x\to y$ if and only if $x\neq y$ and $y\leq y_{*}\vee x$.  Now, Theorem~\ref{thm:nu_tamari_duality} implies that $\Tamari_{\nu_{\alpha}}\cong\Tamari_{\overline{\nu}_{\alpha}}$, and let us denote this isomorphism by $\Psi$.  It follows that $\mu\in\JI\big(\Tamari_{\nu_{\alpha}}\bigr)$ if and only if $\Psi(\mu)\in\MI\bigl(\Tamari_{\overline{\nu}_{\alpha}}\bigr)$, and for $\mu_{1},\mu_{2}\in\JI\bigl(\Tamari_{\nu_{\alpha}}\bigr)$ we have $\mu_{2}\leq(\mu_{2})_{*}\vee\mu_{1}$ if and only if $\Psi(\mu_{2})\geq\Psi(\mu_{2})^{*}\wedge\Psi(\mu_{1})$.  We may therefore use Lemma~\ref{lem:mi_arrows} to describe $\Galois(\Tamari_{\nu_{\alpha}})$.

	\medskip
	
	Let $V$ denote the set of pairs $(p,q)$ with $1\leq p\leq t_{k}\leq q<t_{k+1}$ for some $k\in[r-1]$.  Let $G$ be the directed graph with vertex set $V$, which has a directed edge precisely under the conditions given in the statement.  
	
	It is straightforward to verify that the elements of $V$ are precisely the lattice points in the rectangle from $(0,0)$ to $(n,n)$ that do not have a coordinate equal to $0$ or $n$, and that lie weakly above $\overline{\nu}_{\alpha}$.  Therefore, for each such pair $(p,q)$ we can find a unique $\overline{\nu}_{\alpha}$-path that has a unique valley at $(p,q)$.  Hence, the vertices of $G$ are in bijection with the meet-irreducible elements of $\Tamari_{\overline{\nu}_{\alpha}}$ via the map $(p,q)\mapsto\mu_{p,q}$.  
	
	By Lemma~\ref{lem:mi_arrows} we see that there exists a directed edge $(p_{1},q_{1})\to(p_{2},q_{2})$ in $G$ if and only if $\mu_{p_{1},q_{1}}\neq\mu_{p_{2},q_{2}}$ and $\mu_{p_{2},q_{2}}\geq\mu_{p_{2},q_{2}}^{*}\wedge\mu_{p_{1},q_{1}}$.  In view of the first paragraph of this proof, this is exactly the case when there exists a directed edge $\Psi^{-1}(\mu_{p_{1},q_{1}})\to\Psi^{-1}(\mu_{p_{2},q_{2}})$ in $\Galois\bigl(\Tamari_{\nu_{\alpha}}\bigr)$.  We conclude that $G$ and $\Galois\bigl(\Tamari_{\nu_{\alpha}}\bigr)$ are isomorphic.
\end{proof}

\begin{figure}
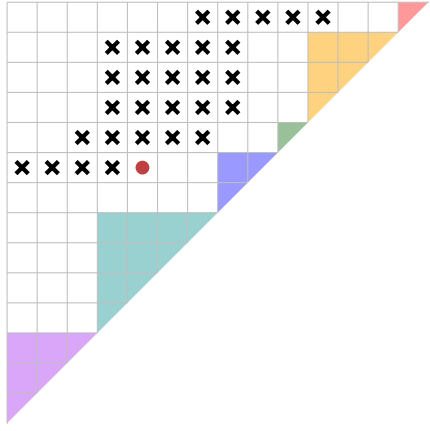

	\centering
	\galoisoncells{1}
	\caption{In the Galois graph of the $\nu$-Tamari lattice given by the above bounce path, there exist a directed arrow from the box with a red dot to any of the boxes with a cross.  (We identify boxes with the lattice points on their bottom right corner.)}
	\label{fig:illustation_galois_graph_nu_tamari}
\end{figure}

We may identify the lattice points $(p,q)$ above $\nu_{\alpha}$ with the box that has $(p,q)$ as its lower right corner.  From this perspective, Figure~\ref{fig:illustation_galois_graph_nu_tamari} illustrates
Theorem~\ref{thm:nu_tamari_galois}, by indicating to which boxes an arrow from the highlighted box may go in the Galois graph.

\begin{corollary}\label{cor:nu_tamari_galois_outedges}
	Let $\alpha$ be an integer composition.  In the Galois graph of $\Tamari_{\nu_{\alpha}}$ the vertex $(p,q)$ has precisely $p\bigl(\lvert\alpha\rvert-q\bigr)-1$ outgoing arrows.
\end{corollary}
\begin{proof}
	This follows immediately from Theorem~\ref{thm:nu_tamari_galois}.
\end{proof}

\section{$\Tamari_{\alpha}$ and $\Tamari_{\nu_{\alpha}}$ are Isomorphic}
	\label{sec:isomorphism_tamaris}
In this section we prove that for every integer composition $\alpha$ the lattices $\Tamari_{\nu_{\alpha}}$ and $\Tamari_{\alpha}$ are isomorphic.  In order to achieve this we make use of the bijective correspondence between the sets $\Symmetric_{\alpha}(231)$ and $\Dyck_{\alpha}$ that was first described in \cite[Theorem~1.2]{muehle18tamari}, and recalled here in Sections~\ref{sec:noncrossings_to_paths} and \ref{sec:permutations_to_noncrossings}.  

Let $\Theta_{1}$ and $\Theta_{2}$ be the bijections from Theorems~\ref{thm:bijection_paths_nc} and \ref{thm:bijection_nc_perms}.  Consequently, the map 
\begin{displaymath}
	\Theta\colon\Dyck_{\overline{\alpha}}\to\Symmetric_{\alpha}(231),\quad \mu\mapsto\Theta_{2}\circ\Theta_{1}(\mu)
\end{displaymath}
is a bijection.

\begin{remark}
	Note that there is a slight subtlety to our approach.  The bijection $\Theta$ maps elements of $\Dyck_{\overline{\alpha}}$ to elements of $\Symmetric_{\alpha}(231)$, and therefore \emph{reverses} the composition.  In fact, this comes in handy, because Theorems~\ref{thm:parabolic_tamari_duality}~and~\ref{thm:nu_tamari_duality} state that reversing the composition corresponds to taking the lattice dual.  Theorem~\ref{thm:parabolic_tamari_galois} describes the Galois graph of $\Tamari_{\alpha}$, while Theorem~\ref{thm:nu_tamari_galois} uses the dual perspective to describe the Galois graph of $\Tamari_{\nu_{\alpha}}$.  Therefore, everything is in place for the map $\Theta$ to take effect.  This is also the reason why we did not use the map $\lactoperm\circ\nntolac$ in this section.
\end{remark}

We now show that the map $\Theta$ extends to an isomorphism from $\Galois(\Tamari_{\nu_{\alpha}})$ to $\Galois(\Tamari_{\alpha})$.  Let us write $w_{a,b}$ for the unique join-irreducible element of $\Tamari_{\alpha}$ whose unique descent is $(a,b)$, 
and recall that $\mu_{p,q}$ is the unique meet-irreducible element of $\Tamari_{\overline{\nu}_{\alpha}}$ that has its unique valley at $(p,q)$.

\begin{lemma}\label{lem:bijection_irreducibles}
	Let $\mu_{p,q}\in\MI\bigl(\Tamari_{\overline{\nu}_{\alpha}}\bigr)$ such that $t_{k}\leq q<t_{k+1}$.  The image of $\Theta(\mu_{p,q})$ is $w_{a,b}\in\JI\bigl(\Tamari_{\alpha}\bigr)$, where 
	\begin{align*}
		a & = s_{r-k-1}+q-n+s_{r-k}+1,\\
		b & = s_{r-k}+p.
	\end{align*}
	In particular, $a$ belongs to the $(r-k)$-th $\alpha$-region, and $b$ belongs to the $i$-th $\alpha$-region for some $i>r-k$.
\end{lemma}
\begin{proof}
	Since $\mu_{p,q}$ has a unique valley at $(p,q)$, it follows from Construction~\ref{constr:path_to_nc} that $\Theta_{1}(\mu_{p,q})$ has a unique bump at positions $a,b$ with $a=s_{r-k-1}+\ell+1$, where $q=t_{k}+\ell$ for a unique index $k\in[r]$, and $b=s_{r-k}+p$.  (The value of $b$ follows immediately, since there is no other bump in $\Theta_{1}(\mu_{p,q})$ that could potentially block elements.)  By Construction~\ref{constr:nc_to_perm} we get that $w=\Theta(\mu_{p,q})$ has $(a,b)$ as its unique inversion, which proves the first claim.
	
	By construction we have $s_{r-k}=n-t_{k}$, and since $t_{k}\leq q<t_{k+1}$ we conclude that $s_{r-k-1}<a\leq s_{r-k}$, meaning that $a$ belongs to the $(r-k)$-th $\alpha$-region.  Moreover, since $p\geq 1$, we conclude that $b>s_{r-k}$, which implies the last statement.
\end{proof}

\begin{theorem}\label{thm:isomorphism_galois_graphs}
	Let $\alpha$ be an integer composition.  The map $\Theta$ induces an isomorphism from $\Galois\bigl(\Tamari_{\nu_{\alpha}}\bigr)$ to $\Galois\bigl(\Tamari_{\alpha}\bigr)$.
\end{theorem}
\begin{proof}
	Let $\alpha=(\alpha_{1},\alpha_{2},\ldots,\alpha_{r})$ with $n=\lvert\alpha\rvert$.  We denote by $G_{\nu_{\alpha}}$ the Galois graph of $\Tamari_{\nu_{\alpha}}$, and by $G_{\alpha}$ that of $\Tamari_{\alpha}$.  It follows from Lemma~\ref{lem:bijection_irreducibles} and the fact that $\Theta$ is a bijection which sends descents to valleys, that the vertex sets of $G_{\nu_{\alpha}}$ and $G_{\alpha}$ are in bijection via $\Theta$.  

	\medskip
	
	Now let $(p_{1},q_{1}),(p_{2},q_{2})$ be two vertices of $G_{\nu_{\alpha}}$ such that $t_{k_{1}}\leq q_{1}<t_{k_{1}+1}$ and $t_{k_{2}}\leq q_{2}<t_{k_{2}+1}$ for some $k_{1},k_{2}\in[r]$.  
	
	By Lemma~\ref{lem:bijection_irreducibles} we have $\Theta(\mu_{p_{1},q_{1}})=w_{a_{1},b_{1}}$ and $\Theta(\mu_{p_{2},q_{2}})=w_{a_{2},b_{2}}$, where 
	\begin{displaymath}\begin{aligned}
		& a_{1} = s_{r-k_{1}-1}+s_{r-k_{1}}+q_{1}+1-n, && b_{1} = s_{r-k_{1}}+p_{1},\\
		& a_{2} = s_{r-k_{2}-1}+s_{r-k_{2}}+q_{2}+1-n, && b_{2} = s_{r-k_{2}}+p_{2}.
	\end{aligned}\end{displaymath}
	Moreover, if we assume that $(p_{1},q_{1})\neq(p_{2},q_{2})$, then we see that $(a_{1},b_{1})\neq(a_{2},b_{2})$.  
	
	Suppose that there is a directed edge $(p_{1},q_{1})\to(p_{2},q_{2})$ in $G_{\nu_{\alpha}}$.  Theorem~\ref{thm:nu_tamari_galois} then implies that $t_{k_{1}}-p_{1}\leq t_{k_{2}}-p_{2}<t_{k_{2}}$ and $q_{1}\leq q_{2}$.  Since $q_{1}\leq q_{2}$, we conclude that $t_{k_{1}}\leq t_{k_{2}}$. Therefore, from the definition of $t_k$, we have two cases, $k_1=k_2$ or $k_1 < k_2$.
	
	(i) If $k_1=k_2$, then it follows that $a_{1}$ and $a_{2}$ belong to the same $\alpha$-region, and we must have $a_{1}\leq a_{2}$.  From the fact that $t_{k_{1}}-p_{1}\leq t_{k_{2}}-p_{2}$ follows $p_{2}\leq p_{1}$, and therefore $b_{2}\leq b_{1}$.  Lemma~\ref{lem:bijection_irreducibles} implies further that $a_{2}<b_{2}$; it follows now from Theorem~\ref{thm:parabolic_tamari_galois} that there exists a directed edge $(a_{1},b_{1})\to(a_{2},b_{2})$ in $G_{\alpha}$.

	(ii) If $k_1 < k_2$, then Lemma~\ref{lem:bijection_irreducibles} implies that $a_{1}$ belongs to the $(r-k_{1})$-th $\alpha$-region, and $a_{2}$ belongs to the $(r-k_{2})$-th $\alpha$-region.  It is clear that $r-k_{2}<r-k_{1}$, which implies $a_{2}<a_{1}$.  Moreover, we have 
	\begin{align*}
		t_{k_{1}}-p_{1} & = n-s_{r-k_{1}}-p_{1} = n-(s_{r-k_{1}}+p_{1}) = n-b_{1},\\
		t_{k_{2}}-p_{2} & = n-s_{r-k_{2}}-p_{2} = n-(s_{r-k_{2}}+p_{2}) = n-b_{2}.
	\end{align*}
	By assumption we have $t_{k_{1}}-p_{1}\leq t_{k_{2}}-p_{2}$, which implies $b_{2}\leq b_{1}$.  Finally, $t_{k_{2}}-p_{2}<t_{k_{1}}$ implies $b_{2}>s_{r-k_{1}}$.  Therefore, Theorem~\ref{thm:parabolic_tamari_galois} tells us that there is a directed edge $(a_{1},b_{1})\to(a_{2},b_{2})$ in $G_{\alpha}$.
	
	\medskip

	Conversely, let $(a_{1},b_{1}),(a_{2},b_{2})$ be two vertices of $G_{\alpha}$ such that $a_{1}$ belongs to the $(r-k_{1})$-th $\alpha$-region and $a_{2}$ belongs to the $(r-k_{2})$-th $\alpha$-region.

	By Lemma~\ref{lem:bijection_irreducibles} we have $\Theta^{-1}(w_{a_{1},b_{1}})=\mu_{p_{1},q_{1}}$ and $\Theta(w_{a_{2},b_{2}})=\mu_{p_{2},q_{2}}$, where 
	\begin{displaymath}\begin{aligned}
		& p_{1} = b_{1}-s_{r-k_{1}}, && q_{1} = n+a_{1}-s_{r-k_{1}-1}-s_{r-k_{1}}-1,\\
		& p_{2} = b_{2}-s_{r-k_{2}}, && q_{2} = n+a_{2}-s_{r-k_{2}-1}-s_{r-k_{2}}-1,\\
	\end{aligned}\end{displaymath}
	and $t_{k_{1}}\leq q_{1}<t_{k_{1}+1}$ and $t_{k_{2}}\leq q_{2}<t_{k_{2}+1}$.  Moreover, if we assume that $(a_{1},b_{1})\neq(a_{2},b_{2})$, then we see that $(p_{1},q_{1})\neq(p_{2},q_{2})$.  
	
	Suppose that there is a directed edge $(a_{1},b_{1})\to(a_{2},b_{2})$ in $G_{\alpha}$.  Theorem~\ref{thm:parabolic_tamari_galois} implies first of all that $r-k_{2}\leq r-k_{1}$.
	
	(i) If $k_{1}=k_{2}$, then Theorem~\ref{thm:parabolic_tamari_galois} implies also that $a_{1}\leq a_{2}<b_{2}\leq b_{1}$.  It follows immediately that $q_{1}\leq q_{2}$, and $t_{k_{1}}-p_{1}\leq t_{k_{1}}-p_{2}$.  Moreover, since $p_{2}>0$, we have $t_{k_{1}}-p_{2}<t_{k_{1}}$.  Theorem~\ref{thm:nu_tamari_galois} states that there exists a directed edge $(p_{1},q_{1})\to(p_{2},q_{2})$ in $G_{\nu_{\alpha}}$.
	
	(ii) If $r-k_{2}<r-k_{1}$, then Theorem~\ref{thm:parabolic_tamari_galois} implies further that $s_{r-k_{1}}<b_{2}\leq b_{1}$.  This means exactly that $s_{r-k_{1}}<p_{2}+s_{r-k_{2}}\leq p_{1}+s_{r-k_{1}}$.  It follows that $n-s_{r-k_{1}}>n-s_{r-k_{2}}-p_{2}\geq n-s_{r-k_{1}}-p_{1}$, which means precisely that $t_{k_{1}}>t_{k_{2}}-p_{2}\geq t_{k_{1}}-p_{1}$.  Moreover, $s_{r-k_{1}}>s_{r-k_{2}}$ implies $t_{k_{1}}<t_{k_{2}}$, and we conclude that $q_{1}<t_{k_{1}+1}\leq t_{k_{2}}\leq q_{2}$.  Theorem~\ref{thm:nu_tamari_galois} now states that there exists a directed edge $(p_{1},q_{1})\to(p_{2},q_{2})$ in $G_{\nu_{\alpha}}$.
	
	\medskip
	
	We have thus shown that $\Theta$ induces an isomorphism of directed graphs from $G_{\nu_{\alpha}}$ to $G_{\alpha}$, and the proof is complete.
\end{proof}

We may now prove Theorem~\ref{thm:parabolic_nu_tamari_isomorphism}, which states that $\Tamari_{\nu_{\alpha}}$ and $\Tamari_{\alpha}$ are isomorphic.

\tamari*
\begin{proof}
	Theorems~\ref{thm:parabolic_tamari_cul_extremal} and \ref{thm:nu_tamari_cul_extremal} state that the lattices $\Tamari_{\alpha}$ and $\Tamari_{\nu_{\alpha}}$ are both extremal lattices, and Theorem~\ref{thm:isomorphism_galois_graphs} states that their Galois graphs are isomorphic.  We may thus conclude the result as described in Remark~\ref{rem:isomorphic_graphs_isomorphic_lattices}.
\end{proof}

\begin{corollary}\label{cor:dual_nu_tamari_parabolic}
	For every integer composition $\alpha$, the lattice $\Tamari_{\overline{\nu}_{\alpha}}$ is isomorphic to the dual of $\Tamari_{\alpha}$.
\end{corollary}
\begin{proof}
	This follows from Theorems~\ref{thm:nu_tamari_duality} and \ref{thm:parabolic_nu_tamari_isomorphism}.
\end{proof}

\begin{remark}
	We may as well prove Theorem~\ref{thm:parabolic_tamari_duality} by combining Theorems~\ref{thm:parabolic_nu_tamari_isomorphism} and \ref{thm:nu_tamari_duality}.
\end{remark}

	In fact, we suspect that the map $\Theta$ is an isomorphism from $\Tamari_{\overline{\nu}_{\alpha}}$ to $\Tamari_{\alpha}$, but we have not succeeded proving this directly. See also Remark~\ref{rem:galois_difficulty}. Figure~\ref{fig:cover_map} shows a cover relation in $\Tamari_{\nu_{(3,4,2,1,3,1)}}$ and the corresponding cover relation in $\Tamari_{(1,3,1,2,4,3)}$ under the map $\Theta$.

\begin{openproblem}\label{rem:tamari_direct_isomorphism}
	Show that $\Theta$ is a lattice isomorphism from $\Tamari_{\overline{\nu}_{\alpha}}$ to $\Tamari_{\alpha}$.
\end{openproblem}


\begin{figure}
	\centering
	\begin{tikzpicture}\small
		\def\dx{7};
		\def\dy{1.75};
		\draw(1*\dx,1.3*\dy) node{
			\permA{1}
		};
		\draw(1*\dx,2*\dy) node{
			\ncA{1}
		};
		\draw(1*\dx,3.85*\dy) node{
			\pathA{1}{0}
		};

		\draw(1.5*\dx,1.3*\dy) node{$\gtrdot_{L}$};
		\draw(1.5*\dx,3.75*\dy) node{$\lessdot_{\overline{\nu}_{\alpha}}$};

		\draw(2*\dx,1.3*\dy) node{
			\permB{1}
		};
		\draw(2*\dx,2*\dy) node{
			\ncB{1}
		};
		\draw(2*\dx,3.85*\dy) node{
			\pathB{1}{0}
		};
		
		\draw(1*\dx,2.5*\dy) node{$\downarrow\Theta_{1}$};
		\draw(2*\dx,2.5*\dy) node{$\downarrow\Theta_{1}$};
		\draw(1*\dx,1.6*\dy) node{$\downarrow\Theta_{2}$};
		\draw(2*\dx,1.6*\dy) node{$\downarrow\Theta_{2}$};
	\end{tikzpicture}
	\caption{A cover relation in $\Tamari_{\alpha}$ and the corresponding cover relation in $\Tamari_{\overline{\nu}_{\alpha}}$ for $\alpha=(1,3,1,2,4,3)$.}
	\label{fig:cover_map}
\end{figure}
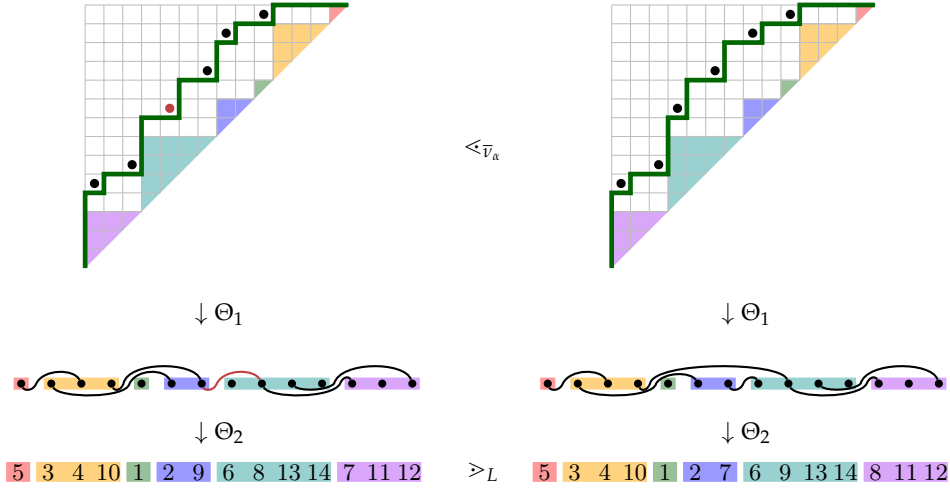

\part{The Zeta Map in Parabolic Cataland} \label{part:zeta}

In the last part of this article, we prove the Steep-Bounce Conjecture of Bergeron, Ceballos and Pilaud; see \cite{ceballos_hopf_2018}*{Conjecture~2.2.8}.  This conjecture arises as a combinatorial approach to understand an intriguing connection between the study of certain lattice walks in the positive quarter plane, and certain Hopf algebras with applications in the theory of multivariate diagonal harmonics.

We prove this conjecture by relating two families of nested Dyck paths via the left-aligned colorable trees ($\alpha$-trees) from Section~\ref{sec:lac_trees}. Since a tree $T$ can be an $\alpha$-tree for potentially several choices of $\alpha$, it is convenient to consider the following definition.

\begin{definition} \label{def:LAC-tree}
	A \tdef{LAC tree} of size $n$ is a pair $(T, \alpha)$, where $T$ is an $\alpha$-tree for a composition $\alpha$ of $n$ ($\alpha\models n$). For $n > 0$, we denote by 
\begin{displaymath}
	\LAC_{n} \defs \bigl\{(T,\alpha)\mid\alpha\models n,T\in\Trees_{\alpha}\bigr\}
\end{displaymath}
the set of all LAC trees of size $n$.
\end{definition}

In contrast to the bijections defined in Section~\ref{sec:parabolic_cataland_bijections}, where a composition $\alpha$ is fixed, the bijections we define in this section will consider all compositions of $n$ at the same time.

\section{Level-Marked Dyck Paths}

In view of Definition~\ref{def:parabolic_dyck_path}, a \tdef{Dyck path} is simply a $(1,1,\ldots,1)$-Dyck path.  Let $\Dyck_{n}$ denote the set of all Dyck paths with $2n$ steps.  

\begin{definition}\label{def:level_marked_path}
	A \tdef{marked Dyck path} $\mu^\bullet$ is a Dyck path with two types of north-steps: \tdef{marked} ones denoted by $N_{\bullet}$, and \tdef{unmarked} ones denoted by $N_{\circ}$. It is \tdef{level-marked} if, for each lattice point $\vec{p}$ on $\mu$, the number of unmarked north-steps before $\vec{p}$ does not exceed the number of east-steps before $\vec{p}$. 
\end{definition}

\begin{remark}
As an equivalent definition, $\mu^\bullet$ is level-marked if and only if for each lattice point $(p, q)$ on $\mu^\bullet$, there are at least $q-p$ marked north-steps before it.
\end{remark}

The set of all level-marked Dyck paths is denoted by $\Dyck_{n}^{\bullet}$.  Level-marked Dyck paths have been introduced in \cite[Section~2.2.3]{ceballos_hopf_2018} under the name ``colored Dyck paths'', and it was shown in that paper that $\Dyck_{n}^{\bullet}$ is in bijection with the set of lattice walks in $\mathbb{N}\times\mathbb{N}$ with $2n$ steps taken from the set $\bigl\{(0,1),(-1,1),(1,-1)\bigr\}$ that start at the origin and end on the $x$-axis.  The bijection goes by sending $N_{\bullet}$ to $(0,1)$, $N_{\circ}$ to $(-1,1)$ and $E$ to $(1,-1)$.  An example of a level-marked Dyck path is shown in the middle of Figure~\ref{fig:bij_tree_dyck}.

\begin{figure}
  \centering
  \includegraphics[page=2,width=\textwidth]{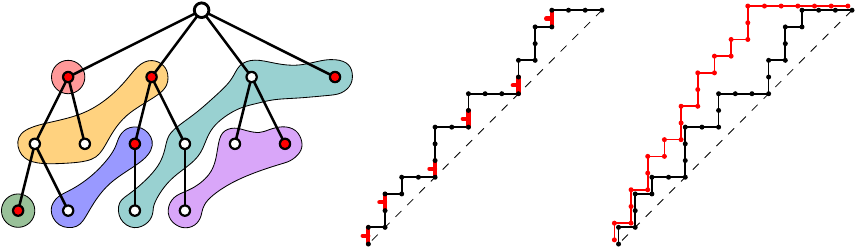}
  \caption{Example of left-aligned colorable trees, level-marked Dyck paths and steep pairs which are in bijection.}
  \label{fig:bij_tree_dyck}
\end{figure}

We define the \tdef{right-to-left traversal} of a plane rooted tree $T$ to be the depth-first search in $T$ starting from the root, where children of the same node are visisted \emph{from right to left}. 

The following construction describes how to obtain a marked Dyck path from an $\alpha$-tree with $n=\lvert\alpha\rvert$.

\begin{construction}\label{constr:lac_to_level}
	Given an integer composition $\alpha$, let $T\in\Trees_{\alpha}$ and $C$ its corresponding left-aligned coloring. We construct a marked Dyck path $\mu_{T}^{\bullet}$, starting from the origin, by performing a right-to-left traversal of $T$.  Whenever we visit a new node in $T$ with a $C$-color we have not seen before, we add a marked north-step $N_\bullet$, otherwise we add an unmarked north-step $N_\circ$.  When we visit a node that we have visited before, we add an east-step.  
\end{construction}

The left part of Figure~\ref{fig:bij_tree_dyck} illustrates Construction~\ref{constr:lac_to_level}.  In the $\alpha$-tree on the left of that figure, the first vertices per color in the right-to-left traversal are marked.  

We now prove that the following map is well defined with the correct domain and image:
\begin{equation}\label{eq:lac_to_dyck}
	\lactodyck\colon\LAC_{n}\to\Dyck_{n}^{\bullet},\quad (T,\alpha)\mapsto\mu_{T}^{\bullet}.
\end{equation}

\begin{proposition}\label{prop:lac_to_level_valid}
	For a LAC tree $(T,\alpha)$, the marked Dyck path $\mu_{T}^{\bullet}$ obtained by Construction~\ref{constr:lac_to_level} is indeed a level-marked Dyck path with $2\lvert\alpha\rvert$ steps.
\end{proposition}
\begin{proof}

Since Construction~\ref{constr:lac_to_level} is simply a variant of the classical bijection between rooted plane trees and Dyck paths, it is clear that $\mu_T^\bullet$ is a Dyck path. To see that it is level-marked, we only need to consider the endpoint $(p,q)$ of any north step in $\mu_T^\bullet$ which corresponds to the first visit of a node $u$ in $T$. By Lemma~\ref{lem:lac_property}(i), all nodes on the path from $u$ to the root, including $u$ and excluding the root, are of different colors. There are exactly $q-p$ such nodes, meaning that up to $(p,q)$ we have seen at least $q-p$ colors in the traversal of Construction~\ref{constr:lac_to_level}. Hence, at least $q-p$ marked north steps come before the point $(p,q)$, making $\mu_T^\bullet$ a level-marked Dyck path.
\end{proof}

We now explain how to obtain a colored tree from a level-marked Dyck path.

\begin{construction}\label{constr:level_to_lac}
	For $n>0$ and $\mu^{\bullet}\in\Dyck_{n}^{\bullet}$, suppose that $\mu^{\bullet}$ has precisely $r$ marked north-steps.  We initialize our construction with the tree $T_{0}$ that consists of a single root node, the empty coloring $f_{0}$, and a list of colors $L_{0} = [0]$, where $0$ stands for the color of the root.  Throughout this construction we maintain a pointer on the colors.  Initially, this pointer points to the $0$ in $L_{0}$.  Moreover, since we traverse the tree during its construction, we say that the \tdef{current node} is the node that we are visiting in the current step, which is the root at the beginning.
	
	We parse $\mu^{\bullet}$ as a word.  Let $x$ be the $i$-th letter of this word.  We then update $T_{i-1}$, $f_{i-1}$, $L_{i-1}$ as follows.
	
	\begin{enumerate}[(i)]
		\item 
			If $x=N_{\bullet}$, then we extend $T_{i-1}$ to $T_{i}$ by adding a node $v_{i}$ to the current node of $T_{i-1}$ as its leftmost child. 
			We set the current node of $T_{i}$ to be $v_{i}$.  We insert a new color $c$ into $L_{i-1}$ right after the pointer, and move the pointer to $c$.  We extend $f_{i-1}$ to $f_{i}$ by setting $f_{i}(v_{i})=c$.
		\item
			If $x=N_{\circ}$, then we extend $T_{i-1}$ to $T_{i}$ by adding a child $v_{i}$ to the current node of $T_{i-1}$ as its leftmost child.  We set the current node of $T_{i}$ to be $v_{i}$.  We take $L_{i}$ to simply be $L_{i-1}$ with the pointer moved to the next color $c'$.  We extend $f_{i-1}$ to $f_{i}$ by setting $f_{i}(v_{i})=c'$.
		\item
			If $x=E$, then we take $T_{i}$ to simply be $T_{i-1}$ with the current node set to the parent of the current node of $T_{i-1}$.  We take $L_{i}$ to be $L_{i-1}$ with the pointer moved to the previous color. Since $\mu^\bullet$ is a Dyck path, such a color (may be $0$ for the root) must exist.  We set $f_{i}=f_{i-1}$.
	\end{enumerate}
	
	Let $f'_{i}$ be the coloring of $T_{i}$ that is obtained from $f_{i}$ by renaming the colors according to first appearance in the left-to-right traversal of $T_{i}$. After $2n$ steps we have parsed $\mu^{\bullet}$ completely, and we output the colored tree $\mathcal{T}_{\mu^{\bullet}}=(T_{2n},f'_{2n})$.
\end{construction}

If $\mathcal{T}_{\mu^{\bullet}}=(T_{\mu^{\bullet}},f')$ is the output of Construction~\ref{constr:level_to_lac}, then we set $\alpha_{i}$ to be the number of nodes of $T_{\mu^{\bullet}}$ whose $f'$-color is $i$, and we set $\alpha_{\mu_{\bullet}}=(\alpha_{1},\alpha_{2},\ldots,\alpha_{r})$ for $r$ the number of colors used by $f'$. 
We now prove that the following map is well defined:
\begin{equation}\label{eq:dyck_to_lac}
	\dycktolac\colon\Dyck_{n}^{\bullet}\to\LAC_{n},\quad\mu^{\bullet}\mapsto (T_{\mu^{\bullet}},\alpha_{\mu^{\bullet}}).
\end{equation}

\begin{lemma}\label{lem:increasing_color_property}
	At the end of the $i$-th step of Construction~\ref{constr:level_to_lac}, the list $L_{i}$ of colors (renamed by first appearances in left-to-right traversal) is increasing.  In particular, if $u$ is a non-root node of $T_{i}$ and $v$ is a descendant of $u$, then $f'_{i}(u)<f'_{i}(v)$.
\end{lemma}
\begin{proof}
	It follows from Construction~\ref{constr:level_to_lac} that the colors in the unique shortest path from the root of $T_{i}$ to $v$ exactly coincide with the colors in $L_{i}$ from the beginning to the pointer, and they appear in the same order on that path and in $L_{i}$.
	
	Now, suppose that there are two colors $c_{1},c_{2}$ such that $c_{1}$ is after $c_{2}$ in $L_{i}$.  Let $u_{1}$ be the first node in the LR-prefix order of $T_{i}$ with $f'_{i}(u_{1})=c_{1}$.  There must be an ancestor $u_{2}$ of $u_{1}$ with $f'_{i}(u_{2})=c_{2}$ by the argument in the first paragraph.  
	
	Since $u_{2}$ precedes $u_{1}$ in the LR-prefix order of $T_{i}$, it follows that $c_{2} < c_{1}$ due to the renaming.  This yields $f'_{i}(u_{2})<f'_{i}(u_{1})$ as desired.
\end{proof}

\begin{proposition}\label{prop:level_to_lac_valid}
	For a level-marked Dyck path $\mu^{\bullet}\in\Dyck_{n}^{\bullet}$ with $r$ marked north-steps, let $(T,f)$ be the colored tree obtained from Construction~\ref{constr:level_to_lac}.  The coloring $f$ yields a composition $\alpha$ such that $T$ is compatible with $\alpha$.
\end{proposition}
\begin{proof}
	We suppose that $\mu^{\bullet}$ has $r$ marked north-steps.	Let $\mathcal{T}_{\mu^{\bullet}}=(T,f)$ be the colored tree obtained from $\mu^\bullet$ through Construction~\ref{constr:level_to_lac}.  It follows that $f$ uses exactly $r$ colors, and for $i\in[r]$ we let $\alpha_{i}$ be the number of nodes of $T$ whose $f$-color is $i$.  Since every non-root node of $T$ receives a unique color, it is clear that $(\alpha_{1},\alpha_{2},\ldots,\alpha_{r})$ is a composition of $n$.
	
	For $s\in[r]$, let $T_{s}$ be the restriction of $T$ to the nodes of $f$-color at most $s$.  We prove by induction on $s$ that the induced subtree $T'_{s}$ of $T$ whose nodes are colored at the end of step $s$ in Construction~\ref{constr:lac_tree} is precisely $T_{s}$.
	
	The base case $s=0$ is clear, since 
	both $T_{0}$ and $T'_{0}$ consist of a single root node.  Now suppose that 
	$T_{s}=T'_{s}$.  By Lemma~\ref{lem:lac_property}(i), we see that $T'_{s+1}$ is connected.  By Lemma~\ref{lem:increasing_color_property}, the $f$-color of any node is bigger than those of all its ancestors, so $T_{s+1}$ is also connected.
	
	Now, let $u$ be the last node with $f(u)=s+1$ in the LR-prefix order of $T$. 
	Consider a node $v \notin T'_{s}$ that is a child of some node in $T_{s}=T'_{s}$ while strictly preceding $u$ in LR-prefix order. Suppose that $v$ is created in step $i$ of Construction~\ref{constr:level_to_lac}, with $L_{i}$ the list of colors (after renaming) available just after.  Note that every ancestor of $v$ has $f$-color at most $s$, so that $u$ is in particular not an ancestor of $v$. Therefore, $v$ comes after $u$ in the right-to-left traversal of $T$, which means that $s+1$ is already present in $L_{i}$.  The pointer of $L_{i}$ cannot be after $s+1$, because that would imply the existence of an ancestor of $v$ of $f$-color $s+1$, which is a contradiction.  The pointer of $L_{i}$ cannot be before $s+1$, because otherwise by Lemma~\ref{lem:increasing_color_property} we would have the contradiction $s+1 > f(v) > s$ because $v \notin T_s$.  It follows that $f(v)=s+1$, which holds for all nodes of $T'_{s+1}\setminus T'_{s}$ that come before $u$ in the LR-prefix order.  
	
	Since $u$ is the last node in LR-prefix order that has $f$-color $s+1$, no descendant of $u$ belongs to $T_{s+1}$.  Let $X$ be the set of nodes in $T_{s+1}\setminus T_{s}$.  We have just shown that $X$ contains all the nodes of $T'_{s+1}\setminus T'_{s}$ that precede $u$ in LR-prefix order. Moreover, all nodes in $X$ are active in the $(s+1)$-st step of Construction~\ref{constr:lac_tree}.  Since $\lvert X\rvert=\alpha_{s+1}$, all the vertices in $X$ receive color $s+1$ in Construction~\ref{constr:lac_tree}.  We conclude that $T_{s+1}=T'_{s+1}$.
	
	This completes the induction step, and we see that $T$ is indeed compatible with $\alpha$, and $f$ is indeed the left-aligned coloring of $T$.
\end{proof}

We conclude this subsection with the proof that $\lactodyck$ is a bijection whose inverse is $\dycktolac$.

\begin{theorem}\label{thm:lac_to_level}
	For every positive integer $n$, the map $\lactodyck$ is a bijection whose inverse is $\dycktolac$.
\end{theorem}
\begin{proof}
	In view of Propositions~\ref{prop:lac_to_level_valid} and \ref{prop:level_to_lac_valid} it remains to show that $\lactodyck \circ \dycktolac = \operatorname{id}$ and $\dycktolac \circ \lactodyck = \operatorname{id}$. We observe that the map $\lactodyck$ without markings and colors is a small twist of the well-known bijection from the set of plane rooted trees with $n$ non-root nodes to the set of Dyck paths with $2n$ steps, described for instance in \cite{deutsch99dyck}*{Appendix~E.1}.

	\medskip

	To show that $\lactodyck \circ \dycktolac = \operatorname{id}$, for $\mu^{\bullet}\in\Dyck_{n}^{\bullet}$, let $(T,\alpha)=\dycktolac(\mu^{\bullet})$ and $\hat{\mu}^{\bullet}=\lactodyck\bigl((T,\alpha)\bigr)$. As mentioned before, the underlying Dyck paths $\mu$ and $\hat{\mu}$ are equal, so it remains to show that the marked north-steps are in the same positions.  By construction, the marked north-steps of $\mu^{\bullet}$ determine the right-most nodes per color in $(T,\alpha)$, which in turn determine the marked north-steps in $\hat{\mu}^{\bullet}$ in the same way.  We conclude that $\mu^{\bullet}=\hat{\mu}^{\bullet}$.

	\medskip

	To show that $\dycktolac \circ \lactodyck = \operatorname{id}$, for $(T,\alpha)\in\LAC_{n}$, let $\mu^{\bullet}=\lactodyck\bigl((T,\alpha)\bigr)$ and $(T',\alpha')=\dycktolac(\mu^{\bullet})$.  As before, we conclude that $T=T'$, and it remains to show that $\alpha=\alpha'$.  We observe already that $\alpha$ has the same number of parts as $\alpha'$, as both are equal to the number of marked north steps in $\mu^\bullet$ by construction.

	Let $T_{s},T'_{s}$ denote the subtrees of $T$ and $T'$, respectively, that consist of all nodes of color at most $s$, including the root.  We show that $T_{s}=T'_{s}$ by induction on $s$.  The base case $s=0$ is clear, since both $T_{0}$ and $T'_{0}$ consist of the root only.  Now suppose that $T_{s}=T'_{s}$.  At least one of the active nodes of $T_{s}$ must belong to the set $V$ of last nodes per color with respect to $(T,\alpha)$.  We denote by $u$ the first active node in $V$ in the LR-prefix order of $T$.  It follows that all active nodes of $T_{s}$ up to $u$ in LR-prefix order receive color $s+1$.  The same reasoning works for $T'_{s}$ as $\alpha$ and $\alpha'$ have the same number of parts, and we conclude that $T_{s+1}=T'_{s+1}$.  This completes the induction step, and we have $\alpha=\alpha'$.
\end{proof}

\section{The Steep-Bounce Theorem}
	\label{sec:steep_bounce_theorem}
Recall the definitions of steep and bounce Dyck paths from Section~\ref{sec:parabolic_dycks}.  
A pair $(\mu_{1},\mu_{2})$ of Dyck paths of the same length is \tdef{nested} if $\mu_{1}$ always stays weakly below~$\mu_{2}$.  
We denote by $\SP_{n}$ the set of nested pairs $(\mu_1,\mu_2)$ such that \mbox{$\mu_1,\mu_2\in\Dyck_n$} and the top path $\mu_2$ is steep. We refer to the elements of $\SP_{n}$ as \tdef{steep pairs} for simplicity. 
Similarly, we denote by $\BP_{n}$ the set of nested pairs $(\mu_1,\mu_2)$ such that $\mu_1,\mu_2\in\Dyck_n$ and the bottom path $\mu_1$ is bounce, and refer to its elements as \tdef{bounce pairs}. 

The following construction encodes a level-marked Dyck path as a steep pair.

\begin{construction}[\cite{ceballos_hopf_2018}*{Section~2.2.3}]\label{constr:level_to_steeppair}
	Given $n>0$ and $\mu^{\bullet}\in\Dyck_{n}^{\bullet}$, we first take $\mu_{1}\in\Dyck_{n}$ the path obtained from $\mu^{\bullet}$ by forgetting the marking.  We then construct a Dyck path $\mu_{2}\in\Dyck_{n}$ by forgetting the east-steps, replacing each marked north-step $N_\bullet$ by a north-step $N$, and each unmarked north-step $N_\circ$ by a valley $EN$.  We then append sufficiently many east-steps at the end of the so-created path to reach the coordinate $(n,n)$. We return the pair $(\mu_{1},\mu_{2})$.
\end{construction}

\begin{lemma}\label{lem:level_to_steeppair_valid}
	For a level-marked Dyck path $\mu^{\bullet}$, the pair $(\mu_{1},\mu_{2})$ obtained by Construction~\ref{constr:level_to_steeppair} is a nested pair of Dyck paths, where $\mu_{2}$ is steep.
\end{lemma}
\begin{proof}
	The fact that $\mu_{2}$ is steep and is weakly above $\mu_{1}$ is immediate from the construction.
\end{proof}

It follows from \cite{ceballos_hopf_2018}*{Section~2.2.3} that the map 
\begin{equation}\label{eq:dyck_to_steeppair}
	\dycktosteeppair\colon\Dyck_{n}^{\bullet}\to\SP_{n},\quad\mu^{\bullet}\mapsto(\mu_{1},\mu_{2})
\end{equation}
is a bijection.  
Let 
\begin{equation}\label{eq:lac_to_steeppair}
	\lactosteep\colon\LAC_{n}\to\SP_{n},\quad (T,\alpha)\mapsto\dycktosteeppair\circ\lactodyck\bigl((T,\alpha)\bigr)
\end{equation}
be the composition of the maps $\lactodyck$ and $\dycktosteeppair$.  Theorem~\ref{thm:lac_to_level} states that $\lactodyck$ is a bijection, which implies that $\lactosteep$ is also a bijection. 
This map is illustrated in the right part of Figure~\ref{fig:bij_tree_dyck}, and also in the right part of Figure~\ref{fig:steep_bounce_zeta_map}. We denote by $\steeptolac$ the inverse of $\lactosteep$.

Recall that for a fixed composition $\alpha$, the map $\lactonn\colon\Trees_{\alpha}\to\Dyck_{\alpha}$ from \eqref{eq:lac_to_nn} is a bijection between the set of $\alpha$-trees and the set of $\alpha$-Dyck paths. This map can be naturally extended to a map on the set of LAC trees considering all compositions simultaneously as follows: 
\begin{equation}\label{eq:lac_to_bouncepair}
	\lactobounce\colon\LAC_{n}\to\BP_{n},\quad (T,\alpha)\mapsto\bigl(\nu_{\alpha},\lactonn(T)\bigr).
\end{equation}
\begin{remark}
Although the maps $\lactonn$ and $\lactobounce$ are essentially the same, we assign them different names because they have different domains and codomains. 
On one hand, $\lactonn$ is defined on the set of plane rooted trees that are compatible with a fixed $\alpha$ ($\alpha$-trees), and its image is the set of Dyck paths that lie weakly above a fixed bounce path $\nu_\alpha$ ($\alpha$-Dyck paths). 
On the other hand, $\lactobounce$ is defined on the set of pairs $(T,\alpha)$ where $T$ is an $\alpha$-tree and $\alpha$ is a composition of $n$ (LAC trees), and its image is the set of bounce pairs of the form $(\nu_\alpha,\mu)$.  
\end{remark}
Theorem~\ref{thm:lac_to_nn} states that $\lactonn$ is a bijection, which implies that $\lactobounce$ is also a bijection.  Let $\bouncetolac$ be the inverse of $\lactobounce$.  The left part of Figure~\ref{fig:steep_bounce_zeta_map} illustrates the map $\lactobounce$. We now define a bijection $\steeppairtobouncepair$ by
\begin{equation} \label{eq:gamma_def}
\steeppairtobouncepair\defs\lactobounce\circ\steeptolac.
\end{equation}
Note that $\steeppairtobouncepair$ sends steep pairs to bounce pairs.

\begin{figure}
	\centering
	\includegraphics[page=11, scale=0.8]{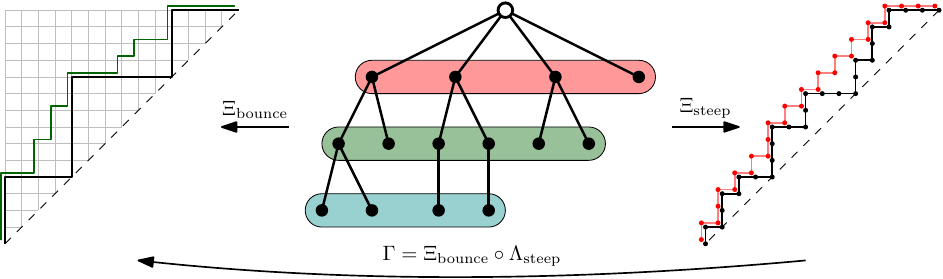}
	\caption{An illustration of the Steep-Bounce zeta map $\Gamma$.}
	\label{fig:steep_bounce_zeta_map}
\end{figure}

We are now in the position to prove the first main result of this section. 
\steepbounce*
\begin{proof}
	Let $(\mu_{1},\mu_{2})\in\SP_{n}$ such that $\mu_{2}$ ends with $r$ consecutive east-steps, and $(\mu'_{1},\mu'_{2})=\steeppairtobouncepair(\mu_{1},\mu_{2})$. 	
	It is clear that $\mu'_{1}$ touches the main diagonal precisely $r+1$ times, since the parameter $r$ describes precisely the number of colors in the LAC tree $(T,\alpha)$ that relates both pairs of Dyck paths.
\end{proof}

We obtain a proof of \cite{ceballos_hopf_2018}*{Conjecture~2.2.1} as a corollary.  We refer the reader to \cite{ceballos_hopf_2018} for any undefined terminology.

\begin{corollary}[{\cite[Conjecture~2.2.1]{ceballos_hopf_2018}}]
	The graded dimension of degree $n$ of the Hopf algebra of the special family of pipe dreams considered in~\cite[Section~2.2]{ceballos_hopf_2018} is equal to the number of walks in the quarter plane starting from the origin, ending on the $x$-axis, and consisting of $2n$ steps taken from the set $\bigl\{(-1, 1), (1, -1), (0, 1)\bigr\}$.
\end{corollary}

An illustration of the bijection $\steeppairtobouncepair$ can be found in Figure~\ref{fig:steep_bounce_zeta_map}; the correspondence between all 
LAC trees, 
bounce pairs and 
steep pairs
for $n=3$ is shown in Figure~\ref{fig:example_zeta}.  

\begin{figure}
	\includegraphics[width=0.33\textwidth,page=12]{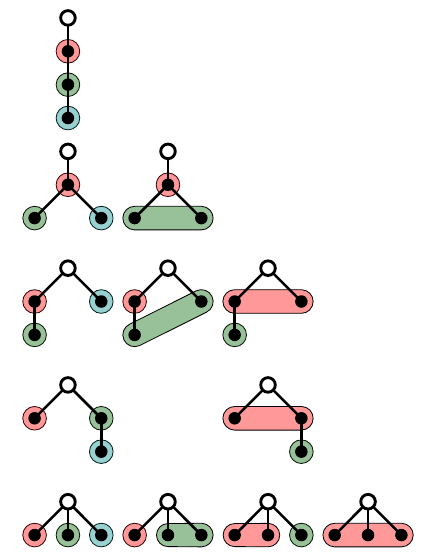}%
	\includegraphics[width=0.33\textwidth,page=19]{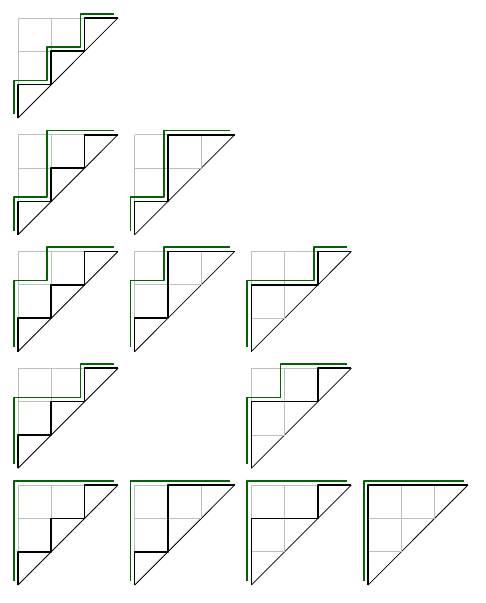}%
	\includegraphics[width=0.33\textwidth,page=13]{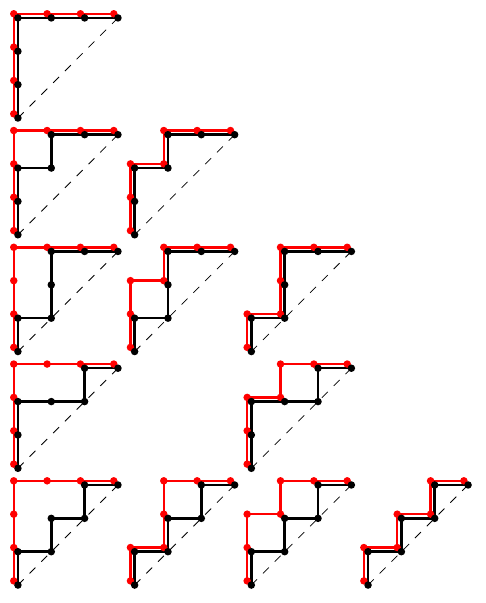}%
	\caption{The LAC trees, bounce pairs, and steep pairs of size $n=3$.}
	\label{fig:example_zeta}
\end{figure}

\section{A New Interpretation of the Zeta Map}
	\label{sec:new_zeta_map}
Somewhat surprisingly, the map~$\steeppairtobouncepair$ turns out to generalize the classical zeta map in $q,t$-Catalan combinatorics.  Before specifying this connection, we briefly recall the history of the zeta map.  Readers familiar with this topic may skip the following section and proceed directly to Section~\ref{sec:steep_bounce_zeta_map}.

\subsection{The Zeta Map}
	\label{sec:classical_zeta_map}
The classical zeta map is a certain bijection from $\Dyck_{n}$ onto itself.  Its inverse was discovered in \cite{AndrewsKrattenthalerOrsinaPapi} in connection with nilpotent ideals in certain Borel subalgebras of $sl(n)$.  It later became an important map in the theory of diagonal harmonics and $q,t$-Catalan combinatorics~\cite{Haglund-qt-catalan}.  

One remarkable ingredient in this theory is a bivariate polynomial $C_{n}(q,t)$ which can be defined as the bigraded Hilbert series of the alternating component of a certain module of diagonal harmonics.  Surprisingly, the dimension of that space is equal to the \tdef{Catalan number} $c_{n}\defs\frac{1}{n+1}\binom{2n}{n}$, and therefore $C_{n}(q,t)$ is a polynomial with nonnegative integer coefficients whose evaluation at $q=t=1$ recovers $c_{n}$.
	
Intensive attempts to understand these polynomials led to important develomments~\cites{GarsiaHaiman-remarkableCatalanSequence,Haiman-vanishingTheorems,haglund_conjectured_2003,garsia_proof_2002}.  In particular, Haglund gave a combinatorial interpretation of the polynomial $C_{n}(q,t)$ in terms of a pair of statistics on Dyck paths known as $\emph{area}$ and $\emph{bounce}$~\cite{haglund_conjectured_2003}.  Shortly thereafter, Haiman announced another combinatorial interpretation in terms of two statistics $\emph{area}$ and $\emph{dinv}$.  Unexpectedly, these two pairs of statistics were different but gave rise to the same expression:
\begin{equation} \label{eq:zeta}
	C_n(q,t) = \sum_{\mu\in\Dyck_{n}} q^{\area(\mu)} t^{\bounce(\mu)} = \sum_{\mu\in\Dyck_{n}} q^{\dinv(\mu)} t^{\area(\mu)}.
\end{equation}
The \tdef{zeta map} $\zeta$ is a bijection on $\Dyck_{n}$ that explains this phenomenon by sending the pair $(\area,\dinv)$ to the pair $(\bounce,\area)$. 

\subsection{Some Statistics on Dyck Paths}
	\label{sec:dyck_statistics}
Let us have a closer look at these statistics.  First we need to describe how to obtain a steep and a bounce Dyck path from a given $\mu\in\Dyck_{n}$.  

\begin{construction}\label{constr:dyck_to_bounce}
	Given $n>0$ and $\mu\in\Dyck_{n}$, we construct a Dyck path $\mu_{\bounce}$ from $\mu$ as follows.  We start at the origin and move north until we reach the last point that is still weakly below $\mu$, and then move east until we hit the main diagonal.  This process is repeated until we reach coordinate $(n,n)$.  It is clear that $\mu_{\bounce}\in\Dyck_{n}$ is bounce, and we call it the \tdef{bounce path} of $\mu$.  
\end{construction}

The \tdef{bounce parameters} of $\mu$ are the $x$-coordinates $0=b_{0}<b_{1}<\cdots<b_{r}=n$ of the contact points of $\mu_{\bounce}$ and the main diagonal.

\begin{construction}\label{constr:dyck_to_steep}
	Given $n>0$ and $\mu\in\Dyck_{n}$, we construct a Dyck path $\mu_{\mathrm{steep}}$ from $\mu$ as follows.  We start at the origin and move north until we reach the first point at which we can add an east-step while still remaining weakly above $\mu$.  We append this east-step, and repeat this process until we reach a lattice point with $y$-coordinate $n$.  We add sufficiently many east-steps such that we end at $(n,n)$.  
	
	It is clear that $\mu_{\mathrm{steep}}\in\Dyck_{n}$ is steep, and we call it the \tdef{steep path} of $\mu$.
\end{construction}

Now for $\mu\in\Dyck_{n}$, its \tdef{area vector} $\mathbf{a}(\mu)\defs(a_{1},a_{2},\ldots,a_{n})$ is the integer vector whose $i$-th entry $a_{i}$ equals the number of unit boxes in row $i$ that are below $\mu$ and above the main diagonal.  

\begin{definition}\label{def:dyck_statistics}
	Given $\mu\in\Dyck_{n}$ with area vector $\mathbf{a}(\mu)=(a_{1},a_{2},\ldots,a_{n})$ and bounce parameters $b_{0},b_{1},\ldots,b_{r}$, we define \tdef{area}, \tdef{dinv}, and \tdef{bounce} of $\mu$, respectively, by
	\begin{align*}
		\area(\mu) & \defs a_{1}+a_{2}+\cdots+a_{n},\\
		\dinv(\mu) & \defs \Bigl\lvert\bigl\{(i,j)\mid 1\leq i<j\leq n,\;a_{i}\in\{a_{j},a_{j}+1\}\bigr\}\Bigr\rvert,\\
		\bounce(\mu) & \defs \sum_{i=1}^{r-1}{n-b_{i}}.
	\end{align*}
\end{definition}

\begin{figure}
	\begin{center}
		\includegraphics[width=\textwidth,page=16]{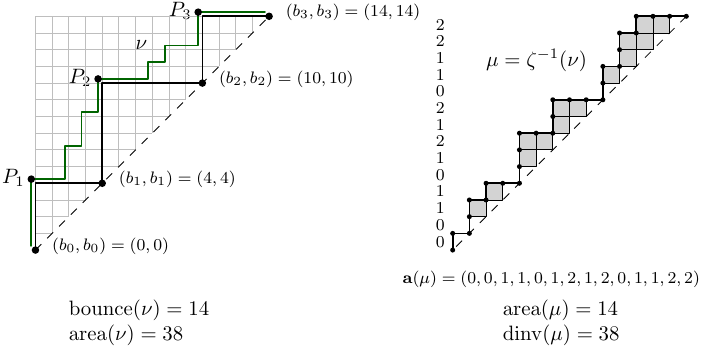}
	\end{center}
	\caption{An example of the inverse of the zeta map.}
	\label{fig:zetamap_statistics}
\end{figure}

Figure~\ref{fig:zetamap_statistics} illustrates these statistics.  We end this subsection by explicitly describing the inverse of the zeta map.  We refer the reader to \cite{AndrewsKrattenthalerOrsinaPapi} and \cite[Proof of Theorem~3.15]{Haglund-qt-catalan} for the original sources, and to \cites{ceballos010,ALW-sweep,williams_sweeping_2018} for further information on generalizations. 

\begin{construction}\label{constr:zeta_inverse}
	Given $\nu\in\Dyck_{n}$, let $\nu_{\bounce}$ be its bounce path with bounce parameters $b_{0},b_{1},\ldots,b_{r}$. For $k\in[r]$, we denote by $\vec{p}_{k}=(b_{k-1},b_{k})$ the peak of $\nu_{\bounce}$ at height $b_{k}$.  By construction, $\vec{p}_{k}$ is also on $\nu$. 
	Clearly, the number of steps between $\vec{p}_{k}$ and $\vec{p}_{k+1}$ is $(b_{k}-b_{k-1})+(b_{k+1}-b_{k})=b_{k+1}-b_{k-1}$. For $k\in[r-1]$, the sequence $\mathbf{a}^{(k)}$ arises from the subpath of $\nu$ between $\vec{p}_{k}$ and $\vec{p}_{k+1}$ by replacing each east-step by $k-1$ and each north-step by $k$.  
	
	We fashion these $r-1$ sequences together by interlacing the sequences $\mathbf{a}^{(k)}$ and $\mathbf{a}^{(k+1)}$ for all $k\in[r-1]$, where we insert the values $k-1$ and $k+1$ relative to the values $k$, and where we never put $k-1$ directly before $k+1$. The resulting sequence $\mathbf{a}$ has precisely $n$ entries, $b_{k+1}-b_{k}$ of which are equal to $k$ for $k\in\{0,1,\ldots,r-1\}$.  
\end{construction}

This construction is illustrated in Figure~\ref{fig:zeta_inverse_description}.

\begin{figure}
	\begin{center}
		\includegraphics[width=\textwidth,page=17]{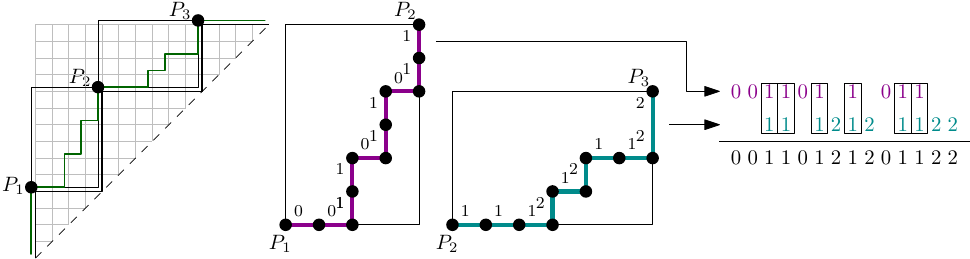}
	\end{center}
	\caption{An illustration of Construction~\ref{constr:zeta_inverse}.}
	\label{fig:zeta_inverse_description}
\end{figure}

\begin{proposition}[\cite{Haglund-qt-catalan}*{Proof of Theorem~3.15}]\label{prop:zeta_inverse_valid}
	For $\nu\in\Dyck_{n}$, the sequence $\mathbf{a}$ obtained by Construction~\ref{constr:zeta_inverse} is the area-vector of some Dyck path $\mu\in\Dyck_{n}$, and we have $\mu=\zeta^{-1}(\nu)$.
\end{proposition}

\subsection{The Steep-Bounce Zeta Map}
	\label{sec:steep_bounce_zeta_map}
In this section we prove Theorem~\ref{thm:steep_bounce_zeta_map}, relating the zeta map $\zeta$ to our map $\steeppairtobouncepair$.

Let $T$ be a plane rooted tree with $n$ non-root nodes.  A node of $T$ is of \tdef{level} $k$ if its graph distance from the root is $k+1$. 
In other words, the level of a node is its depth \emph{minus one}. 
Let $\mathbf{a}(T)=(a_{1},a_{2},\ldots,a_{n})$ be the vector obtained by reading levels of each node according to the right-to-left traversal of $T$.  We define an integer partition $\alpha_{T}=(\alpha_{1},\alpha_{2},\ldots)$ of $n$ by setting $\alpha_{i}$ to the number of nodes of level $i+1$.  It then follows immediately that $T$ is $\alpha_{T}$-compatible.  The pair $(T,\alpha_{T})$ is a \tdef{horizontal LAC tree}.  We have already seen a horizontal LAC tree in the middle of Figure~\ref{fig:steep_bounce_zeta_map}.

\begin{lemma}\label{lem:zetaproof1}
	For $n>0$, the map $\lactosteep$ restricts to a bijection from
	\begin{itemize}
		\item the set of horizontal LAC trees $(T,\alpha_{T})$ where $T$ has $n$ non-root nodes, to
		\item the set of pairs $(\mu,\mu_{\mathrm{steep}})$ where $\mu\in\Dyck_{n}$.
	\end{itemize}
	Moreover, if $(\mu,\mu_{\mathrm{steep}})=\lactosteep\bigl((T,\alpha_{T})\bigr)$, then $\mathbf{a}(\mu)=\mathbf{a}(T)$. 
\end{lemma}
\begin{proof}
	Recall that $\lactosteep = \dycktosteeppair \circ \lactodyck$.  Given a horizontal LAC tree $(T,\alpha)$, let $\mu_\bullet = \lactodyck\bigl((T,\alpha)\bigr)$ be the corresponding level-marked Dyck path, and $(\mu_1,\mu_2)=\dycktosteeppair (\mu_\bullet)$ be the steep pair associated with $\mu_\bullet$. We denote by $\mathbf{a}(\mu_1)=(a_1,a_2,\ldots,a_n)$ the area vector of $\mu_{1}$.  We recall that $\mu_1$ is the Dyck path obtained from $\mu_\bullet$ by forgetting the markings. The proof proceeds in three steps.

	\medskip
	
	We first prove that $\mathbf{a}(\mu_1)=\mathbf{a}(T)$.  Recall that $\mu_1$ is obtained from the right-to-left traversal of $T$: upon each visit to a new node, we add a north-step; every time we go back up one level, we add an east-step.  The contribution of a north-step to the area vector is equal to the number of north-steps preceding it minus the number of east-steps preceding it; this is clearly equal to the level of the corresponding node of the tree.

	\medskip
	
	We now prove that $(T,\alpha)$ is a horizontal LAC tree if and only if $\mu_\bullet$ is such that the $j$-th north-step is marked if and only if $a_i<a_j$ for all $i<j$.
	
	Observe that for every $\mu\in\Dyck_{n}$ there is a unique way to mark north-steps to satisfy the desired property.  Moreover, every plane rooted tree $T$ corresponds to a unique horizontal LAC-tree.  Since the set of Dyck paths with $2n$ steps is in bijection with the the set of plane rooted trees with $n$ non-root nodes~\cite{deutsch99dyck}*{Appendix~E.1}, and the map $\lactodyck$ is a bijection, it remains to show that $\lactodyck$ assigns to every horizontal LAC tree a marked Dyck path with the desired property.
	
	Let $(T,\alpha)$ be the horizontal LAC tree derived from $T$.  The color of a node in $T$ is equal to its level. In the right-to-left traversal of the tree, each time we reach the first node of a certain level, we get a marked north-step $N_\bullet$ in $\mu_\bullet$. All other nodes correspond to regular north-steps $N_\circ$. Since a north-step contributes an entry to the area vector which is equal to the level of its corresponding node, the marked north-steps in $\mu_\bullet$ are exactly those that contribute an entry $a_j$ such that $a_i<a_j$ for all $i<j$. 

	\medskip

	We finally prove that $\mu_{2}$ is the steep path of $\mu_{1}$ if and only if $\mu_\bullet$ is such that the $j$-th north-step is marked if and only if $a_i<a_j$ for all $i<j$.
	 
	This is straightforward in view of the following alternative description of the steep path of $\mu_{1}$. We parse $\mu_{1}$ as follows. In the beginning, for each north-step we write down $NE$, until we reach one that contributes a $1$ to the area vector of $\mu_{1}$,  for which we write down $NN$.  We then continue to append copies of $NE$ until we reach a north-step that contributes a $2$ to the area vector of $\mu_{1}$, for which we append $NN$ again.  We continue this process until a total of $n$ letters $N$ are written down, and we finish by adding the necessary number of letters $E$ at the end.
	 
	The word so written corresponds to the steep path of $\mu_{1}$, and the valleys of $(\mu_{1})_{\mathrm{steep}}$ correspond precisely to the entries $a_{j}$ of the area vector of $\mu_{1}$ with $a_{i}<a_{j}$ for all $i<j$.  By Construction~\ref{constr:level_to_steeppair}, the valleys of $\mu_{2}$ correspond precisely to the marked north-steps of $\mu_{\bullet}$. 
	
	\medskip
	
	These three claims together show that $(T,\alpha)$ is a horizontal LAC tree if and only if $\lactosteep\bigl((T,\alpha)\bigr)$ is $\bigl(\mu_{1},(\mu_{1})_{\mathrm{steep}}\bigr)$, and that $\mathbf{a}(T)=\mathbf{a}(\mu_{1})$.
\end{proof}

\begin{lemma}\label{lem:zetaproof2}
	For $n>0$, the map $\bouncetolac$ restricts to a bijection from
	\begin{itemize}
		\item the set of pairs $(\nu_{\bounce},\nu)$ where $\nu\in\Dyck_{n}$, to
		\item the set of horizontal LAC trees $(T,\alpha_{T})$ where $T$ has $n$ non-root nodes.
	\end{itemize}	
	Moreover, if $(T,\alpha_{T})=\bouncetolac(\nu_{\bounce},\nu)$, then $\mathbf{a}\bigl(\zeta^{-1}(\nu)\bigr)=\mathbf{a}(T)$. 
\end{lemma}
\begin{proof}
	Given a nested pair of Dyck paths $(\nu_\alpha,\nu)$ with $\alpha=(\alpha_1,\dots,\alpha_r)$ a composition of $n$, let $(T,\alpha)= \bouncetolac(\nu_\alpha,\nu)$. For $k\in [r]$, recall that $s_k=\alpha_1+\alpha_2 + \cdots +\alpha_k$ and that $a_k$ is the number of active nodes at the beginning of the $(k+1)$-st step of Construction~\ref{constr:lac_tree} (not to be confused with components $\alpha_k$ in $\alpha$).  By definition, $s_0=0$, and $a_0$ is the number of children of the root. Since $(T,\alpha)$ is a horizontal LAC tree, we have $a_k=\alpha_{k+1}$.  The proof proceeds in two steps.

	\medskip
	
	We first show that $(T,\alpha)$ is a horizontal LAC tree if and only if $\nu_\alpha=\nu_{\bounce}$.  
	
	We have already argued that the set of horizontal LAC trees with $n$ non-root nodes is in bijection with the set of Dyck paths with $2n$ steps, and therefore also with the set of nested pairs of the form $(\nu_{\bounce},\nu)$ with $\nu\in\Dyck_n$.  Thus, both sets have $c_n$ elements.  In order to prove the claim, it thus suffices to show that $(T,\alpha)$ being a horizontal LAC tree implies $\nu_\alpha=\nu_{\bounce}$. 
	
	We already know that $\nu_\alpha$ lies weakly below $\nu$, and we need to show that each peak of $\nu_\alpha$ touches the path $\nu$ at the initial point of an east-step of $\nu$. Equivalently, we need to show that, for each $k\in [r-1]$, the path $\nu$ has a valley with coordinates $(p,q)$ such that $s_{k-1}<p\leq s_k$ and $q=s_k$. This valley corresponds to the rightmost internal node of color $k$, which is the $i$-th node of color $k$ for some $1\leq i \leq \alpha_k$. Since $T$ is a horizontal LAC tree, the rightmost child of this node is the rightmost node of color $k+1$, which is the $\alpha_{k+1}$-st active node in the $k$-th step of Construction~\ref{constr:lac_tree}. By definition of $\lactobounce$, we see that 
$s_{k-1}<p=s_k-i+1\leq s_k$ and $q=s_k+a_k-\alpha_{k+1}=s_k$, as desired.

	\medskip
	
	We now show that if $(T,\alpha)$ is a horizontal LAC tree, then $\mathbf{a}\bigl(\zeta^{-1}(\nu)\bigr)=\mathbf{a}(T)$.
	
	Let $\mu=\zeta^{-1}(\nu)$ as described in Construction~\ref{constr:zeta_inverse}.  The path $\nu_{\bounce}$ touches the diagonal at $(0,0)=(s_0,s_0),(s_1,s_1),\ldots, (s_r,s_r)=(n,n)$. Therefore,
$\mathbf{a}(\mu)$ has exactly $s_{k+1}-s_k=\alpha_{k+1}$ entries equal to $k$ for $0\leq k \leq r-1$, which is exactly the number of entries equal to $k$ in $\mathbf{a}(T)$.

	For $1\leq k\leq r-1$, we denote by ${\mathbf{a}}_k(\mu)$ and ${\mathbf{a}}_k(T)$, respectively, the restrictions of $\mathbf{a}(\mu)$ and $\mathbf{a}(T)$ to the values in $\{k-1,k\}$.  It remains to show that ${\mathbf{a}}_k(\mu)={\mathbf{a}}_k(T)$.

	The vector ${\mathbf{a}}_k(T)$ can be obtained from the right-to-left traversal of~$T$: each time we reach a new node at level $k-1$ we add a value $k-1$ and each time we reach a new node at level $k$ we add a value $k$. 
	
	The vector ${\mathbf{a}}_k(\mu)={\mathbf{a}}_k\bigl(\zeta^{-1}(\nu)\bigr)$ can be obtained from the subpath of $\nu$ between $(s_{k-1},s_k)$ and $(s_k,s_{k+1})$ by replacing each east-step with $k-1$ and each north-step by $k$.  

	Recall that $T=\nntolac(\nu)$ as described in Construction~\ref{constr:path_to_lac}.   By Lemma~\ref{lem:nb_children}, for $1\leq i \leq s_k-s_{k-1}=\alpha_k$, the number of north-steps in $\nu$ with $x$-coordinate $s_{k-1}+i$ is equal to the number of children of the $i$-th node of color $k$ in the right-to-left traversal of $T$.  From this we deduce that ${\mathbf{a}}_k(T)={\mathbf{a}}_k(\mu)$. 
\end{proof}

\steepbouncezeta*
\begin{proof}
	By Lemmas~\ref{lem:zetaproof1} and \ref{lem:zetaproof2}, the map 
	\begin{displaymath}
		\steeppairtobouncepair^{-1} = (\lactobounce \circ \steeptolac)^{-1} = \lactosteep \circ \bouncetolac
	\end{displaymath}
	restricts to a bijection from 
	\begin{itemize}
		\item the set of pairs $(\nu_{\bounce},\nu)$ where $\nu\in\Dyck_n$, to
		\item the set of pairs $(\mu,\mu_{\mathrm{steep}})$ where $\mu\in\Dyck_n$.
	\end{itemize}	
	Moreover, if $(\mu,\mu_{\mathrm{steep}})=\steeppairtobouncepair^{-1}(\nu_{\bounce},\nu)$ then $\mathbf{a}(\mu)=\mathbf{a}\bigl(\zeta^{-1}(\nu)\bigr)$, which implies that $\mu=\zeta^{-1}(\nu)$. 
\end{proof}

\begin{remark} \label{rem:comb_zeta}
	We can also see Equation~\eqref{eq:zeta} combinatorially through our bijections. Let $(T,\alpha)$ be a horizontal LAC tree, with $(\mu,\mu_{\mathrm{steep}})=\lactosteep\bigl((T,\alpha)\bigr)$, and $(\nu_\alpha,\nu) = \lactobounce\bigl((T,\alpha)\bigr)$. Since $\mu$ is simply $\lactodyck\bigl((T,\alpha)\bigr)$ without markings, we know that each north-step in $\mu$ corresponds to a non-root node in $T$, with the corresponding entry in the area vector $\mathbf{a}(\mu)$ to be the level of that node.  
	
	Therefore, we have (we sum over all non-root nodes)
	\begin{align*}
		\area(\mu) = \sum_{u \in T} \mathrm{level}(u) &= \sum_{u \in T} \sum_{v\;\mathrm{above}\;u} 1 \\
			&= \sum_{v \in T} \sum_{u\;\mathrm{below}\;v} 1 = \sum_{i=1}^{r-1} (n-b_i) = \bounce(\nu).
	\end{align*}

	The pairs contributing to $\dinv(\mu)$ correspond to pairs of nodes $(u,v)$ such that $v$ comes before $u$ in the right-to-left traversal of $T$ and $\mathrm{level}(v)-\mathrm{level}(u)\in\{0,1\}$.  Now, if we order the nodes in $T$
first by increasing level, and then by occurrence in the right-to-left traversal, then each node can be identified with an integer between $1$ and $n$.  It follows that the cells below $\nu$ correspond precisely to the pairs contributing to $\dinv(\mu)$, and we see that $\area(\nu) = \dinv(\mu)$.
\end{remark}

\subsection{The Zeta Map on Parking Functions}
There is another generalization of the zeta map, due to Haglund and Loehr~\cite{haglund_conjectured_2005}, which we can obtain as a special case of a labeled version of our bijection~$\steeppairtobouncepair$. Before going into the details, let us introduce some definitions. 

A \tdef{vertical labeling} of a Dyck path $\mu\in \Dyck_n$ is an assignment of the numbers from $1$ through~$n$ to the boxes directly to the right of the north steps of $\mu$ such that the numbers are increasing from below in each column. The set $\vertical(\Dyck_n)$ of pairs $(\mu,P)$ where $P$ is a vertical labeling of $\mu\in \Dyck_n$ is commonly referred to as the set of \tdef{parking functions} of length $n$.

A \tdef{diagonal labeling} of a Dyck path $\nu\in \Dyck_n$ is an assignment of the numbers from $1$ through~$n$ to the boxes on the main diagonal such that, for every consecutive pair of steps $EN$ of $\nu$ (valleys of $\nu$), the number in the same column as the $E$ step is smaller than the number in the same row as the $N$ step. We denote by $\diagonal(\Dyck_n)$ the set of pairs $(\nu,Q)$ where $Q$ is a diagonal labeling of $\nu\in\Dyck_n$.

In~\cite{haglund_conjectured_2005}~\cite[Conjecture~5.2]{Haglund-qt-catalan}, Haglund and Loehr conjectured two combinatorial interpretations of the bigraded Hilbert series of the module of diagonal harmonics:
\begin{displaymath}
	\mathcal{DH}_n(q,t)=
	\sum_{(\mu,P)\in\vertical(\Dyck_n)} q^{\dinv'(\mu,P)} t^{\area(\mu)} =
	\sum_{(\nu,Q)\in\diagonal(\Dyck_n)} q^{\area'(\nu,Q)} t^{\bounce(\nu)},
\end{displaymath}
where $\dinv'$ and $\area'$ are generalizations of the $\dinv$ and $\area$ statistics that depend on the labelings $P$ and $Q$. See~\cite[Chapter~5]{Haglund-qt-catalan} for more details.  
Haglund and Loehr also provided a bijection 
\begin{displaymath}
	\widetilde \zeta \colon \vertical(\Dyck_n) \to \diagonal(\Dyck_n)
\end{displaymath}
that explains the equivalence of their two interpretations, and their conjecture follows from the more general statement of the shuffle conjecture~\cite{haglund_shuffle_2005}, which was proven by Carlsson and Mellit in~\cite{carlsson_proof_2018}. 

The map $\widetilde \zeta$ is a straightforward generalization of the zeta map $\zeta$. It is defined by
\begin{displaymath}
	\widetilde \zeta (\mu,P) \defs \bigl(\zeta(\mu),Q\bigr),
\end{displaymath}
where the order in which the labels in $Q$ appear, in the north-east direction, is obtained by reading the labels in $P$ along diagonals $y=x+i$, in the north-east direction, for $i\in\{0,1,\dots, n-1\}$~\cite[Chapter~5]{Haglund-qt-catalan}.\footnote{We are using the version of the $\widetilde \zeta$ map described by Haglund in~\cite[Chapter~5]{Haglund-qt-catalan}, which differs from Haglund and Loehr's original description by just reversing the path $\zeta(\mu)$ and the labels on the diagonal.}  
An example is illustrated in Figure~\ref{fig_zetamap_on_parkingfunctions}.

\begin{figure}
\begin{center}
\includegraphics[width=\textwidth,page=18]{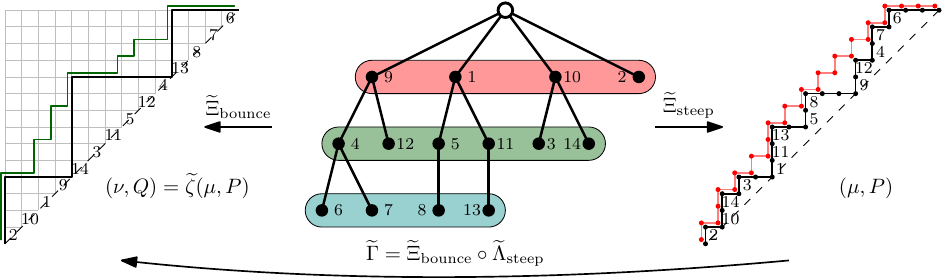}
\caption{An example of the $\widetilde \zeta$ map from parking functions to diagonally labeled Dyck paths, computed in terms of the map $\steeppairtobouncepairlabel$.}
\label{fig_zetamap_on_parkingfunctions}
\end{center}
\end{figure}

We will show that $\widetilde \zeta$ can be obtained in terms of a labeled version of our map $\steeppairtobouncepair$. 
Let $T$ be a plane rooted tree with $n$ non-root nodes.  A \tdef{right increasing labeling} $R$ of $T$ is an assignment of the numbers from $1$ though $n$ to the non-root nodes of $T$ that increases along right-most children, \ie if $v$ is the right-most child of $u$, then $R(v) > R(u)$.
We define a \tdef{labeled horizontal LAC tree} as a triple $(T,\alpha_{T},R)$ where $(T,\alpha_{T})$ is a horizontal LAC tree and $R$ is a right increasing labeling of $T$. 
We denote by $\LAClabel_n$ the set of labeled horizontal LAC trees with $n$ non-root nodes.

The restriction of the map $\lactosteep$ from horizontal LAC trees to pairs of the form $(\mu,\mu_\steep)$ (as in Lemma~\ref{lem:zetaproof1}) can be extended naturally in the context of labeled horizontal LAC trees. 
We denote by $\SPlabel$ the set of triples $(\mu,\mu_\steep,P)$ where $\mu\in \Dyck_n$ and $P$ is a vertical labeling of~$\mu$. Let
\begin{equation}\label{eq:lac_to_steeppair_label}
	\lactosteeplabel\colon\LAClabel_{n}\to\SPlabel 
\end{equation}
be the map that sends a labeled horizontal LAC tree $(T,\alpha_{T},R)$ to $(\mu,\mu_\steep,P)$, where $(\mu,\mu_\steep)=\lactosteep\bigl((T,\alpha_{T})\bigr)$ and $P$ is obtained from $R$ as follows: the non-root nodes in $T$ are in correspondence with the north steps of $\mu$ through the bijection $\lactosteep$, and we assign the labels of the non-root nodes in $R$ to their corresponding north steps. Since the labels in $T$ increase along right-most children, the resulting labeling of the north steps in $\mu$ is a vertical labeling. And vice-versa, every vertical labeling of $\mu$ arises uniquely from a labeling of $T$ that increases along right-most children. As a consequence, the map $\lactosteeplabel$ is a bijection. We denote by~$\steeptolaclabel$ its inverse.    

Similarly, we also have a labeled version of the restriction of the map $\lactobounce$ from horizontal LAC trees to pairs of the form $(\nu_\bounce,\nu)$ (as in Lemma~\ref{lem:zetaproof2}).
We denote by $\BPlabel$ the set of triples $(\nu_\bounce,\nu,Q)$ where $\nu\in \Dyck_n$ and $Q$ is a diagonal labeling of~$\nu$. Let
\begin{equation}\label{eq:lac_to_bouncepair_label}
	\lactobouncelabel\colon\LAClabel_{n}\to\BPlabel
\end{equation}
be the map that sends a labeled horizontal LAC tree $(T,\alpha_{T},R)$ to $(\nu_\bounce,\nu,Q)$, where $(\nu_\bounce,\nu)=\lactobounce\bigl((T,\alpha_{T})\bigr)$ and $Q$ is obtained from $R$ as follows: the labels in $Q$, from bottom to top, are obtained by reading from right-to-left the labels of non-root nodes of color $i$ in $T$, for $i\in\{1,2,\ldots\}$.
{Since the valleys of $\nu$ are in correspondence with pairs of nodes $(u,v)$ in $T$ such that $v$ is the right-most child of~$u$, it is not hard to see that $Q$ is a diagonal labeling of $\nu$ and that the map~$\lactobouncelabel$ is a bijection}. 
We denote by $\bouncetolaclabel$ its inverse.

Finally, we denote by 
\begin{equation}\label{eq:steeppair_to_bouncepair_label}
	\steeppairtobouncepairlabel \colon \SPlabel\to\BPlabel
\end{equation}
the composition $\steeppairtobouncepairlabel=\lactobouncelabel \circ \steeptolaclabel$. 
As a consequence of Theorem~\ref{thm:steep_bounce_zeta_map} and the discussion above, we get the following corollary.

\begin{corollary}\label{cor_zeta_parkingfunctions}
	For $n>0$, the map $\steeppairtobouncepairlabel$ is a bijection from
	\begin{itemize}
		\item the set of triples $(\mu,\mu_{\steep},P)$ such that $\mu\in\Dyck_n$ and $P$ is a vertical labeling of $\mu$, to
		\item the set of triples $(\nu_{\bounce},\nu,Q)$ such that $\nu\in\Dyck_n$ and $Q$ is a diagonal labeling of~$\nu$. 
	\end{itemize}
	Moreover, if $(\nu_{\bounce},\nu,Q)=\steeppairtobouncepairlabel(\mu,\mu_{\steep},P)$, then $(\nu,R)=\widetilde\zeta(\mu,P)$, where $\widetilde \zeta$ is the Haglund-Loehr's zeta map from parking functions to diagonally labeled Dyck paths.
\end{corollary}

\begin{remark}
Corollary~\ref{cor_zeta_parkingfunctions} is essentially a labeled version of Theorem~\ref{thm:steep_bounce_zeta_map}. One might expect to have a generalization of Theorem~\ref{thm:steep_bounce_theorem} involving labeled versions of steep pairs and bounce pairs. 
Indeed, one can naturally define a vertically labeled steep pair as a triple $(\mu_1,\mu_2,P)$  where $(\mu_1,\mu_2)$ is a nested pair such that the top path $\mu_2$ is steep and $P$ is a vertical labeling of~$\mu_1$. 
However, for the definition of diagonally labeled bounce pairs $(\nu_1,\nu_2,Q)$, we need $(\nu_1,\nu_2)$ to be a nested pair such that the bottom path $\nu_1$ is bounce as expected, but the requirement on the labeling $Q$ depends on both $\nu_1$ and $\nu_2$ and is more involved.\end{remark}

\begin{remark}
Suppose that we remove the increasing condition in vertical labelings and the valley condition in diagonal labelings, but instead record as statistics the number of violations of these conditions, \ie the number of descents from below to above along columns in vertical labelings, and the number of valleys with the entry below larger than the one above in diagonal labelings. Then our bijection $\widetilde{\Gamma}$ generalizes naturally to these labelings while transferring the two statistics from one to the other. It is because they are relayed by labeled LAC trees without the right increasing condition, and both statistics correspond to the number of non-root nodes whose label is larger than that of its rightmost child. The zeta map for parking functions is simply the case where all these statistics are zero. 
\end{remark}

\section*{Discussion}

In this section we highlight some possible directions of future research.

\medskip
\noindent
{\bf Enumerative questions.} 
From an enumerative point of view, Mishna and Rechnitzer found in~\cite{lattice-walk} that the generating function of the lattice walks mentioned after Definition~\ref{def:level_marked_path}, which are in bijection with the level-marked Dyck paths, is not differentially algebraic with respect to the variable marking the end of the walk. This is an indication that the enumeration of these objects is complicated. It would be interesting to study the enumeration of certain sub-families with simple enumeration formulas.

\medskip
\noindent
{\bf Generalizations for other Coxeter groups.} 
An intriguing follow-up question is to exhibit combinatorial models for the algebraic generalizations of parabolic $231$-avoiding permutations, parabolic nonnesting partitions, and parabolic noncrossing partitions to other Coxeter groups in \cite{muehle18tamari}*{Section~6}.  Observe that these generalizations may depend on the choice of Coxeter element.  These combinatorial models may then be used to generalize the other families described here accordingly. 
If such an extension is possible, it will hopefully lead to a Steep-Bounce zeta map in other types, which would possibly be related to the zeta map for irreducible crystallographic root systems described by Thiel in~\cite{thiel-other-types}.

\medskip
\noindent
{\bf Fu{\ss}-Catalan and rational Catalan generalizations.} Another natural extension would be to consider Fu{\ss}-Catalan and rational Catalan objects. Is it possible to define an analogue of LAC trees in this world? If such an adaptation exists, we may define a Steep-Bounce map on Fu{\ss}-Catalan or even rational Catalan objects, therefore potentially generalize the zeta map in rational Catalan combinatorics.

\medskip
\noindent
{\bf Generalizations of the zeta map and statistics for multivariate diagonal harmonics.} 
One of the most remarkable properties of the zeta map is that it describes combinatorial statistics for the Hilbert series of certain modules in diagonal harmonics. As mentioned by Armstrong, Loehr and Warrington in~\cite{ALW-sweep}, one point of view is that, rather than having two statistics $\area$ and $\dinv$, we have only one statistic area and a nice map $\zeta$. There are generalizations for multisets of variables that give rise to certain modules of multivariate diagonal harmonics~\cite{bergeron_multivariate_2013}. One important and difficult problem in this area is to find combinatorial models and a multiset of combinatorial statistics to describe them.  
In the three variate case, these modules are closely related to intervals in the Tamari lattice. Two of the statistics have been conjectured in~\cite{BergeronPrevilleRatelle} but the third statistic is still missing. 
Furthermore, at the moment no generalization of the zeta map for Tamari intervals is known. 
Can the results in this paper be used to find the third missing statistic or a zeta map for intervals in the Tamari lattice? Can we push this forward for more sets of variables in connection to the results in~\cite{ceballos_hopf_2018}? A positive answer to these questions may lead to new understanding of multivariate diagonal harmonics.

\section*{Acknowledgements}

The authors thank Nantel Bergeron, Myrto Kallipoliti, Marni Mishna, Vincent Pilaud, and Robin Sulzgruber for inspiring discussions.

\appendix

\section{Lattice-Theoretic Background}
	\label{sec:lattice_theory}
In this appendix we provide some lattice-theoretic background that is important for the proof of Theorem~\ref{thm:parabolic_nu_tamari_isomorphism}.  

\subsection{Basics}
	\label{sec:poset_basics} 
Let $\Poset=(P,\leq)$ be a finite partially ordered set (or \tdef{poset} for short).  An element $x\in P$ is \tdef{minimal} if for every $y\in P$ with $y\leq x$ follows $y=x$.  Dually, $x$ is \tdef{maximal} if for every $y\in P$ with $x\leq y$ follows $x=y$.  Then, $\Poset$ is \tdef{bounded} if there exists a unique minimal and a unique maximal element, and we denote these bounds by $\least$ and $\grtst$, respectively.

Two elements $x,y\in P$ form a \tdef{cover relation} if $x<y$ and there exists no $z\in P$ with $x<z<y$.  We write $x\lessdot y$ in this case, and say that $x$ is a \tdef{lower cover} of $y$ and that $y$ is a \tdef{upper cover} of $x$.  

A \tdef{chain} of $\Poset$ is a totally ordered subset of $P$.  A chain is \tdef{saturated} if it can be written as a sequence of cover relations, and it is \tdef{maximal} if it is saturated and contains a minimal and a maximal element of $\Poset$.  The \tdef{length} of $\Poset$ is one less than the maximum cardinality of the maximal chains of $\Poset$, and it is denoted by $\ell(\Poset)$.

The poset $\Poset$ is a \tdef{lattice} if for every $x,y\in P$ there exists a least upper bound $x\vee y$ (the \tdef{join} of $x$ and $y$), and a greatest lower bound $x\wedge y$ (the \tdef{meet} of $x$ and $y$).  It follows that every finite lattice is bounded.

If $\Poset$ is a lattice, then an element $j\in P\setminus\{\least\}$ is \tdef{join irreducible} if for every $x,y\in P$ with $j=x\vee y$ we have $j\in\{x,y\}$.  Dually, $m\in P\setminus\{\grtst\}$ is \tdef{meet irreducible} if for every $x,y\in P$ with $m=x\wedge y$ we have $m\in\{x,y\}$.  We denote the set of all join-irreducible elements of $\Poset$ by $\JI(\Poset)$, and the set of all meet-irreducible elements of $\Poset$ by $\MI(\Poset)$.  Moreover, we observe that if $P$ is finite, then every $j\in\JI(\Poset)$ has a unique lower cover $j_{*}$, and every $m\in\MI(\Poset)$ has a unique upper cover $m^{*}$.

\subsection{Extremal Lattices}
	\label{sec:extremal_lattices}
Let $\Lattice=(L,\leq)$ be a finite lattice.  We follow \cite{markowsky92primes} and say that $\Lattice$ is \tdef{extremal} if $\bigl\lvert\JI(\Lattice)\bigr\rvert=\ell(\Lattice)=\bigl\lvert\MI(\Lattice)\bigr\rvert$.  Extremal lattices admit a nice presentation as lattices that arise from certain directed graphs~\cite{markowsky92primes}, see also \cite{thomas19rowmotion}.  

Suppose that $\Lattice$ is extremal with $\ell(\Lattice)=n$, and we choose a maximal chain $C:x_{0}\lessdot x_{1}\lessdot\cdots\lessdot x_{n}$.  We may totally order the join- and meet-irreducible elements of $\Lattice$ as $j_{1},j_{2},\ldots,j_{n}$ and $m_{1},m_{2},\ldots,m_{n}$ such that
\begin{displaymath}
	j_{1}\vee j_{2}\vee\cdots\vee j_{i} = x_{i} = m_{i+1}\wedge m_{i+2}\wedge\cdots\wedge m_{n}
\end{displaymath}
holds for all $i\in[n]$.

From this order we may define a directed graph with vertex set $[n]$, where we have a directed edge $i\to k$ if and only if $i\neq k$ and $j_{i}\not\leq m_{k}$.  This graph is the \tdef{Galois graph} of $\Lattice$, denoted by $\Galois(\Lattice)$.  Figure~\ref{fig:extremal_lattice} shows an extremal lattice, and Figure~\ref{fig:galois_graph} shows its Galois graph.  The join- and meet-irreducible elements of the lattice in Figure~\ref{fig:extremal_lattice} are labeled as described in the previous paragraph.  We observe that this lattice has a unique maximal chain of maximum length.

\begin{figure}
	\centering
	\begin{subfigure}[t]{.25\textwidth}
		\centering
		\begin{tikzpicture}\small
			\def\x{1};
			\def\y{1.5};
			\tikzstyle{poset}=[draw,circle,scale=.6]
			\draw(2*\x,1*\y) node[poset](n1){};
			\draw(3*\x,1.5*\y) node[poset,label=below:{\color{white!50!black}\scriptsize $1$}](n2){};
			\draw(1*\x,2*\y) node[poset,label=below:{\color{white!50!black}\scriptsize $3$},label=above:{\color{white!50!black}\scriptsize $1$}](n3){};
			\draw(2.5*\x,2*\y) node[poset,label=below:{\color{white!50!black}\scriptsize $2$},label=above:{\color{white!50!black}\scriptsize $3$}](n4){};
			\draw(4*\x,2*\y) node[poset,label=below:{\color{white!50!black}\scriptsize $4$},label=above:{\color{white!50!black}\scriptsize $2$}](n5){};
			\draw(2*\x,2.5*\y) node[poset,label=above:{\color{white!50!black}\scriptsize $4$}](n6){};
			\draw(3*\x,3*\y) node[poset](n7){};
			\draw(n1) -- (n2);
			\draw(n1) -- (n3);
			\draw(n2) -- (n4);
			\draw(n2) -- (n5);
			\draw(n3) -- (n6);
			\draw(n4) -- (n6);
			\draw(n5) -- (n7);
			\draw(n6) -- (n7);
		\end{tikzpicture}
		\caption{An extremal lattice.}
		\label{fig:extremal_lattice}
	\end{subfigure}
	\hspace*{.2cm}
	\begin{subfigure}[t]{.3\textwidth}
		\centering
		\begin{tikzpicture}\small
			\def\x{1};
			\def\y{1};
			\draw(2*\x,1*\y) node(n1){$4$};
			\draw(1*\x,2*\y) node(n2){$3$};
			\draw(3*\x,2*\y) node(n3){$1$};
			\draw(2*\x,3*\y) node(n4){$2$};
			\draw[->](n1) -- (n2);
			\draw[->](n1) -- (n3);
			\draw[->](n2) -- (n4);
			\draw[->](n4) -- (n3);
		\end{tikzpicture}
		\caption{The Galois graph of the lattice in Figure~\ref{fig:extremal_lattice}.}
		\label{fig:galois_graph}
	\end{subfigure}
	\hspace*{.2cm}
		\begin{subfigure}[t]{.35\textwidth}
		\centering
		\begin{tikzpicture}\small
			\def\x{1};
			\def\y{1.5};
			\draw(2*\x,1*\y) node(n1){\scriptsize $(-,1234)$};
			\draw(3*\x,1.5*\y) node(n2){\scriptsize $(1,234)$};
			\draw(1*\x,2*\y) node(n3){\scriptsize $(3,14)$};
			\draw(2.5*\x,2*\y) node(n4){\scriptsize $(12,34)$};
			\draw(4*\x,2*\y) node(n5){\scriptsize $(14,2)$};
			\draw(2*\x,2.5*\y) node(n6){\scriptsize $(123,4)$};
			\draw(3*\x,3*\y) node(n7){\scriptsize $(1234,-)$};
			\draw(n1) -- (n2);
			\draw(n1) -- (n3);
			\draw(n2) -- (n4);
			\draw(n2) -- (n5);
			\draw(n3) -- (n6);
			\draw(n4) -- (n6);
			\draw(n5) -- (n7);
			\draw(n6) -- (n7);
		\end{tikzpicture}
		\caption{The lattice of maximal orthogonal pairs of the graph in Figure~\ref{fig:galois_graph}.}
		\label{fig:lattice_maximal_orthogonal_pairs}
	\end{subfigure}
	\caption{An extremal lattice with its associated Galois graph.}
	\label{fig:extremal_galois}
\end{figure}
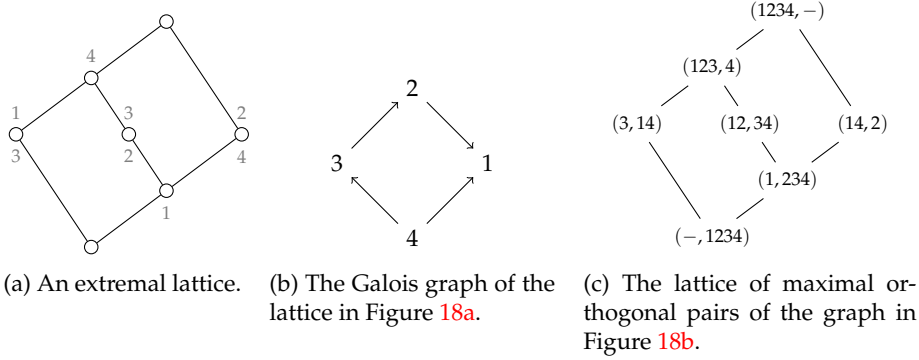

Now let $G$ be a simple directed graph with vertex set $[n]$ and edges $i\to k$ only if $i>k$.  Two sets $X,Y\subseteq[n]$ form an \tdef{orthogonal pair} if $X\cap Y=\emptyset$ and there exists no edge $i\to k$ in $G$ with $i\in X$ and $k\in Y$.  An orthogonal pair $(X,Y)$ is \tdef{maximal} if whenever there exist $Y'\supseteq Y$ and $X'\supseteq X$ such that $(X',Y')$ is an orthogonal pair, then $X'=X$ and $Y'=Y$.  Let $\MO(G)$ denote the set of maximal orthogonal pairs of $G$.  

For $(X,Y),(X',Y')\in\MO(G)$ we declare $(X,Y)\sqsubseteq(X',Y')$ if $X\subseteq X'$ (or equivalently $Y\supseteq Y'$).  It can be verified that $\Lattice(G)\defs\bigl(\MO(G),\sqsubseteq\bigr)$ is a lattice, where the meet $(X,Y)\sqcap(X',Y')$ is the unique maximal orthogonal pair $(X\cap X',Z)$ and the join $(X,Y)\sqcup(X',Y')$ is the unique maximal orthogonal pair $(Z,Y\cap Y')$.  We have the following result.

\begin{theorem}[{\cite[Section~4]{markowsky92primes}}]\label{thm:markowskys_representation}
	If $G$ is a simple directed graph on the vertex set $[n]$ which has edges $i\to k$ only if $i>k$, then $\Lattice(G)$ is an extremal lattice of length $n$.  Conversely, every finite extremal lattice $\Lattice$ is isomorphic to $\Lattice\bigl(\Galois(\Lattice)\bigr)$.  
\end{theorem}

The lattice of maximal orthogonal pairs of the directed graph in Figure~\ref{fig:galois_graph} is shown in Figure~\ref{fig:lattice_maximal_orthogonal_pairs}.  In this figure, we have omitted set parentheses and commas for brevity.

\begin{remark}\label{rem:isomorphic_graphs_isomorphic_lattices}
	Let $\Lattice,\Lattice'$ be two extremal lattices with corresponding Galois graphs $\Galois(\Lattice),\Galois(\Lattice')$.  If $f\colon\Galois(\Lattice)\to\Galois(\Lattice')$ is an isomorphism of directed graphs, then it is clear that $(X,Y)$ is a maximal orthogonal pair of $\Galois(\Lattice)$ if and only if $\bigl(f(X),f(Y)\bigr)$ is a maximal orthogonal pair of $\Galois(\Lattice')$.  It follows that $\Lattice\bigl(\Galois(\Lattice)\bigr)\cong\Lattice\bigl(\Galois(\Lattice')\bigr)$, and in view of Theorem~\ref{thm:markowskys_representation} we conclude that $\Lattice\cong\Lattice'$.
\end{remark}

\begin{remark}\label{rem:galois_difficulty}
	Observe that not every isomorphism of Galois graphs extends directly to an isomorphism of the corresponding extremal lattices.  Consider for instance the pentagon lattice together with the bijection that swaps top and bottom and fixes everything else.  This map is the identity on the Galois graph, but is not a lattice isomorphism.
\end{remark}

\subsection{Congruence-Uniform Lattices}
	\label{sec:congruence_uniform_lattices}
Let us recall the following doubling construction due to Day~\cite{day79characterizations}.  Let $\Poset=(P,\leq)$ be a finite poset.  For $X\subseteq P$ we define 
\begin{displaymath}
	P_{\leq X} \defs \{y\in P\mid y\leq x\;\text{for some}\;x\in X\}.
\end{displaymath}
Let $\mathbf{2}$ denote the unique two-element lattice on the ground set $\{0,1\}$.  The \tdef{doubling} of $\Poset$ by $X$ is the subposet of the direct product $\Poset\times\mathbf{2}$ induced by the set
\begin{displaymath}
	\Bigl(P_{\leq X}\times\{0\}\Bigr)\uplus\Bigl(\bigl((P\setminus P_{\leq X})\cup X\bigr)\times\{1\}\Bigr).
\end{displaymath}
A lattice is \tdef{congruence uniform} if it can be constructed from the singleton lattice by a sequence of doublings by intervals~\cite[Theorem~5.1]{day79characterizations}.  Figure~\ref{fig:doubling_sequence} shows an example of a congruence-uniform lattice and the sequence of doublings from which it is constructed.

\begin{figure}
	\centering
	\begin{tikzpicture}
		\def\d{2};
		\tikzstyle{fposet}=[fill,circle,scale=\s]
		\tikzstyle{poset}=[draw,circle,scale=\s]
		\draw(1*\d,1*\d) node{
			\begin{tikzpicture}\small
				\def\x{.75};
				\def\s{.4};
				\draw(1*\x,5.35*\x) node{};
				\draw(1*\x,1*\x) node[fposet]{};
			\end{tikzpicture}
		};
		\draw(1.25*\d,.25*\d) node{\tiny $\rightarrow$};
		\draw(1.5*\d,1*\d) node{
			\begin{tikzpicture}\small
				\def\x{.75};
				\def\s{.4};
				\draw(1*\x,5.35*\x) node{};
				\draw(1*\x,1*\x) node[poset](n1){};
				\draw(1*\x,2*\x) node[fposet](n2){};
				\draw(n1) -- (n2);
			\end{tikzpicture}
		};
		\draw(1.75*\d,.25*\d) node{\tiny $\rightarrow$};
		\draw(2*\d,1*\d) node{
			\begin{tikzpicture}\small
				\def\x{.75};
				\def\s{.4};
				\draw(1*\x,5.35*\x) node{};
				\draw(1*\x,1*\x) node[poset](n1){};
				\draw(1*\x,2*\x) node[fposet](n2){};
				\draw(1*\x,3*\x) node[fposet](n3){};
				\draw(n1) -- (n2);
				\draw(n2) -- (n3);
			\end{tikzpicture}
		};
		\draw(2.3*\d,.25*\d) node{\tiny $\rightarrow$};
		\draw(2.75*\d,1*\d) node{
			\begin{tikzpicture}\small
				\def\x{.75};
				\def\s{.4};
				\draw(1*\x,5.35*\x) node{};
				\draw(1*\x,1*\x) node[fposet](n1){};
				\draw(1*\x,2*\x) node[fposet](n2){};
				\draw(2*\x,2.5*\x) node[poset](n3){};
				\draw(1*\x,3*\x) node[fposet](n4){};
				\draw(2*\x,3.5*\x) node[poset](n5){};
				\draw(n1) -- (n2);
				\draw(n2) -- (n4);
				\draw(n2) -- (n3);
				\draw(n3) -- (n5);
				\draw(n4) -- (n5);
			\end{tikzpicture}
		};
		\draw(3*\d,.25*\d) node{\tiny $\rightarrow$};
		\draw(3.75*\d,1*\d) node{
			\begin{tikzpicture}\small
				\def\x{.75};
				\def\s{.4};
				\draw(1*\x,5.35*\x) node{};
				\draw(1*\x,1*\x) node[poset](n1){};
				\draw(2*\x,1.5*\x) node[poset](n2){};
				\draw(1*\x,2*\x) node[poset](n3){};
				\draw(2*\x,2.5*\x) node[poset](n4){};
				\draw(1*\x,3*\x) node[fposet](n5){};
				\draw(3*\x,3*\x) node[poset](n6){};
				\draw(2*\x,3.5*\x) node[poset](n7){};
				\draw(3*\x,4*\x) node[poset](n8){};
				\draw(n1) -- (n2);
				\draw(n1) -- (n3);
				\draw(n2) -- (n4);
				\draw(n3) -- (n4);
				\draw(n3) -- (n5);
				\draw(n4) -- (n6);
				\draw(n4) -- (n7);
				\draw(n5) -- (n7);
				\draw(n6) -- (n8);
				\draw(n7) -- (n8);
			\end{tikzpicture}
		};
		\draw(4.25*\d,.25*\d) node{\tiny $\rightarrow$};
		\draw(5*\d,1*\d) node{
			\begin{tikzpicture}\small
				\def\x{.75};
				\def\s{.4};
				\draw(1*\x,5.35*\x) node{};
				\draw(1*\x,1*\x) node[poset](n1){};
				\draw(2*\x,1.5*\x) node[poset](n2){};
				\draw(1*\x,2*\x) node[poset](n3){};
				\draw(2*\x,2.5*\x) node[poset](n4){};
				\draw(1*\x,3*\x) node[fposet](n5){};
				\draw(1*\x,3.5*\x) node[fposet](n6){};
				\draw(3*\x,3*\x) node[poset](n7){};
				\draw(2*\x,4*\x) node[fposet](n8){};
				\draw(3*\x,4.5*\x) node[poset](n9){};
				\draw(n1) -- (n2);
				\draw(n1) -- (n3);
				\draw(n2) -- (n4);
				\draw(n3) -- (n4);
				\draw(n3) -- (n5);
				\draw(n4) -- (n7);
				\draw(n4) -- (n8);
				\draw(n5) -- (n6);
				\draw(n6) -- (n8);
				\draw(n7) -- (n9);
				\draw(n8) -- (n9);
			\end{tikzpicture}
		};
		\draw(5.5*\d,.25*\d) node{\tiny $\rightarrow$};
		\draw(6.25*\d,1*\d) node{
			\begin{tikzpicture}\small
				\def\x{.75};
				\def\s{.4};
				\draw(1*\x,5.35*\x) node{};
				\draw(1*\x,1*\x) node[poset](n1){};
				\draw(2*\x,1.5*\x) node[poset,label=below right:{\color{white!50!black}\scriptsize $d$}](n2){};
				\draw(1*\x,2*\x) node[poset,label=below left:{\color{white!50!black}\scriptsize $a$}](n3){};
				\draw(2*\x,2.5*\x) node[poset](n4){};
				\draw(1*\x,3*\x) node[poset,label=below left:{\color{white!50!black}\scriptsize $b$}](n5){};
				\draw(3*\x,3*\x) node[poset,label=below right:{\color{white!50!black}\scriptsize $c$}](n6){};
				\draw(.5*\x,3.5*\x) node[poset,label=below left:{\color{white!50!black}\scriptsize $f$}](n7){};
				\draw(1*\x,3.5*\x) node[poset,label=below right:{\color{white!50!black}\scriptsize $e$}](n8){};
				\draw(.5*\x,4*\x) node[poset](n9){};
				\draw(2*\x,4*\x) node[poset](n10){};
				\draw(1.5*\x,4.5*\x) node[poset](n11){};
				\draw(3*\x,5.25*\x) node[poset](n12){};
				\draw(n1) -- (n2);
				\draw(n1) -- (n3);
				\draw(n2) -- (n4);
				\draw(n3) -- (n4);
				\draw(n3) -- (n5);
				\draw(n4) -- (n6);
				\draw(n4) -- (n10);
				\draw(n5) -- (n7);
				\draw(n5) -- (n8);
				\draw(n7) -- (n9);
				\draw(n8) -- (n9);
				\draw(n8) -- (n10);
				\draw(n9) -- (n11);
				\draw(n6) -- (n12);
				\draw(n10) -- (n11);
				\draw(n11) -- (n12);
			\end{tikzpicture}
		};
	\end{tikzpicture}
	\caption{A sequence of six interval doublings that yields a congruence-uniform lattice with twelve elements.  At each step we double by the highlighted interval.  Moreover, we have labeled the join-irreducible elements in the rightmost lattice.}
	\label{fig:doubling_sequence}
\end{figure}
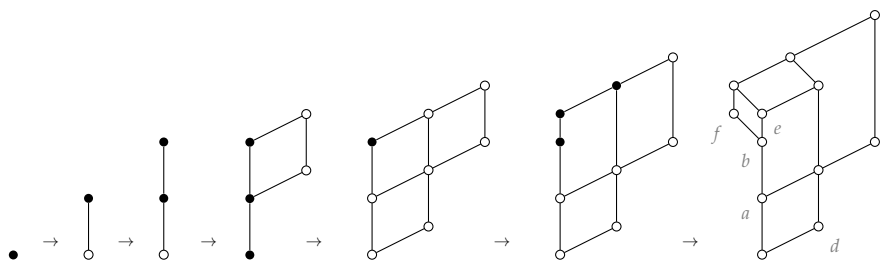

If a congruence-uniform lattice is also extremal, then we can realize the corresponding Galois graph as a directed graph whose vertices \emph{are} the join-irreducible elements.  

\begin{proposition}[{\cite[Corollary~2.17]{muehle18noncrossing}}]\label{prop:extremal_cu_galois}
	Let $\Lattice=(L,\leq)$ be an extremal, congruence-uniform lattice.  The Galois graph of $\Lattice$ is the directed graph whose vertex set is $\JI(\Lattice)$, and which has a directed edge $i\to k$ if and only if $i\neq k$ and $k\leq k_{*}\vee i$.
\end{proposition}

Let us illustrate that the construction in Proposition~\ref{prop:extremal_cu_galois} does not in general yield the Galois graph of an extremal lattice.  Let $\Lattice$ be the extremal lattice from Figure~\ref{fig:extremal_lattice}.  We can check that $\Lattice$ is not congruence uniform.  If we construct a directed graph on the join-irreducible elements of $\Lattice$ according to the rules in Proposition~\ref{prop:extremal_cu_galois}, we get the graph in Figure~\ref{fig:non_extremal_graph}, which is not the same graph as $\Galois(\Lattice)$ shown in Figure~\ref{fig:galois_graph}.  

\begin{figure}
	\centering
	\begin{tikzpicture}\small
		\def\x{1};
		\def\y{1};
		\draw(2*\x,1*\y) node(n1){$4$};
		\draw(1*\x,2*\y) node(n2){$3$};
		\draw(3*\x,2*\y) node(n3){$1$};
		\draw(2*\x,3*\y) node(n4){$2$};
		\draw[->](n1) -- (n3);
		\draw[->](n2) -- (n4);
		\draw[->](n4) -- (n3);
	\end{tikzpicture}
	\caption{The graph of the lattice in Figure~\ref{fig:extremal_lattice} constructed according to the rules in Proposition~\ref{prop:extremal_cu_galois}.}
	\label{fig:non_extremal_graph}
\end{figure}
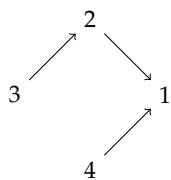

Conversely, the congruence-uniform lattice $\Lattice$ on the right of Figure~\ref{fig:doubling_sequence} is also extremal.  The directed graph $G(\Lattice)$ constructed according to the rules in Proposition~\ref{prop:extremal_cu_galois} is shown in Figure~\ref{fig:extremal_cu_graph}.  Figure~\ref{fig:extremal_cu_mo_lattice} shows the lattice of maximal orthogonal pairs of $G(\Lattice)$, which is isomorphic to $\Lattice$.  This proves that $G(\Lattice)$ is indeed (isomorphic to) the Galois graph of $\Lattice$.  As an example, the map $(a,b,c,d,e,f)\mapsto(1,2,6,4,3,5)$ realizes $G(\Lattice)$ in Figure~\ref{fig:extremal_cu_graph} as a graph that satisfies the conditions of Theorem~\ref{thm:markowskys_representation}.

\begin{figure}
	\centering
	\begin{subfigure}[t]{.35\textwidth}
		\centering
		\begin{tikzpicture}\small
			\def\x{1};
			\draw(2*\x,1.5*\x) node(na){$a$};
			\draw(1*\x,1*\x) node(nb){$b$};
			\draw(3*\x,2*\x) node(nc){$c$};
			\draw(2*\x,2.5*\x) node(nd){$d$};
			\draw(1*\x,2*\x) node(ne){$e$};
			\draw(3*\x,1*\x) node(nf){$f$};
			\draw[->](nb) -- (na);
			\draw[->](nc) -- (na);
			\draw[->](nc) -- (nd);
			\draw[->](nc) -- (ne);
			\draw[->](nc) -- (nf);
			\draw[->](nd) -- (ne);
			\draw[->](ne) -- (nb);
			\draw[->](ne) -- (na);
			\draw[->](nf) -- (nb);
			\draw[->](nf) -- (na);
		\end{tikzpicture}
		\caption{The Galois graph of the lattice on the right of Figure~\ref{fig:doubling_sequence}.}
		\label{fig:extremal_cu_graph}
	\end{subfigure}
	\hspace*{1cm}
	\begin{subfigure}[t]{.55\textwidth}
		\centering
		\begin{tikzpicture}\small
			\def\x{2.25};
			\def\y{1.35};
			\draw(1*\x,1*\y) node(n1){\scriptsize $(-,abcdef)$};
			\draw(2*\x,1.5*\y) node(n2){\scriptsize $(d,abcf)$};
			\draw(1*\x,2*\y) node(n3){\scriptsize $(a,bcdef)$};
			\draw(2*\x,2.5*\y) node(n4){\scriptsize $(ad,bcf)$};
			\draw(1*\x,3*\y) node(n5){\scriptsize $(ab,cdef)$};
			\draw(3*\x,3*\y) node(n6){\scriptsize $(acd,b)$};
			\draw(.5*\x,3.5*\y) node(n7){\scriptsize $(abf,cde)$};
			\draw(1*\x,3.5*\y) node(n8){\scriptsize $(abe,cdf)$};
			\draw(.5*\x,4*\y) node(n9){\scriptsize $(abef,cd)$};
			\draw(2*\x,4*\y) node(n10){\scriptsize $(abde,cf)$};
			\draw(1.5*\x,4.5*\y) node(n11){\scriptsize $(abdef,c)$};
			\draw(3*\x,5.25*\y) node(n12){\scriptsize $(abcdef,-)$};
			\draw(n1) -- (n2);
			\draw(n1) -- (n3);
			\draw(n2) -- (n4);
			\draw(n3) -- (n4);
			\draw(n3) -- (n5);
			\draw(n4) -- (n6);
			\draw(n4) -- (n10);
			\draw(n5) -- (n7);
			\draw(n5) -- (n8);
			\draw(n7) -- (n9);
			\draw(n8) -- (n9);
			\draw(n8) -- (n10);
			\draw(n9) -- (n11);
			\draw(n6) -- (n12);
			\draw(n10) -- (n11);
			\draw(n11) -- (n12);
		\end{tikzpicture}
		\caption{The lattice of maximal orthogonal pairs of the graph in Figure~\ref{fig:extremal_cu_graph}.}
		\label{fig:extremal_cu_mo_lattice}
	\end{subfigure}
	\caption{An illustration of Proposition~\ref{prop:extremal_cu_galois}.}
\end{figure}
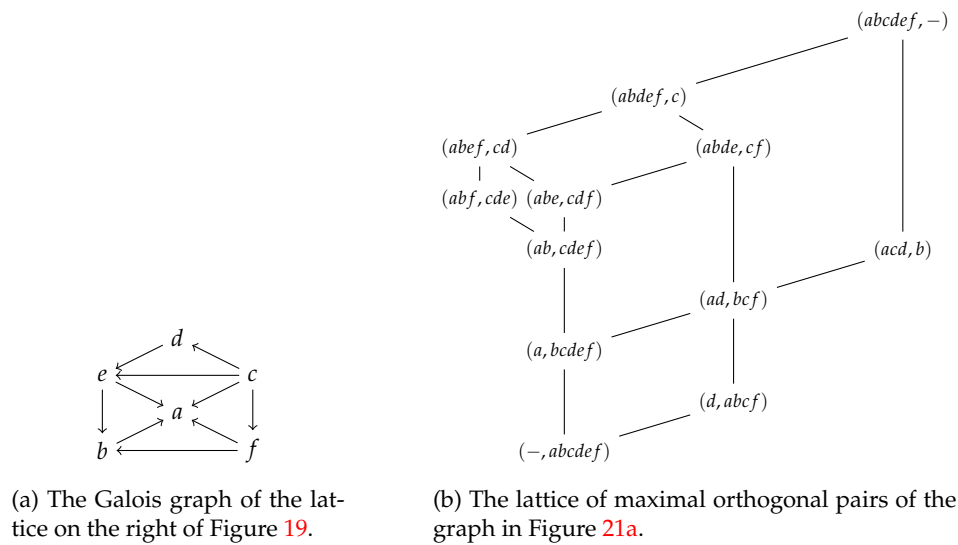

\bibliographystyle{elsarticle-harv}
\bibliography{literature}

\end{document}